\documentclass{elsarticle}
\usepackage{esvect}
\usepackage[utf8]{inputenc}
\usepackage{tikz}
\usepackage{tikz-cd}
\usetikzlibrary{shapes,arrows}
\tikzstyle{block}=[draw opacity=0.7,line width=1.4cm]
\usepackage{amssymb}
\usepackage{amsmath}
\usepackage[hidelinks]{hyperref}
\usepackage{multirow}
\usepackage{amsthm}
\usepackage{amsfonts}
\usepackage[mathscr]{eucal}
\usepackage{enumitem}
\setlist[enumerate]{itemsep=0mm}
\usepackage{graphicx}
\usepackage{hyphenat}
\usepackage{mathtools}
\usepackage{etoolbox}
    \makeatletter
    \def\ps@pprintTitle{%
       \let\@oddhead\@empty
       \let\@evenhead\@empty
       \let\@oddfoot\@empty
       \let\@evenfoot\@oddfoot
    }
    \makeatother

\newcommand{\ie}{i.\,e.}

\newcommand{\eg}{e.\,g.}

\def\Lifts{{\cal L}}
\def\TO{\mathop{\mathrm{TO}}\nolimits}
\def\Forb{\mathop{\mathrm{Forb_{he}}}\nolimits}
\def\Forbi{\mathop{\mathrm{Forb_e}}\nolimits}
\def\dom{\mathop{\mathrm{Dom}}\nolimits}
\def\Forbm{\mathop{\mathrm{Forb_m}}\nolimits}
\def\Forbh{\mathop{\mathrm{Forb_h}}\nolimits}
\def\Age{\mathop{\mathrm{Age}}\nolimits}
\def\Cl{\mathop{\mathrm{Cl}}\nolimits}
\def\ACl{\mathop{\mathrm{Acl}}\nolimits}
\def\Rel{\mathop{\mathrm{Rel}}\nolimits}
\def\Str{\mathop{\mathrm{Str}}\nolimits}
\def\CSP{\mathop{\mathrm{CSP}}\nolimits}
\def\sh{\mathop{\mathrm{Sh}}\nolimits}
\def\ordclass#1{\vv{#1}}
\def\ordclasssup#1#2{\vv*{#1}{#2}}

\def\str#1{\mathbf {#1}}
\def\arity#1{\mathop{\mathrm{arity}}\nolimits(\rel{}{#1})}
\def\farity#1{\mathop{\mathrm{arity}}\nolimits(\func{}{#1})}
\def\nbrel#1#2{R\ifstrempty{#1}{}{_{#1}}\ifstrempty{#2}{}{^{#2}}}
\def\rel#1#2{\nbrel{\ifstrempty{#1}{}{\str{#1}}}{#2}}
\def\func#1#2{\nbfunc{\ifstrempty{#1}{}{\str{#1}}}{#2}}
\def\nbfunc#1#2{F\ifstrempty{#1}{}{_{#1}}\ifstrempty{#2}{}{^{#2}}}
\def\extl#1#2{Q\ifstrempty{#1}{}{_{#1}}\ifstrempty{#2}{}{^{#2}}}
\def\ext#1#2{\extl{\ifstrempty{#1}{}{\str{#1}}}{#2}}
\def\F{{\mathcal F}}
\def\K{{\mathcal K}}
\def\Piece{{\mathfrak P}}
\def\APiece{{\mathfrak A}}
\def\PieceEq{{\mathcal P}}
\def\Incompatible{{\cal I}}
\def\Fraisse{Fra\"{\i}ss\' e}
\theoremstyle{plain}
\newtheorem{theorem}{Theorem}[section]
\newtheorem{corollary}[theorem]{Corollary}
\newtheorem{prop}[theorem]{Proposition}
\newtheorem{observation}[theorem]{Observation}
\newtheorem{lemma}[theorem]{Lemma}

\theoremstyle{definition}
\newtheorem{definition}[theorem]{Definition}
\newtheorem{example}[theorem]{Example}
\newtheorem*{remark*}{Remark}
\newtheorem*{claim*}{Claim}
\newtheorem{claim}[theorem]{Claim}
\theoremstyle{remark}
\newtheorem*{remark}{Remark}

\begin{document}
\begin{frontmatter}

\title{All those Ramsey classes\\
{(Ramsey classes with closures\\ and forbidden homomorphisms)}}

\author[1]{Jan Hubi\v cka\corref{cor1}%
\fnref{fn1,fn2}}
\ead{hubicka@kam.mff.cuni.cz}
\author[2]{Jaroslav Ne\v set\v ril\fnref{fn2}}
\ead{nesetril@iuuk.mff.cuni.cz}
\address[1]{Charles University, Faculty of Mathematics and Physics, Department of Applied Mathematics (KAM), Malostransk\' e n\' am. 25, 118 00 Praha 1, Czech Republic}
\address[2]{Charles University, Faculty of Mathematics and Physics, Computer Science Institute of Charles University (IUUK), Malostransk\' e n\' am. 25, 118 00 Praha 1, Czech Republic}

\cortext[cor1]{Corresponding author}
\fntext[fn1]{Supported  by  project  18-13685Y  of  the  Czech  Science Foundation (GA\v CR).}
\fntext[fn2]{In the final stages both authors were supported by ERC Synergy DYNASNET grant.}


\begin{abstract}
We prove the Ramsey property for classes of ordered structures with closures
and given local properties.  This generalises earlier results: the Ne\v set\v
ril--R\"odl Theorem, the Ramsey property of partial orders and metric spaces as
well as the authors' Ramsey lift of bowtie-free graphs.  We use this framework
to solve several open problems and give new examples of Ramsey classes. Among
others, we find Ramsey lifts of convexly ordered $S$-metric spaces and prove
the Ramsey theorem for finite models (\ie{}\ structures with both functions and
relations) thus providing the ultimate generalisation of the structural Ramsey
theorem.  Both of these results are natural, and easy to state, yet  their
proofs involve most of the theory developed here.

We also characterise Ramsey lifts of classes of structures defined by finitely
many forbidden homomorphisms and extend this to special cases of classes with
closures.  This has numerous applications.  For example, we find Ramsey lifts
of many Cherlin--Shelah--Shi classes.
\end{abstract}
\end{frontmatter}
\eject
\tableofcontents
\clearpage
\section {Introduction}
Extending classical results, structural Ramsey theory emerged at
the beginning of 1970's in a series of papers
\cite{Graham1971,Graham1972,Nevsetvril1976,Halpern1966,Abramson1978,Milliken1979}.
This development is outlined in \cite{Graham1990} and~\cite{Nevsetvril1995}. A proper foundation for the area was given when
the notions of a Ramsey class, the $\str{A}$-Ramsey property and the ordering
property~\cite{Leeb,Nevsetvril1976} were isolated.  However, the list of Ramsey classes, which
may be seen as top of the line among the various Ramsey properties, was originally somewhat limited. This is no surprise due to the connection with ultrahomogeneous structures~\cite{Nevsetvril1989a}: all Ramsey classes of undirected graphs have been known since 1977~\cite{Nevsetvril1977} (recently, a full classification was also given for directed graphs~\cite{Jasinski2013}). This connection led to the classification
programme of Ramsey classes~\cite{Nevsetril2005} and, in an important new twist, to the connection with topological dynamics and ergodic theory~\cite{Kechris2005}. Thanks to intensive research, we know many more examples of Ramsey classes nowadays.

This paper is a contribution to this development. We present general theorems showing that classes satisfying given local properties are Ramsey. These are far-reaching generalisations of the authors' solution to the bowtie-free problem~\cite{Hubivcka2014}. This development also led to rethinking some of the fundamentals of Ramsey theory. This is outlined in this introduction.

Let us start with the key definition of this paper.  Let $\K$ be a class of structures endowed with embeddings between its members (mostly the embeddings will be clear from the context). For objects $\str{A},\str{B}\in \K$ denote by ${\str{B}\choose \str{A}}$ the set of all sub-objects of $\str{B}$, which are isomorphic to $\str{A}$. (By a sub-object we mean that the inclusion is an embedding.) Using this notation, the definition of a Ramsey class gets the following form:

\medskip

A class $\mathcal C$ is a \emph{Ramsey class} if for every two objects $\str{A}\in \mathcal C$ and $\str{B}\in\mathcal C$ and for every positive integer $k$ there exists an object $\str{C}\in\mathcal C$ such that the following holds: For every partition of ${\str{C}\choose \str{A}}$ into $k$ classes there is $\widetilde{\str B} \in {\str{C}\choose \str{B}}$ such that ${\widetilde{\str{B}}\choose \str{A}}$ belongs to one class of the partition.  It is usual to shorten the last part of the definition to $\str{C} \longrightarrow (\str{B})^{\str{A}}_k$.

\subsection*{Which classes are Ramsey?}
In other words, which classes allow such a generalisation of the Ramsey theorem?

These questions may be less elusive than it seems at first glance as we can use the above-mentioned connection between Ramsey classes and ultrahomogeneous structures.
The \emph{Ramsey classification programme} was depicted in~\cite{Nevsetril2005} by the following diagram:

\begin{center}
\begin{tikzpicture}[auto,
    box/.style ={rectangle, draw=black, thick, fill=white,
      text width=9em, text centered,
      minimum height=2em}]
     \tikzstyle{line} = [draw, thick, -latex',shorten >=2pt];
    \matrix [column sep=5mm,row sep=3mm] {
      \node [box] (Ramsey) {Ramsey\\ classes};
      &&\node [box] (amalg) {amalgamation classes};
      \\
\\
      \node [box] (lift) {special structures};
      &&\node [box] (lim) {ultrahomogeneous structures};
      \\
    };
    \begin{scope}[every path/.style=line]
      \path (Ramsey)   -- (amalg);
      \path (amalg)   -- (lim);
      \path (lim)   -- (lift);
      \path (lift)   -- (Ramsey);
    \end{scope}
  \end{tikzpicture}
\end{center}

Here is the Ramsey classification programme in words: Under mild assumptions, every Ramsey class leads to an amalgamation class (by~\cite{Nevsetvril1989a}, and in full generality by~\cite{Nevsetril2005,Kechris2005}) and amalgamation classes in turn correspond to (infinite) ultrahomogeneous structures (\Fraisse{} limits~\cite{Fraisse1953}), which are the objects of interest of the (Lachlan--Cherlin) classification programme of ultrahomogeneous structures~\cite{Lachlan1980,Cherlin1998,Cherlin2013}.

However, not every ultrahomogeneous structure gives a Ramsey class. We often need to make the structure even more uniform and rigid by adding some additional information (such as ordering).
For such special ultrahomogeneous structures we can then hope to prove the last implication.

Recently, this programme took a more concrete form~\cite{Bodirsky2011a,The2013b,Melleray2015} asking whether every $\omega$-categorical ultrahomogeneous structure $\str{A}$ has a finite (or more generally precompact) expansion (called a \emph{lift} in this paper) so that the corresponding class of all its finite substructures (called its \emph{age}) is Ramsey.

If such a (more concrete) approach were true then the lack of symmetry (expressed by ultrahomogeneity) and lack of rigidity (expressed by special lifts) would be the only obstacles to being Ramsey and the Ramsey classification programme.
However, Evans~\cite{Evans} recently found examples of ultrahomogeneous structures (based on Hrushovski's predimension constructions) which have no ``good'' Ramsey lift
(that is, the lifted class  adds only finitely many additional relations of each arity, \ie{}\ a \emph{precompact lift}, see Definition~\ref{defn:precompact}).  This indicates that the answer to the classification programme  may be more complicated than originally thought. For a refinement of~\cite{Evans} using the main result of this paper, see~\cite{Evans2}.

\medskip

As mentioned, amalgamation property is the central necessary condition for a class to be Ramsey.
The main result of this paper (Theorem~\ref{thm:mainstrongclosures}) gives a sufficient structural
condition for a class of ordered structures to be Ramsey. The
condition can be seen as a strengthening of amalgamation (and we call it \emph{$(\mathcal
R,\mathcal U)$-multiamalgamation}) involving an explicit closure description $\mathcal U$ and additional assumptions about completions relative to a given Ramsey class $\mathcal R$. 
Theorem~\ref{thm:mainstrongclosures} is inspired by our recent result~\cite{Hubivcka2014} that bowtie-free graphs have a precompact Ramsey lift.
This is another example of a combinatorial phenomenon dear to P.~Erd\H os: A seemingly special problem may lead to a rich theory.

We also prove Theorem~\ref{thm:mainstrong} which, somewhat surprisingly, gives sufficient conditions for a subclass of a Ramsey class to be Ramsey: local finiteness (see Definition~\ref{def:localfinite}) and strong amalgamation are enough.
Both the aforementioned theorems have a form of an implication: To show that a class $\mathcal K$ is Ramsey one needs a class $\mathcal R\supseteq \mathcal K$ which is known to be Ramsey.
Our Ramsey theorem for finite models (Theorem~\ref{thm:models}) provides such Ramsey class $\mathcal R$, even for languages containing both relations and (partial) functions.
\medskip

Structural Ramsey theory uses  the partite construction as its main proof technique. It was developed by Ne\v set\v ril and R\"odl in a series of papers~\cite{Nevsetvril1976,Nevsetvril1977,Nevsetvril1979,Nevsetvril1981,Nevsetvril1982,Nevsetvril1983,Nevsetvril1984,Nevsetvril1987,Nevsetvril1989,Nevsetvril1990,Nevsetvril2007}. We use a form of partite construction with unary closures as introduced in our earlier paper~\cite{Hubivcka2014} extending it to non-unary closures and also further generalising the iterated partition construction~\cite{Nevsetvril2007} to a local amalgamation argument. In a way, our paper is further evidence for the surprising effectivity of the partite construction in structural Ramsey theory.
This paper gives the presently most general formulation of the partite construction. Of course one could formulate this in categorical terms (as opposed to the model-theoretic language used here) but it remains to be seen if such a translation would produce any new interesting Ramsey classes. Nevertheless, Theorems~\ref{thm:mainstrong} and~\ref{thm:mainstrongclosures} present a unified approach to many ad-hoc applications of the partite construction.

We develop a
method for giving an explicit Ramsey lift for classes defined by forbidden
homomorphism-embeddings.  Generalising Ramsey lifts of classes defined by forbidden homomorphisms from a finite set~\cite{Nevsetvril}
and bowtie-free graphs~\cite{Hubivcka2014}, we give a sufficient condition to being Ramsey
for classes defined by forbidden homomorphisms from an infinite set as well as for classes of structures with functions. As a consequence of this we prove that all classes defined by means of forbidden homomorphisms from finitely many structures have a precompact Ramsey expansion (Corollary~\ref{cor:forbH}).
For classes defined by forbidden monomorphisms the situation is much more complex (even on the side of universality, where undecidability is conjectured) and we essentially prove that the Ramsey property in many instances does not present any new restrictions (see Theorem~\ref{thm:CSSramsey}).

In Section~\ref{sec:lifts} we work with classes of structures determined by a (possibly infinite) set of forbidden homomorphisms (or, more precisely, homomorphism-embeddings, a more restricted notion of homomorphism which is an
embedding on irreducible structures).
Such classes were studied earlier (\eg{}\ in~\cite{Komjath1988, Komjath1999, Cherlin1999}), however, in order to reach the level of description needed for Ramsey constructions (particularly for the partite construction) we have to describe our classes more explicitly. This leads to the notions of \emph{pieces} and \emph{witnesses}  (see Definition~\ref{defn:piece}) inspired by our earlier papers~\cite{Hubicka2013,Hubicka2009}.
The whole process can be described as \emph{homogenisation} (a term coined in~\cite{Covington1990}) and it amounts to lifting the class to an amalgamation class.

Many known amalgamation classes are in fact (or can be easily lifted to) multiamalgamation classes. This allows us to give  multiple applications of the main results in Section~\ref{sec:examples}. Our starting point is the Ne\v set\v ril--R\" odl theorem~\cite{Nevsetvril1977} (here Theorem~\ref{thm:NR}). Our examples of Ramsey classes include known examples such as (finite) acyclic graphs and partially ordered sets with linear extension~\cite{Nevsetvril1984}, ordered metric spaces~\cite{Nevsetvril2007} or convexly ordered $\mathcal H$-colourable graphs. Many new examples follow. In particular, we fully characterise Ramsey lifts of metric spaces with a given closed set of distances (Section~\ref{sec:metric2}, Theorem~\ref{thm:Smetric}), thereby solving a problem from~\cite{The2010} and contributing to a problem from~\cite{Kechris2005}. We also consider classes with function symbols and give the first examples of classes defining partial orders not only on vertices, but also on $n$-tuples and neighbourhoods.
These examples also provide a better understanding of the nature of Theorems~\ref{thm:mainstrong} and~\ref{thm:mainstrongclosures}.
As a consequence we are able to prove the Ramsey property for structures with a linear order on both vertices and relations (Theorem~\ref{thm:TO}).

\medskip

As has been well known since the beginnings of structural Ramsey theory, 
orderings of structures play a special role. In fact, Ramsey
classes always fix a linear order (see \eg{}\ \cite{Kechris2005}). 
We cannot
escape this here.  Theorems~\ref{thm:mainstrong} and~\ref{thm:mainstrongclosures} take the form of 
implication and we thus do not have to speak about ordering at all. It is implicit
and will be mentioned in examples illustrating general results.  In 
Theorem~\ref{thm:NRsimple} and
Section~\ref{sec:lifts} we incorporate the ordering into the language. In
Section~\ref{sec:examples} we relate this to the more traditional approach by considering Ramsey lifts
which add the order.

Some of the results of this paper were outlined in our conference paper~\cite{hubivcka2015ramsey}.

The rich spectrum of examples of Ramsey classes presented in this paper should, we hope, convince the reader that Ramsey classes are by no means isolated outliers.

\subsection{Preliminaries}
\label{sec:perliminaries}
While structural Ramsey theory was developed primarily for graphs, hypergraphs and relational structures,
we use the following generalisation of model-theoretic structures which use both relations and partial functions. 

Let $L=L_\mathcal R\cup L_\mathcal F$ be a language with relational symbols $\rel{}{}\in L_\mathcal R$ and function symbols $F\in L_\mathcal F$, each having associated \emph{arity} denoted by $\arity{}$ for relations and $\farity{}$ for functions.
An \emph{$L$-structure} $\str{A}$ is a structure with {\em vertex set} (or \emph{domain}) $A$, functions $\func{A}{}\colon \dom(\func{A}{})\to A$, $\dom(\func{A}{})\subseteq A^{\farity{}}$, $\func{}{}\in L_\mathcal F$ and relations $\rel{A}{}\subseteq A^{\arity{}}$, $\rel{}{}\in L_\mathcal R$.
The elements of $A$ are called \emph{vertices}.

 The language is usually fixed and understood from the context (and it is in most cases denoted by $L$).  If $A$ is finite, we say that $\str A$ is a \emph{finite structure}. We consider only structures with countably (that is, finitely or countably infinitely) many vertices.
The class of all (countable) $L$-structures will be denoted by $\Str(L)$.
If the language $L$ consists only of relational symbol, we call it a \emph{relational language} and the $L$-structures are called \emph{relational structures}.

Let $\str{A}$ and $\str{B}$ be $L$-structures.
A~\emph{homomorphism} $f\colon \str{A}\to \str{B}$ is a mapping $f\colon A\to B$ such that  for every $\rel{}{}\in L_\mathcal R$ and $\func{}{}\in L_\mathcal F$ we have:
\begin{enumerate}
\item $(x_1,x_2,\ldots, x_{\arity{}})\in \rel{A}{}\implies (f(x_1),f(x_2),\ldots,f(x_{\arity{}}))\in \rel{B}{}$,
\item $f(\dom(\func{A}{}))\subseteq \dom(\func{B}{})$, and
for every $(x_1,x_2,\allowbreak \ldots, x_{\farity{}})\in \dom(\func{A}{})$ it holds that
$$f(\func{A}{}(x_1,x_2,\allowbreak \ldots, x_{\farity{}}))=\func{B}{}(f(x_1),f(x_2),\ldots,f(x_{\farity{}})).$$
\end{enumerate} For a subset $A'\subseteq A$ we denote by $f(A')$ the set $\{f(x): x\in A'\}$ and by $f(\str{A})$ the homomorphic image of a structure $\str{A}$.

 If $f$ is injective, it is called a \emph{monomorphism}. A monomorphism $f$ is an \emph{embedding} if for every $\rel{}{}\in L_\mathcal R$ and $\func{}{}\in L_\mathcal F$ we have:
\begin{enumerate}
\item $(x_1,x_2,\ldots, x_{\arity{}})\in \rel{A}{}\iff (f(x_1),f(x_2),\ldots,f(x_{\arity{}}))\in \rel{B}{}$, and
\item
$(x_1,x_2,\ldots, x_{\farity{}})\in\dom(\func{A}{}) \iff (f(x_1),f(x_2),\ldots,f(x_{\farity{}}))\in \dom(\func{B}{}).$
\end{enumerate}
  If the inclusion $A\subseteq B$ is an embedding, we say that $\str{A}$ is a \emph{substructure} (or a \emph{subobject}) of $\str{B}$. For an embedding $f\colon\str{A}\to \str{B}$ we say that $\str{A}$ is \emph{isomorphic} to $f(\str{A})$ and $f(\str{A})$ is also called a \emph{copy} of $\str{A}$ in $\str{B}$. Thus $\str{B}\choose \str{A}$ is defined as the set of all embeddings of $\str{A}$ in $\str{B}$ (which may be viewed as the set of copies of $\str{A}$ in $\str{B}$).

Notice that an $L$-structure $\str{A}$  is a substructure of $\str{B}$ if $A\subseteq B$ and all relations and functions of $\str{B}$ restricted to $A$
are precisely the corresponding relations and functions of $\str{A}$. In particular, if some $n$-tuple
$\vv{t}$ of vertices of $A$ is in $\dom(\func{B}{})$ then it is also in
$\dom(\func{A}{})$ and $\func{A}{}(\vv{t})=\func{B}{}(\vv{t})$. 
This implies that $\str{B}$ does not induce a substructure on every subset of $\str{B}$ but only on `closed' sets defined as follows.
Given $\str{A}\in \K$ and $B\subseteq A$, the \emph{closure of $B$ in $\str{A}$}, denoted by $\Cl_\str{A}(B)$, is the smallest substructure of $\str{A}$ containing $B$.
A set $B\subseteq A$ is \emph{closed} if $B=\Cl_\str{A}(B)$.
A closure in $\str{A}$ is {\em unary} if for every $B\subseteq A$ it holds that $\Cl_\str{A}(B)=\bigcup_{v\in B}\Cl_\str{A}(v)$.
\medskip

We now review some more standard model-theoretic notions (see
\eg{}\ \cite{Hodges1993}).

\begin{figure}
\centering
\includegraphics{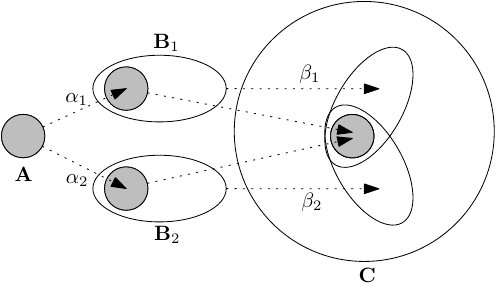}
\caption{An amalgamation of $\str{B}_1$ and $\str{B}_2$ over $\str{A}$.}
\label{amalgamfig}
\end{figure}
Let $\str{A}$, $\str{B}_1$ and $\str{B}_2$ be $L$-structures and $\alpha_1\colon\str{A}\to\str{B}_1$ and $\alpha_2\colon\str{A}\to\str{B}_2$ embeddings. Then
every $L$-structure $\str{C}$
 with embeddings $\beta_1\colon \str{B}_1 \to \str{C}$ and
$\beta_2\colon \str{B}_2\to\str{C}$ such that $\beta_1\circ\alpha_1 =
\beta_2\circ\alpha_2$ is called an \emph{amalgamation} of $\str{B}_1$ and $\str{B}_2$ over $\str{A}$ with respect to $\alpha_1$ and $\alpha_2$ (see Figure~\ref{amalgamfig}).
We will call $\str{C}$ simply an \emph{amalgamation} of $\str{B}_1$ and $\str{B}_2$ over $\str{A}$
(as in the most cases $\alpha_1$ and $\alpha_2$ can be chosen to be inclusion embeddings).

We say that an amalgamation is \emph{strong} if it holds that $\beta_1(x_1)=\beta_2(x_2)$ if and
only if $x_1\in \alpha_1(A)$ and $x_2\in \alpha_2(A)$.  Less formally, a strong
amalgamation glues together $\str{B}_1$ and $\str{B}_2$ with an overlap no
greater than the copy of $\str{A}$ itself.  A strong amalgamation is \emph{free} if $C=\beta_1(B_1)\cup\beta_2(B_2)$ and there are no tuples in any relations of $\str{C}$ spanning both vertices of
$\beta_1(B_1\setminus \alpha_1(A))$ and $\beta_2(B_2\setminus \alpha_2(A))$.

An \emph{amalgamation class} is an isomorphism-closed class $\K$ of finite $L$-structures satisfying the following three conditions:
\begin{enumerate}
\item {\bf Hereditary property:} For every $\str{A}\in \K$ and a substructure $\str{B}$ of $\str{A}$ we have $\str{B}\in \K$;
\item {\bf Joint embedding property:} For every $\str{A}, \str{B}\in \K$ there exists $\str{C}\in \K$ such that $\str{C}$ contains both $\str{A}$ and $\str{B}$ as substructures;
\item {\bf Amalgamation property:} 
For $\str{A},\str{B}_1,\str{B}_2\in \K$ and $\alpha_1$ embedding of $\str{A}$ into $\str{B}_1$, $\alpha_2$ embedding of $\str{A}$ into $\str{B}_2$, there is $\str{C}\in \K$ which is an amalgamation of $\str{B}_1$ and $\str{B}_2$ over $\str{A}$ with respect to $\alpha_1$ and $\alpha_2$.
\end{enumerate}

We will refine amalgamation classes in Definition~\ref{def:multiamalgamation}. The full role of amalgamation classes will be discussed in Section~\ref{sec:lifts}.

\medskip
We consider graphs to be special cases of relational structure in the language containing one symmetric binary relation $E$.
For a structure $\str{A}$ the \emph{Gaifman graph} (in combinatorics
often called a \emph{2-section}) is the graph $\str{G}_\str{A}$ with
vertex set $A$ and $\{x,y\}$  forming an edge of $\str{G}_\str{A}$ if and only if one of the following is satisfied:
\begin{enumerate}
\item There exists a tuple $\vv{t}\in \rel{A}{},\rel{}{}\in L$ such that $x,y\in \vv{t}$,
\item there exists a tuple $\vv{t}\in \dom(\func{A}{}),\func{}{}\in L$ such that $x,y\in \vv{t}$, or
\item there exists a tuple $\vv{t}\in \dom(\func{A}{}),\func{}{}\in L$ such that $x\in \vv{t}$ and $y=\func{A}{}(\vv{t})$ or vice versa.
\end{enumerate}
A structure $\str{A}$ is \emph{connected} if the
Gaifman graph of $\str{A}$
is a connected graph.  A subset $R$ of $A$ is a \emph{(vertex) cut} of $\str{A}$ if
$R$ is closed in $\str{A}$ (that is, $\str{A}$ induces a substructure on $R$) and $\str{G}_\str{A}\setminus R$ is disconnected.
(By $\str{G}_\str{A}\setminus R$ we mean the graph induced by $\str{G}_\str{A}$ on $A\setminus R$.)
Note that not every graph cut of $\str{G}_\str{A}$ is a cut of $\str{A}$.

\section{Construction of Ramsey classes}
\label{sec:results}
The main results of this paper will be introduced here. 
Several old and new concepts have to be recalled and introduced in this section.
\subsection{Statement of the results}
\label{sec:statement}

\begin{figure}
\centering
\includegraphics{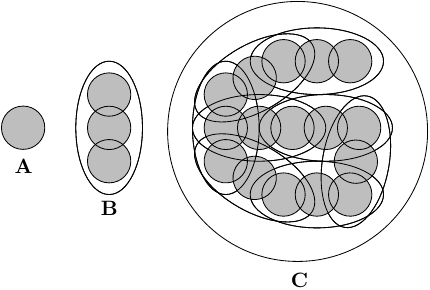}
\caption{Construction of a Ramsey object by multiamalgamation.}
\label{fig:multiamalgam}
\end{figure}
First, we develop a generalised notion of amalgamation which will serve as a useful tool for our constructions of Ramsey objects. As schematically depicted in Figure~\ref{fig:multiamalgam}, Ramsey objects are a result of
amalgamation of multiple copies of a given structure which are all performed at once. In a non-trivial class this leads to many problems. Instead of working with complicated amalgamation diagrams, we  split the process into two steps --- \emph{construction} of the (up to
isomorphism unique) free amalgamation (which yields an incomplete, or ``partial'', structure) followed by a \emph{completion}.

\begin{definition}
\label{def:irreducible}
An $L$-structure $\str{A}$ is \emph{irreducible} if it cannot be created as a free amalgamation of two of its proper substructures.
\end{definition}
Irreducible structures will be our building blocks: In structural Ramsey theory we are fortunate that the structures we are considering are (or may be interpreted as) irreducible (for example, thanks to a linear ordering).

\begin{example}
Observe that a relational structure $\str{A}$ is irreducible if and only if for every pair of distinct vertices $u$, $v$ there is tuple $\vv{t}\in \rel{A}{}$
(of some relation $\rel{}{}\in L$) such that $\vv{t}$ contains both $u$ and $v$.  For $L$-structures in languages containing functions this is not necessarily true.
An example of an irreducible structure in a language with one binary relational symbol $\rel{}{}$ and one unary function symbol $\func{}{}$ is a structure $\str{A}$ (depicted in Figure~\ref{irreduciblefig}) on vertices $A=\{a,b,c,d\}$ where $(a,b)\in \rel{A}{}$,
$\dom(\func{A}{})=\{a,b\}$ and $\func{A}{}(a)=c$, $\func{A}{}(b)=d$. This structure is reducible if $\func{}{}$ is seen as a relation rather than a function.
\begin{figure}
\centering
\includegraphics{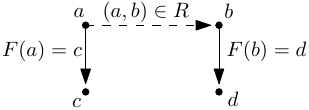}
\caption{An example of an irreducible structure with a binary relation $R$ and a unary function $F$.}
\label{irreduciblefig}
\end{figure}
\end{example}

We introduce the following stronger notion of homomorphism.
\begin{definition}
\label{def:homembed}
 A homomorphism
$f\colon \str{A}\to\str{B}$ is a \emph{homomorphism-embedding}  if $f$ restricted to any irreducible substructure of $\str{A}$      
is an embedding.
\end{definition}
While for undirected graphs without loops the notions of homomorphism and homo\-mor\-phism-em\-bed\-ding coincide, even for relational structures they differ
as shown in the following example.
\begin{example}
Consider language $L$ with one binary relation $\rel{}{2}$ and one ternary relation $\rel{}{3}$.
Structure $\str{A}$ such that $A=\{1,2,3\}$, $\rel{A}{2}=\emptyset$, $\rel{A}{3}=\{(1,2,3)\}$ has a homomorphism to structure $\str{B}$ such that $B=\{1,2,3\}$, $\rel{B}{2}=\{(1,2)\}$, $\rel{B}{3}=\{(1,2,3)\}$ (indeed, the identity is such a homomorphism), but no homomorphism-embedding.
\end{example}
Homomorphism-embedding is the right concept for the problem of turning incomplete structures to complete in the following sense.
\begin{definition}
\label{defn:completion}
Let $\str{C}$ be a structure. An irreducible structure $\str{C}'$ is a \emph{completion}
of $\str{C}$ if there exists a homomorphism-embedding $\str{C}\to\str{C}'$.
If there is a homomorphism-embedding $\str{C}\to\str{C}'$ which is injective,
we call $\str{C}'$ a \emph{strong completion}.

Of particular interest will be whether there exists a completion in a given class $\mathcal K$ of structures. In this case we speak about \emph{$\mathcal K$-completion}.
\end{definition}
For classes of irreducible structures, (strong) completion may be seen as a generalisation of (strong) amalgamation: Let $\K$ be such a class. The (strong)
amalgamation property of $\K$ can be equivalently formulated as follows: For $\str{A}$, $\str{B}_1$, $\str{B}_2 \in \K$ and $\alpha_1$ embedding of $\str{A}$ into $\str{B}_1$, $\alpha_2$ embedding of $\str{A}$ into $\str{B}_2$, there is a (strong) $\mathcal K$-completion
of the free amalgamation of $\str{B}_1$ and $\str{B}_2$ over $\str{A}$ with respect to $\alpha_1$ and $\alpha_2$.

Observe that the free amalgamation is not in $\K$ unless the situation is trivial.
Free amalgamation results in a reducible structure as the pairs of vertices  where
one vertex belong to $\str{B}_1\setminus \alpha_1(\str{A})$ and the other to $\str{B}_2\setminus \alpha_2(\str{A})$ are never both contained in a single tuple
of any relation. Such pairs can be thought of as  \emph{holes} and completion is then a process of filling in the
holes to obtain irreducible structures while preserving all embeddings of irreducible structures.

\medskip

For applications, it is important that in many cases the existence of $\K$-completions and strong $\K$-completions coincide. This can be formulated as follows.
\begin{prop}
\label{prop:strongcompletion}
Let $L$ be a language such that all function symbols in $L$ have arity 1 (there is no restriction on relational symbols) and let $\K$ be a hereditary class of finite irreducible $L$-structures with the strong amalgamation property.
Then every finite $L$-structure $\str{A}$ has a $\K$-completion if and only if it has a strong $\K$-completion.
\end{prop}
\begin{proof}
Assume, to the contrary, that there is an $L$-structure $\str{A}$ with no strong
$\K$-completion, an $L$-structure $\str{B}\in \mathcal K$ and a
homomorphism-embedding $f\colon \str{A}\to\str{B}$ (that is, $\str B$ is a $\mathcal K$-completion of $\str A$). Among
all such examples choose one such that there is no $L$-structure $\str A'$ with homomorphism-embeddings $\str A\to\str A'\to \str B$ such that $|\str A'|\leq |\str A|$ and $\str A'\neq \str B$.

\begin{figure}
\centering
\includegraphics{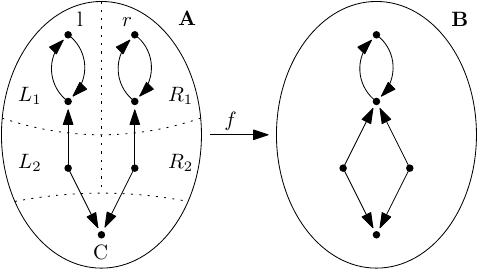}
\caption{An example of a decomposition of a structure $\str{A}$ containing one unary function constructed in the proof of Proposition~\ref{prop:strongcompletion}.}
\label{fig:Merge}
\end{figure}
We decompose the vertex set of $\str{A}$ into five parts denoted by $L_1$, $L_2$, $R_1$, $R_2$, and $C$ as depicted in Figure~\ref{fig:Merge} by the following procedure.

Because $f$ is not a strong $\mathcal K$-completion, we know that there is a pair
of vertices $l\neq r\in A$ such that $f(l)=f(r)$.  Now observe that, by
the non-existence of $\str A'$, for every other pair of vertices $v_1\neq
v_2\in A$ satisfying $f(v_1)=f(v_2)$ it holds that one vertex is in $\Cl_\str{A}(l)$
and the other is in $\Cl_\str{A}(r)$: Indeed, otherwise we could first identify only vertices from $\Cl_\str{A}(l)$ with vertices from $\Cl_\str{A}(r)$, yielding such a structure $\str A'$.

Because vertex closures are irreducible
substructures, we know that $f$ identifies two irreducible substructures
$\str{U}=\Cl_\str{A}(l)$ and $\str{V}=\Cl_\str{A}(r)$ of $\str{A}$ to one and is injective otherwise.

Put $L_1=U\setminus V$ and $R_1=V\setminus U$. Observe that because $l$ and $r$ can be chosen arbitrarily, if a substructure
of $\str{A}$ contains a vertex of $L_1$ it contains all vertices of $L_1$. By symmetry the same holds for
 $R_1$.  Denote by $L_2$ the set of all vertices $v\in A\setminus
L_1$ such that $L_1\subseteq \Cl_\str{A}(v)$.  Analogously denote by $R_2$ the set of
all vertices $v\in A\setminus R_1$ such that $R_1\subseteq \Cl_\str{A}(v)$.  $L_2$
and $R_2$ are disjoint because $f$ is an embedding on irreducible substructures
and thus no vertex closure (which is an irreducible substructure) can contain
both $L_1$ and $R_1$ (as $f(L_1)=f(R_1)$). By a similar irreducibility argument we get that there is no tuple $\vv{t}\in \rel{A}{}$, $\rel{}{}\in L$, containing both a vertex from $L_1\cup L_2$ and a vertex from $R_1\cup R_2$.

Let $C$ be the set of all vertices
whose vertex closure does not contain $L_1$ nor $R_1$, that is, $C=A\setminus (L_1\cup L_2\cup R_1\cup R_2)$.  
Because all functions are unary, $\str{A}$ induces a substructure $\str{C}$ on $C$.
Similarly denote by $\str{A}_l$ the substructure induced by $\str{A}$ on $C\cup L_1\cup L_2$ 
and by $\str{A}_r$ the substructure induced by $\str{A}$ on $C\cup R_1\cup R_2$.

Because $\K$ is hereditary and $f$ is injective on $A\setminus (L_1\cup R_1)$ we know that $\str B$ is a strong $\mathcal K$-completion of $\str{A}_l$, $\str{A}_r$ and $\str{C}$.
Applying the strong amalgamation property of $\K$,
there is $\str{D}\in \mathcal K$ which is a strong amalgamation of $f(\str{A}_l)$ and $f(\str{A}_r)$
over $f(\str{C})$, hence a strong $\K$-completion of $\str A$, which is a contradiction.
\end{proof}
Proposition~\ref{prop:strongcompletion} does not hold for $L$-structures in full generality. 
In the following example one can see that the condition on functions being unary cannot be removed.
\begin{example}[Graphs of girth at least 5]
\label{example:C4}
Consider the class $\mathcal C_{\mathrm{girth}\geq 5}$ of all finite graphs of girth at least 5 (that is, containing no 3-cycles and 4-cycles) seen as relational structures with one binary relation $E$.
This is not a strong amalgamation class of irreducible structures per se, but
can be lifted to a class $\mathcal C^+_{\mathrm{girth}\geq 5}$ by
adding two symbols:
\begin{enumerate}
\item A partial binary function $F$ mapping every pair of disjoint
vertices to the unique vertex connected to both of them if such a vertex
exists (to obtain the strong amalgamation property);
\item a binary relation $R$ containing all pairs of vertices (to obtain irreducibility).
\end{enumerate}

Now observe that the structure $\str{A}$ such that $A=\{1,2,3,4\}$, $E_\str{A} = \{\{1,2\},\allowbreak \{2,3\},\allowbreak \{3,4\},\allowbreak \{4,1\}\}$ (that is, $\str A$ is a graph 4-cycle), $\dom(\func{A}{})=\emptyset$ and $(u,v)\in \rel{A}{}$ if and only if $(u,v)\in E_\str{A}$
 has a $\mathcal C^+_{\mathrm{girth}\geq 5}$-completion (which corresponds to a homomorphism from 4-cycle to an edge) but it has
no strong $\mathcal C^+_{\mathrm{girth}\geq 5}$ completion.  This shows that the assumption of functions being unary
in Proposition~\ref{prop:strongcompletion} is necessary. 

Note that finding a (precompact) Ramsey lift of $\mathcal C^+_{\mathrm{girth}\geq 5}$ is a long standing open
problem (in a special cases stated in~\cite{Nevsetvril1987, Nevsetvril1995})
which cannot be directly solved by applying the results from this paper.
\end{example}
\subsubsection{Ramsey theorem for strong amalgamation classes}
The following is the key definition of this paper. It defines the main property used for obtaining Ramsey classes of $L$-structures.
\begin{definition}
\label{def:localfinite}
Let $\mathcal R$ be a class of finite irreducible structures and let $\mathcal K$ be a subclass of $\mathcal R$. We say
that $\mathcal K$ is a \emph{locally finite subclass} of $\mathcal R$ if for every $\str{C}_0 \in \mathcal R$ there is a finite integer $n = n(\str {C}_0)$ such that 
every structure $\str C$ has a $\K$-completion providing that it satisfies the following:
\begin{enumerate}
\item $\str{C}_0$ is an $\mathcal R$-completion of $\str{C}$,
\item every irreducible substructure of $\str{C}$ is in $\K$, and
\item every substructure of $\str{C}$ with at most $n$ vertices has a $\K$-comple\-tion.
\end{enumerate}
\end{definition}
The true meaning of Definition~\ref{def:localfinite} will be manifested in the examples below and in Section~\ref{sec:examplesstrong}. In most cases we will make use of Proposition~\ref{prop:strongcompletion} and analyse the conditions of Definition~\ref{def:localfinite} for
strong $\K$-completions only.
For the benefit of the reader we now give two simple examples.

\begin{example}[Metric spaces with distances $1$, $2$, $3$, and $4$]
Consider a language $L$ containing binary relations $\rel{}{1}$, $\rel{}{2}$, $\rel{}{3}$, and $\rel{}{4}$ which we understand as
distances. Let $\mathcal R$ be the class of all irreducible finite structures where all four relations are symmetric,
irreflexive and every pair of distinct vertices is in precisely one of relations $\rel{}{1}$, $\rel{}{2}$, $\rel{}{3}$, or $\rel{}{4}$ ($\mathcal R$ may be viewed a class of 4-edge-coloured complete graphs). Let $\mathcal K$ be a subclass of $\mathcal R$ consisting of those structures satisfying
the triangle inequality (\ie{}\ $\mathcal K$ is the class of finite metric spaces with distances 1, 2, 3, and 4).

It is not hard to verify that $L$-structure $\str{C}$ which has a completion to some $\str{C}_0\in  \mathcal R$ (that is, all relations are
symmetric and irreflexive and every pair of distinct vertices is in at most one relation) can be
completed to a metric space if and only if it contains no non-metric triangles (\ie{}\ triangles with
distances 1--1--3, 1--1--4 or 1--2--4) and no 4-cycle with distances 1--1--1--4, see Figure~\ref{fig:4cycles}.
(We prove a generalisation of this fact in Proposition~\ref{prop:sauer1}.)
\begin{figure}
\centering
\includegraphics{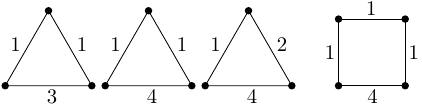}
\caption{Cycles with no completion to a metric space with distances $1,2,3,$ and $4$.}
\label{fig:4cycles}
\end{figure}
 It follows that $\mathcal K$ is a locally
finite subclass of $\mathcal R$ and for every $\str{C}_0 \in \mathcal R$ we can put $n(\str{C}_0) = 4$.
\end{example}

\begin{example}[Metric spaces with distances $1$ and $3$]
\label{example:13}
Now consider the class $\mathcal K_{1,3}$ of all metric spaces which use only distances 1 and 3. It is easy to see
that $\mathcal K_{1,3}$ is not a locally finite subclass of $\mathcal R$. To see that let $\str T\in \mathcal R$ be the triangle with distances 1--1--3. Now consider a cycle $\str C_n$ of length $n$ with one edge of distance three and the others of distance one (as depicted in Figure~\ref{ultra}).
\begin{figure}
\centering
\includegraphics{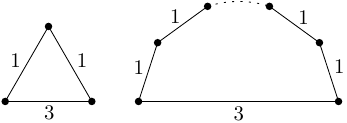}
\caption{Cycles with no completion to a metric spaces with distances 1 and 3.}
\label{ultra}
\end{figure}
$\str{T}$ is a completion of $\str C_n$, however it has no $\mathcal K_{1,3}$-completion. Moreover, every proper substructure of $\str C_n$ (that is, a path consisting of
at most one edge of distance three and others of distance one) does have a $\mathcal K_{1,3}$-completion.
It follows that there is no $n(\str{T})$ and thus $\mathcal K_{1,3}$ is not a locally finite subclass of $\mathcal R$.

We will further discuss metric spaces in Sections~\ref{sec:metric} and~\ref{sec:metric2}.
\end{example}

Our first theorem\footnote{Some papers which use Theorem~\ref{thm:mainstrong} cited it as Theorem~2.1 in the original draft (\eg~\cite{Evans2}).} gives a surprisingly compact sufficient condition
for Ramsey classes:

\begin{theorem}
\label{thm:mainstrong}
Let $L$ be a language, let $\mathcal R$ be a Ramsey class of irreducible finite $L$-structures and let $\K$ be a hereditary
locally finite subclass of $\mathcal R$ with the strong amalgamation property.
Then $\K$ is a Ramsey class.

Explicitly:
For every pair of struc\-tures $\str{A}, \str{B}\in\K$ there exists
a structure $\str{C} \in \K$  such that
$$
\str{C} \longrightarrow (\str{B})^{\str{A}}_2.
$$
\end{theorem}

\begin{remark}[on irreducibility]
The condition on $\mathcal R$ to be a class of irreducible structures may seem
too restrictive. It is however trivially satisfied in all applications we discuss, because it is usually guaranteed by orderings.  It is a
known fact that every Ramsey class (which is an age of a homogeneous structure) fixes a linear order on
vertices, see \eg{}\ \cite{Kechris2005} and \cite{Bodirsky2015} for a combinatorial proof. In such cases 
we can assume that
every structure in the class has a binary relation representing the
order. This order makes the structure irreducible in the sense of
Definition~\ref{def:irreducible}. (In more detail we discuss this in Section~\ref{sec:imaginaries}.)
\end{remark}

Theorem~\ref{thm:mainstrong} is an implication. To show that
a given class $\mathcal K$ of $L$-structures is Ramsey we need a Ramsey class $\mathcal R\supseteq \mathcal K$ such that $\mathcal K$ is locally finite in $\mathcal R$.
The base Ramsey classes can be obtained 
using the Ne\v set\v ril--R\"odl theorem (formulated here as Theorem~\ref{thm:NR}). This theorem can however be applied only if $L$ contains no function symbols.

In order to facilitate base Ramsey classes for general $L$-structures, we prove Theorem~\ref{thm:models}, showing that for every language $L$ containing a binary symbol $\leq$ the class $\ordclass{Str}(L)$
of all finite $L$-structures $\str{A}$ where $\leq_\str{A}$ is a linear ordering of vertices is Ramsey.

Theorem~\ref{thm:models} will be proved using Theorem~\ref{thm:mainstrongclosures} which we introduce in the next section and which may be seen
as the main result of this paper. However, for most applications it is enough (and much more convenient) to
combine Theorems~\ref{thm:mainstrong} and~\ref{thm:models}.
The reader is thus now welcome to skip to Section~\ref{sec:examples} to see multiple
applications of the combination of Theorems~\ref{thm:mainstrong} and~\ref{thm:models}. In particular, notice that
 functions may be necessary while applying Theorem~\ref{thm:mainstrong} even for 
classes of structures in relational languages. This happens, for example, for classes without the strong amalgamation property (the lack of which gives rise to the model-theoretic notion of algebraic closures, see Definition~\ref{def:algebraic}) or classes of structures with definable equivalences (such as the class of metric spaces with distances $1$ and $3$ discussed above, this time giving a link to the model-theoretic notion of elimination of imaginaries, see Section~\ref{sec:imaginaries}).

Theorem~\ref{thm:mainstrongclosures} is more technical and it will be
introduced in the next section. In order to handle functions, we interpret them
as relations with degree restrictions and call this a closure description.

\subsubsection{Ramsey theorem with closures}
The nature of the partite construction (and thus of the proofs in this paper) is to start with a Ramsey theorem, use it as a ``template'' for the partite construction and end up with a stronger Ramsey theorem. In order to introduce functions into the language, we will first understand them as normal relations and use the Ramsey theorem for relational structures as a template upon which we then ensure that the said relations in fact represent functions.

  In this section we introduce the concept of \emph{closure descriptions} which let us speak about
closed substructures in relational languages. Given a closure description
$\mathcal U$ we will define $\mathcal U$-closed structures
and $\mathcal U$-closed substructures, $\mathcal U$-closures, $\mathcal
U$-irreducible structures and $\mathcal U$-homomorphism-embeddings.  All of these
are natural generalisations of the standard notions. The connection to finite models (\ie{}\ structures with
both functions and relations) is exemplified in Proposition~\ref{prop:closures} below.

\begin{definition}
\label{def:closuredesc}
Let $L$ be (not necessarily relational) language.
A~\emph{closure description $\mathcal U$} is a (possibly infinite) set of pairs $(\rel{}{U},\str{R})$ where
$\rel{}{U}\in L$ is a relational symbol of arity $n$ and $\str{R}$ is a non-empty irreducible $L$-structure on vertices $\{1,2,\ldots, m\}$, $m< n$.
We will refer to the relations $\rel{}{U}$ as \emph{closure relations}, and to the structures $\str{R}$ as the \emph{roots of the closures}.
\end{definition}
Given a structure $\str{A}$, a closure description $\mathcal U$ should be understood as follows.
Every pair $(\rel{}{U},\str{R})\in \mathcal U$ declares that the relation $\rel{A}{U}$  of arity $n$ is in fact a partial function $\func{U}{}\colon {\str{A}\choose \str{R}}\to A^{n-|R|}$.
We always assume that
for every $\vv{t}\in \rel{A}{U}$ the first $\vert R\vert $ vertices denote the copy of $\str{R}$ and the remaining $(n-\vert R\vert )$
of the vertices assign a value to the given copy of $\str{R}$.
\begin{definition}
\label{def:uclosed}
Let $L$ be language.
Given an $L$-structure $\str{A}$ and a relation $\rel{A}{}\in L$ of arity $n$,
the \emph{$\rel{A}{}$-out-degree} of
a $k$-tuple $(v_1,v_2,\ldots, v_{k})$ is the number of $(n-k)$-tuples $(v_{k+1},v_{k+2},$ $\ldots, v_{n})$ such that
$(v_1,v_2,\ldots, v_n)\in \rel{A}{}$.

Let $\mathcal U$ be a closure description for $L$.
We say that an $L$-structure $\str{A}$ is \emph{$\mathcal U$-closed} if for every pair $(\rel{}{U},\str{R})\in \mathcal U$
it holds that the $\rel{A}{U}$-out-degree of an $\vert R\vert $-tuple $\vv{t}$ (of vertices of $\str{A}$) is one if and only if $\vv{t}$ represents an embedding of $\str{R}$ to $\str{A}$ and zero otherwise.

Let $\str{A}$ be an  $\mathcal U$-closed structure and $B\subseteq A$. The \emph{$\mathcal U$-closure of $B$ in $\str{A}$},
denote by $\Cl^\mathcal U_\str{A}(B)$,
is the minimal $\mathcal U$-closed substructure of $\str{A}$ containing
$B$.
\end{definition}
The closure description is denoted by $\mathcal U$ (``uz\'av\v er''---Czech word for closure) due to the lack of other letters.

\begin{example}[Interpretation of unary functions in relational language]
Consider a language $L$ consisting of a single unary function $\func{}{}$ and a (relational) language $L'$ consisting of a unary relation $\rel{}{D}$ and a binary relation $\rel{}{U}$.
Given an $L$-structure $\str{A}$ such that for every $v\in \dom(\func{A}{})$ it holds that $F(v)\neq v$, one can construct a relational $L'$-structure $\str{A}'$ on the same vertex set by putting $\rel{}{D}=\dom(\func{A}{})$ and $(u,v)\in \rel{}{U}$ if and only if $\func{A}{}(u)=v$.  Now consider the closure description given by $\mathcal U=\{(\rel{}{U},\str{R})\}$ where $\str{R}$ is the $L'$-structure on the vertex set $\{1\}$ such that $\rel{R}{D}=\{(1)\}$ and $\rel{R}{U}=\emptyset$. Observe that a subset $B\subseteq A$ is closed in $\str{A}$ if and only if it is $\mathcal U$-closed in $\str{A}'$.
It is thus possible to interpret $L$-structures as relational $L'$-structures.

Alternatively, one can view $\str{A}'$ as a directed graph where $\mathcal U$-closed substructures are exactly those induced subgraphs $\str{B}'\subseteq \str A'$ with no directed edges going from $B'$ to $A'\setminus B'$

A similar correspondence between substructures of structures in languages containing functions and $\mathcal U$-closed substructures of relational structures will be established in full generality in Proposition~\ref{prop:closures}. Classes with unary functions are further discussed in Section~\ref{sec:unaryfunctions}.
\end{example}

\begin{remark}
As usual in model theory we introduced functions whose value ranges are singletons. For Ramsey classes this is equivalent to functions from $n$-tuples to $m$-tuples because all
vertices are ordered and thus such function can be translated into $m$ functions to singletons. For closure descriptions we however allow images to be $m$-tuples because
the original motivation for this term is the notion of algebraic closure (see Definition~\ref{def:algebraic}) and in such context mapping $n$-tuples to $m$-tuples is more natural.
\end{remark}
Observe that because roots are non-empty, the empty structure is always $\mathcal U$-closed.
Special cases that are important to us deserve special names: If all roots are singletons (that is, have only one vertex), we speak about \emph{unary closures}. 

With this notion, we can refine our basic concepts of irreducible substructures and ho\-mo\-mor\-phism-embeddings.
\begin{definition}
\label{def:Uirreducible}
\label{def:Uhomomorphism-embedding}
Let $\mathcal U$ be closure description in a (not necessarily relational) language $L$.
An $L$-structure $\str{A}$ is \emph{$\mathcal{U}$-irreducible} if it cannot be created as a free amalgamation of two of its proper $\mathcal{U}$-closed substructures.

A homomorphism
$f\colon \str{A}\to\str{B}$ is a \emph{$\mathcal U$-homomorphism-embedding}  if $f$ restricted to any $\mathcal U$-irreducible substructure of $\str{A}$ is an embedding.

Let $\str{C}$ be a structure. An irreducible structure $\str{C}'$ is a \emph{$\mathcal U$-completion}
of $\str{C}$ if there is a $\mathcal U$-homomorphism-embedding $\str{C}\to\str{C}'$.
If there is a $\mathcal U$-homomorphism-embedding $\str{C}\to\str{C}'$ which is injective,
we call $\str{C}'$ a \emph{strong $\mathcal U$-completion} of $\str C$.

Again, if there exists a $\mathcal U$-completion in a given class $\mathcal K$ of structures, we call it a \emph{$(\mathcal K,\mathcal U)$-completion}.
\end{definition}
Observe that for $\mathcal{U}=\emptyset$ a structure is $\mathcal U$-irreducible if and only if it is irreducible in the sense of Definition~\ref{def:irreducible}. For non-empty $\mathcal U$ this is not necessarily the case.

To make the verifying the existence of a completion easier, we further weaken the notion to the
following variant which is still sufficient for obtaining the Ramsey property:

\begin{definition}
Let $\str{C}$ be a structure and let $\str{B}$ be an irreducible substructure of $\str{C}$.  We
say that an irreducible structure $\str{C}'$ is a \emph{completion of $\str{C}$ with respect to
copies of $\str{B}$} if there exists a function
$f\colon C\to C'$ such that for every  $\widetilde{\str{B}}\in
{\str{C}\choose \str{B}}$ the function $f$ restricted to $\widetilde{B}$
is an embedding of $\widetilde{\str{B}}$ to $\str{C}'$.

If $\str{C}'$ belong to a given class $\K$, then $\str{C}'$ is called \emph{$\K$-completion of $\str{C}$ with respect to copies of $\str{B}$}.
\end{definition}
This is the weakest notion of completion which preserves the Ramsey property for given structures $\str{A}$ and $\str{B}$. Note that $f$ does not need to be a homo\-morphism-embedding (not even a homomorphism).

We now state all the conditions for our main result as one definition:
\begin{definition}
\label{def:multiamalgamation}
Let $L$ be a language, $\mathcal R$ be a Ramsey class of finite irreducible $L$-structures and $\mathcal U$ be a closure description (for $L$).
We say that a subclass $\mathcal K$ of $\mathcal R$  is an \emph{$(\mathcal R,\mathcal U)$-multiamalgamation class} if
the following conditions are satisfied:
\begin{enumerate}
 \item {\bf $\mathcal U$-closed structures:} $\mathcal K$ consists of finite $\mathcal U$-closed $L$-structures.
 \item\label{cond:hereditary} {\bf Hereditary property for $\mathcal U$-closed substructures:} For every $\str{A}\in \K$ and for every $\mathcal U$-closed substructure $\str{B}$ of $\str{A}$ we have $\str{B}\in \K$.
 \item\label{cond:amalgamation} {\bf Strong amalgamation property:}
For $\str{A},\str{B}_1,\str{B}_2\in \K$ and embeddings $\alpha_1\colon\str{A}\to\str{B}_1$, $\alpha_2\colon\str{A}\to\str{B}_2$, there is $\str{C}\in \K$ which is a strong amalgamation of $\str{B}_1$ and $\str{B}_2$ over $\str{A}$ with respect to $\alpha_1$ and $\alpha_2$.
 \item\label{cond:completion} {\bf Locally finite completion property:} Let $\str{B}\in \K$ and $\str{C}_0\in \mathcal R$. 
Then there exists $n=n(\str{B},\str{C}_0)$ such that if a $\mathcal U$-closed $L$-structure $\str{C}$ satisfies the following:
\begin{enumerate}
 \item $\str{C}_0$ is a $\mathcal U$-completion of $\str{C}$, 
 \item every $\mathcal U$-irreducible substructure of $\str{C}$ is in $\mathcal K$, and
 \item every (not necessarily $\mathcal U$-closed) substructure of $\str{C}$ with at most $n$ vertices has a $(\K,\mathcal U)$-completion,
\end{enumerate}
then there exists $\str{C}'\in \K$ which is a completion of $\str{C}$ with respect to copies of $\str{B}$.
\end{enumerate}
\end{definition}
\begin{remark}
We shall see that this seemingly elaborate definition is in fact very flexible and easy to apply.
For an amalgamation class $\mathcal K$ consisting of irreducible structures it is up to interpretation always possible to construct a
closure description $\mathcal U$ such that $\mathcal K$ satisfies the first three conditions in Definition~\ref{def:multiamalgamation}. (The only exception are amalgamation classes whose \Fraisse{} limit contains a nontrivial closure of the empty set. Those can be always corrected by an appropriate interpretation.) Also, as in our definition the empty set is always $\mathcal{U}$-closed, we get the strong joint embedding property: For every $\str{A}, \str{B}\in \K$ there exists $\str{C}\in \K$ such that $\str{C}$ contains both $\str{A}$ and $\str{B}$ as (vertex) disjoint substructures.
It is the locally finite completion property which is the crucial condition for
$\K$ to be a Ramsey class. 
\end{remark}

We can now state our main result\footnote{Some papers which use Theorem~\ref{thm:mainstrongclosures} cited it as Theorem~2.2 of the original draft (\eg~\cite{Evans2}).} as:
\begin{theorem}
\label{thm:mainstrongclosures}
Every $(\mathcal R,\mathcal U)$-multiamalgamation class $\K$ is Ramsey.
\end{theorem}

\subsubsection{Ramsey theorem for finite models}
\label{sec:functions}
Let $L$ be a language containing a binary relation $\leq$. We denote by $\ordclass{\Str}(L)$ the class of all finite $L$-structures (models) $\str{A}\in\Str(L)$ where the set $A$ is linearly ordered by the relation $\leq$.  We show the following theorem.
\begin{theorem}[Ramsey theorem for finite models]
\label{thm:models}
For every language $L$ (possibly involving both relations and functions) containing a binary relation $\leq$ the class $\ordclass{\Str}(L)$ is a Ramsey class.
\end{theorem}
To prove Theorem~\ref{thm:models}, we interpret functions as relations and apply
Theorem~\ref{thm:mainstrongclosures} with an appropriate closure description.
First, we formulate equivalence of closure descriptions and models involving functions.

\begin{prop}
\label{prop:closures}
For every language $L=L_\mathcal R \cup L_\mathcal F$ there is a relational language $L'\supseteq L_\mathcal R$ and a closure
description $\mathcal U$ for $L'$ such that there is function $\Rel$ assigning to every $L$-structure $\str{A}$ an $L'$-structure $\Rel(\str{A})$
on the same vertex set satisfying the following:
\begin{enumerate}
\item $\Rel$ is a bijection between $L$-structures and $\mathcal U$-closed $L'$-structures,
\item $\rel{A}{}=\nbrel{\Rel(\str{A})}{}$ for every $R\in L_\mathcal R$,
\item $f\colon A\to B$ is a homomorphism-embedding from an $L$-structure $\str{A}$ to an $L$-structure $\str{B}$ if and only if it is an $\mathcal U$-homomorphism-embedding
from $\Rel(\str{A})$ to $\Rel(\str{B})$.
\end{enumerate}
\end{prop}
In other words, Proposition~\ref{prop:closures} establishes that $\Rel$ is an isomorphism of the category
of $L$-structures with homomorphism-embeddings and the category of $\mathcal U$-closed $L'$-structures with $\mathcal U$-homomorphism-embeddings.
\begin{proof}
We first define the language $L'$. It will contain $L_\mathcal R$ and furthermore for every $n$-ary function $\func{}{}\in L_\mathcal F$ we add (disjointly) an $n$-ary relational symbol $\rel{}{\func{}{}}$ and all $(j+1)$-ary relational symbols $\rel{}{\func{}{},j,{\vv r}}$, where $1\leq j\leq n$ and $\vv r$ is an $n$-tuple of integers from $\{1,2,\ldots, j\}$ such that the first occurrences of the integers form and increasing sequence and every integer from $\{1,2,\ldots, j\}$ occurs at least once in $\vv r$.

The closure description $\mathcal U$ will consist of all pairs
$(\rel{}{\func{}{},j,{\vv{r}}},\str{R})$, where $\rel{}{\func{}{},j,{\vv{r}}}$ is in $L'$ and
$\str{R}$ is an $L'$-structure on the vertex set $R=\{1,2,\ldots, j\}$ such that
$\vv{r}\in \rel{R}{\func{}{}}$.

For every $L$-structure $\str{A}\in \Str(L)$ we construct
a relational $L'$-structure $\Rel(\str{A})\in \Str(L')$ on
the same vertex set as follows:
\begin{enumerate}
 \item $\nbrel{\Rel(\str{A})}{}=\rel{A}{}$ whenever $\rel{}{}\in L_\mathcal R$,
 \item $\nbrel{\Rel(\str{A})}{\func{}{}}=\dom(\func{A}{})$ for every $\func{}{}\in L_\mathcal F$, and
 \item $(a_1,a_2,\ldots, a_j,b)\in \nbrel{\Rel(\str{A})}{\func{}{},j,{\vv{r}}}$ if and only if $\func{A}{}(a_{r_1},a_{r_2},\ldots, a_{r_n})=b$.
\end{enumerate}

We thus encoded the domain of every function $\func{A}{}$ by the relation $\rel{A}{\func{}{}}$ and
used closure relations $\rel{}{\func{}{},j,{\vv{r}}}$ to represent the function on every tuple in the domain of $\func{A}{}$.
It follows directly from the definitions that $\Rel$ is the desired function.
\end{proof}

Thanks to Proposition~\ref{prop:closures}, it is enough to apply Theorem~\ref{thm:mainstrongclosures}
for relational base Ramsey class $\mathcal R$ which we get from the following unrestricted form of the
Ne\v set\v ril--R\"odl theorem (see Section~\ref{sec:NR} for a more general formulation of this theorem).
\begin{theorem}[Unrestricted Ne\v set\v ril--R\"odl Theorem~\cite{Nevsetvril1977,Abramson1978}]
\label{thm:NRsimple}
For every relational language $L$ containing a binary relation $\leq$ the class $\ordclass{\Str}(L)$ is a Ramsey class.
\end{theorem}
\begin{proof}[Proof of Theorem~\ref{thm:models}]
Given a language $L$, apply Proposition~\ref{prop:closures} to obtain a language $L'$, a closure description $\mathcal U$ and a function $\Rel$.
By Theorem~\ref{thm:NRsimple} we know that $\ordclass{\Str}(L')$ is a Ramsey class and applying
Theorem~\ref{thm:mainstrongclosures} we get that the class of all $\mathcal U$-closed structures
in $\ordclass{\Str}(L')$ is a Ramsey class. To see the locally finite completion property for $n=1$ one can simply put $\str{C'}$ to be $\str{C}$ with $\leq_\str{C}$ completed to linear order. This is always possible because $\str{C}$ has a projection to $\str{C}_0$ which is linearly ordered by $\leq_{\str{C}_0}$.
Because $\Rel$ preservers the relationship of being a substructure, this in turn gives the Ramsey property for $\ordclass{\Str}(L)$.
\end{proof}
\begin{remark}
The proof of the Ramsey theorem for finite models (Theorem~\ref{thm:models}) involves most of the techniques introduced in this paper. It can be generalised further, we however decided to only formulate it in this concise form here. We believe it nicely complements existing results for relational structures (Abramson--Harrington~\cite{Abramson1978}, Ne\v set\v ril--R\"odl~\cite{Nevsetvril1977}). The generalisation to free amalgamation classes has further interesting consequences for the ordering property and is proved in~\cite{Evans3} by applying the main results of this paper.
\end{remark}

\begin{remark}
A Ramsey theorem for structures involving both relations and functions
is given in \cite{Solecki2012}. The notion of functions used
in~\cite{Solecki2012} is however different from the standard model-theoretic one
and corresponds to a combination of relations and unary closures.
The Ramsey property proved in~\cite{Solecki2012} then follows by Theorem~\ref{thm:mainstrongclosures} (but not vice versa).
\end{remark}

\begin{figure}
{
\centering
\begin{tikzpicture}[auto,
    lemma/.style ={rectangle, draw=black,  fill=white,
      text width=8em, text centered,
      minimum height=2em},
    theorem/.style ={rectangle, draw=black, thick, fill=white,
      text width=8em, text centered,
      minimum height=2em},
    block_left/.style ={rectangle, draw=black, thick, fill=white,
      text width=15em, text ragged, minimum height=2em, inner sep=6pt},
    block_noborder/.style ={rectangle, draw=none, thick, fill=none,
      text width=17em, text centered, minimum height=1em},
    block_leftnoborder/.style ={rectangle, draw=none, thick, fill=none,
      text width=9em, text ragged, minimum height=1em},
    block_assign/.style ={rectangle, draw=black, thick, fill=white,
      text width=17em, text ragged, minimum height=3em, inner sep=6pt},
    block_lost/.style ={rectangle, draw=black, thick, fill=white,
      text width=15em, text ragged, minimum height=3em, inner sep=6pt},
      line/.style ={draw, thick, -latex', shorten >=0pt}]
     \tikzstyle{line} = [draw, thick, -latex',shorten >=2pt];
    \matrix [column sep=3mm,row sep=3mm] {
      \node [theorem] (HJ) {Hales--Jewett Theorem};
      &&\node [theorem] (Ramsey) {Ramsey Theorem};
      \\
      &\node [lemma] (Partite Lemma) {Partite Lemma};&\node [lemma] (Partite Construction) {Partite construction};&\node[block_leftnoborder,xshift=-3em] (NRt) {Ne\v{s}et\v ril--\\R\"odl\\ Theorem};\\
      &&\node [theorem] (NR) {Ramsey Property of ordered structures\\(Theorem~\ref{thm:NR})};&&\\
\\
       &\node [lemma] (Partite Construction 1) {Partite construction for $\mathcal U$-substructures\\(Lemma~\ref{lem:preclosures})};\\
      &&\node [lemma] (Partite Construction 2) {$\mathcal U$-closed partite construction\\(Lemma~\ref{lem:closures})};\\
\\
      \node [lemma] (Partite Lemma with Closures){Partite Lemma with closures\\(Lemma~\ref{partlem})};&&\node [lemma] (Iterated) {Iterated partite construction\\(Lemmas~\ref{lem:iteratedpartitestep2} and~\ref{lem:iteratedpartite2})};\\
\\
      \node [theorem] (Res1) {Ramsey property of locally finite strong amalgamation classes\\(\textbf{Theorem~\ref{thm:mainstrong}})};
      &\node [theorem] (Res2) {Ramsey property of multiamalgamation classes\\(\textbf{Theorem~\ref{thm:mainstrongclosures}})};\\
      \node [theorem] (Models) {$\ordclass{\Str}(L)$ is Ramsey\\
				   (\textbf{Theorem~\ref{thm:models}})};
      &\node [theorem] (Lift) {Explicit description of lift of $\Forb(\F)$\\
				   (\textbf{Theorem~\ref{mainthm}})};&
      \node [theorem] (LiftThm) {Ramsey property of lifts of classes $\Forb(\F)$ (\textbf{Theorem~\ref{thm:main}})};\\
    };
    \begin{scope}[every path/.style=line]
      \path (HJ)   -- (Partite Lemma);
      \path (Partite Lemma)   -- (Partite Construction);
      \path (Ramsey)   -- (Partite Construction);
      \path (Partite Construction) -- (NR);
      \draw[dotted] (Partite Lemma.north west) + (-1.5ex,1.5ex) rectangle ([xshift=1.5ex,yshift=-1.5ex]NR.south east);
      \path (HJ) -- (Partite Lemma with Closures);
      \path (Partite Lemma with Closures) -- (Partite Construction 1);
      \path (Partite Lemma with Closures) -- (Partite Construction 2);
      \path (Partite Construction 1) -- (Partite Construction 2);
      \path (NR) -- node{$\str{C}_0\longrightarrow (\str{B})^\str{A}_2$} (Partite Construction 2);
      \path (Partite Construction 2) -- node{$\mathcal U$-closed $\str{C}_1\longrightarrow (\str{B})^\str{A}_2$} (Iterated);
      \path (Partite Lemma with Closures) -- (Iterated);
      \draw[->] (Iterated.south east) to [bend right=90] node[swap]{$n(\str{C_0})\times$} (Iterated.north east);
      \path (Iterated) -- (Res1.north east);
      \path (Iterated) -- (Res2);
      \path (Res2) -- (Models);
      \path (Iterated) -- (LiftThm);
      \path (Lift) -- (LiftThm);
    \end{scope}
  \end{tikzpicture}
}
\caption{The structure of proofs of the main results.}
\label{fig:plan}
\end{figure}
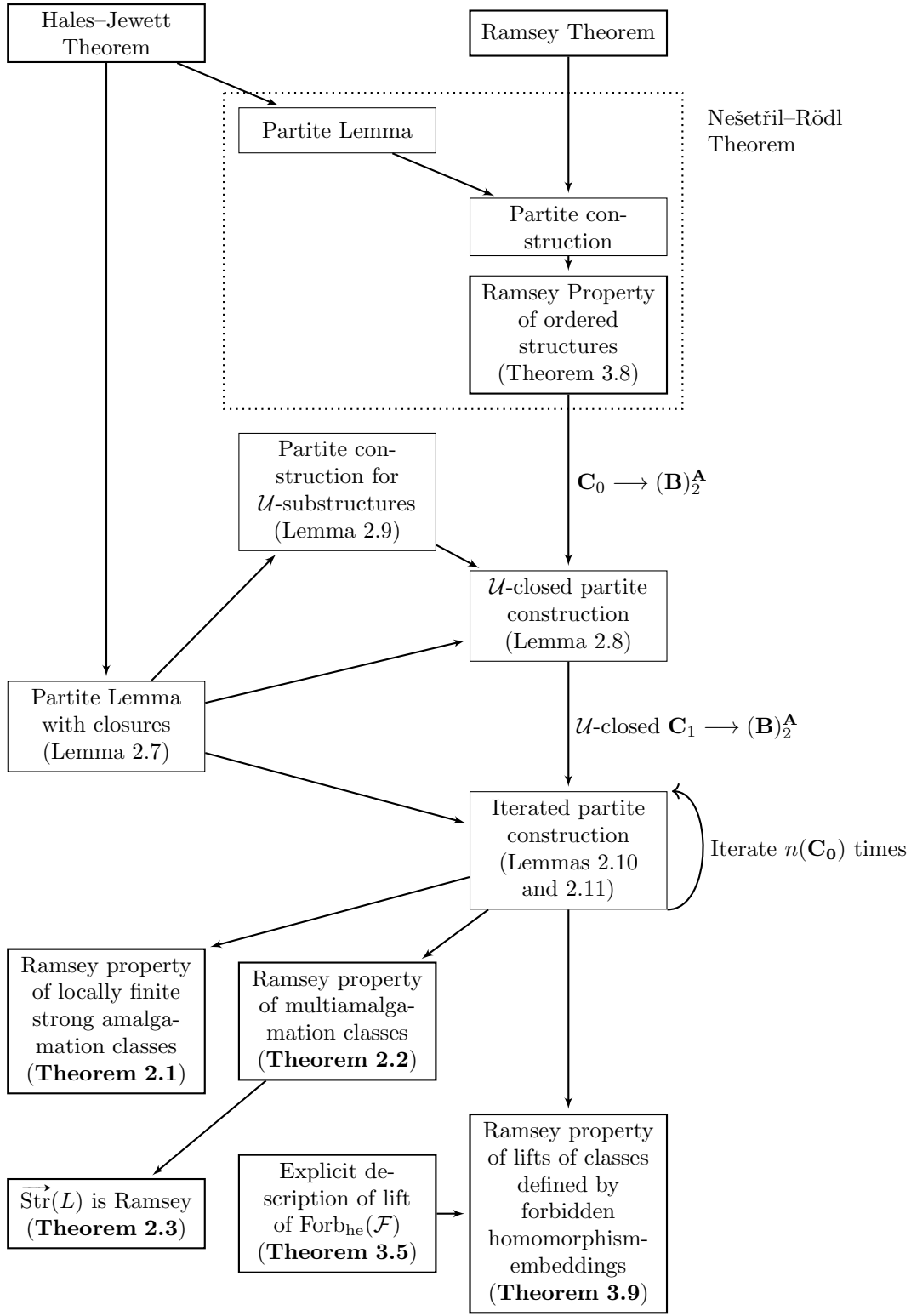

\subsection{Proof structure of Theorems~\ref{thm:mainstrong} and~\ref{thm:mainstrongclosures}}
\label{sec:ramsey}
The overall structure of the proofs of Theorems~\ref{thm:mainstrong} and~\ref{thm:mainstrongclosures}
is depicted in Figure~\ref{fig:plan}. We now further complement it with a few comments which may assist the reader in understanding the proof. 

We give an explicit construction of Ramsey objects.
Using Proposition~\ref{prop:closures}, it suffices to work with relational structures only.
A relational language $L$ will be fixed thorough this section.
Given a Ramsey class $\mathcal R$, its subclass $\mathcal K$ and structures $\str{A},\str{B}\in \K$, we apply to Ramsey property of $\mathcal R$
to obtain $\str{C}_0\longrightarrow (\str{B})^\str{A}_2$. In all applications discussed, this will be done by
an application of the Ne\v set\v ril--R\"odl Theorem (Theorem~\ref{thm:NRsimple}) which is depicted as the first structured Ramsey theorem in Figure~\ref{fig:plan}.
Subsequently, we use three variants of the partite construction to obtain a Ramsey structure $\str{C}$ with the desired properties.
Towards this end, in Sections~\ref{sec:partlem} and~\ref{sec:partiteconstruction} we give a new variant of the partite construction for
classes with closures (generalising our techniques introduced in
\cite{Hubivcka2014} and strengthening them to non-unary closures).  In Section~\ref{sec:iterated} we
introduce the Iterated partite construction  for strong amalgamation classes
(extending results of~\cite{Nevsetvril2007}) and finally we combine
both to obtain our main results in Section~\ref{sec:mainres}.

To construct $\mathcal U$-closed structures (Definition~\ref{def:uclosed}) we
proceed in several steps. The following notions capture two ``weaker'' notions
of closed structures and substructures which will be used in our constructions.
This definition will be used in the iteration of Lemma~\ref{lem:preclosures} and Lemma~\ref{lem:closures} where it is necessary to consider structures
which are not $\mathcal U$-closed, but they satisfy the property for selected substructures.

\begin{definition}
\label{defn:usubstructure}
Let $L$ be language, $\mathcal U$ be a closure description in $L$ and $\str{A}$ a substructure of an $L$-structure $\str{B}$.
We say that $\str{A}$ is a \emph{$\mathcal U$-substructure} of $\str{B}$ if for every pair $(\rel{}{U},\str{R})\in \mathcal U$ it holds that if
a tuple $\vv{t} \in \rel{B}{U}$ has all its root vertices in $A$  then
all vertices of $\vv{t}$ are in $A$.

In other
words there is no vertex $v\in B\setminus A$ with a pair $(\rel{}{U},\str{R})\in \mathcal U$ and a tuple $\vv{s} \in \rel{B}{U}$
containing $v$ such that the first $\vert R\vert $ elements of $\vv{s}$ are in $A$.
\end{definition}
\begin{figure}
\centering
\includegraphics{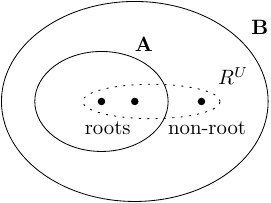}
\caption{$\str{A}$ is not a $\mathcal U$-substructure of $\str{B}$.}
\label{fig:forbidden}
\end{figure}
The forbidden situation is depicted in Figure~\ref{fig:forbidden}.
 The main property
of $\mathcal U$-substructure is captured by the following easy lemma.
\begin{lemma}
\label{lem:usubstructure}
\label{lem:uclosedamalgamation}
For every language $L$ and a closure description $\mathcal U$ in $L$
the following holds:
\begin{enumerate}
\item Let $\str{A}$ be a substructure of a $\mathcal U$-closed $L$-structure $\str{B}$. Then
$\str{A}$ is a $\mathcal U$-substructure of $\str B$ if and only if $\str{A}$ is $\mathcal U$-closed.
\item\label{item:uclosedamalgamation} Let $\str{B}_1$ and $\str{B}_2$ be $\mathcal U$-closed $L$-structures
and $\str{A}$ a $\mathcal U$-closed substructure of both $\str{B}_1$ and $\str{B}_2$.
Further assume that completions in both cases are strong completions.
Then the free amalgamation of $\str{B}_1$ and $\str{B}_2$ over $\str{A}$ is a $\mathcal U$-closed structure.
\end{enumerate}
\end{lemma}
\begin{proof}The proof is easy. In~\ref{item:uclosedamalgamation}, we use the fact that roots in $\mathcal U$ are irreducible structures.
\end{proof}
Another difficulty we need to overcome is the fact that during the partite constructions it is necessary to consider
substructures of $\mathcal U$-closed structures induced on sets which are not $\mathcal U$-closed themselves. Such substructures satisfy the following:
\begin{definition}
Let $\mathcal U$ be a closure description.
We say that $\str{A}$ is \emph{$\mathcal U$-semi-closed} if for every pair
$(\rel{}{U},\str{R})\in \mathcal U$ it holds that the
$\rel{A}{U}$-out-degree of an $\vert R\vert $-tuple $\vv{t}$ of vertices of
$\str{A}$ is at most one if there is an embedding from $\str{R}$ 
to $\vv{t}$ and zero otherwise.
\end{definition}
The following concept of size (which one can view as the smallest size of a generating set) will be the basic parameter for the induction in the iterated partite construction:
\begin{definition}
The \emph{$\mathcal U$-size} of a structure
$\str{B}$ is the number of vertices of the smallest substructure $\str{A}$ of $\str{B}$ such that the $\mathcal U$-closure
of $\str{A}$ in $\str{B}$ is $\str{B}$.
\end{definition}

Observe that for every substructure $\str{B}_0$ of a $\mathcal U$-closed structure $\str{B}$
the $\mathcal U$-size of $\str{B}_0$ is the same as the $\mathcal U$-size of the $\mathcal U$-closure of
$\str{B}_0$ in $\str{B}$.
\subsection{Partite lemma with closures}
\label{sec:partlem}
An essential part of our construction of Ramsey objects respecting a given closure description is a closure refinement of the partite lemma~\cite{Nevsetvril1989} which deals with the following objects.
\begin{definition}[$\str{A}$-partite system]
Let $L$ be a language and $\str{A}$ be an $L$-structure. Assume that $A = \{1, 2,\ldots, a\}$.  An \emph{$\str{A}$-partite $L$-system} is a tuple $(\str{A},{\mathcal X}_\str{B},\allowbreak \str{B})$ 
where $\str{B}$ is an $L$-struc\-ture and $\mathcal X_\str{B}=\{X^1_\str{B},X^2_\str{B},\ldots, X^a_\str{B}\}$  is a partition of the vertex set of $\str{B}$ into $a$ classes ($X^i_\str{B}$  are called \emph{parts} of $\str{B}$)  such that 
\begin{enumerate}
\item the mapping $\pi$ such that $\pi$ maps every $x \in X^i_\str{B}$ to $i$, $i = 1,2,\ldots,a$, is a
homo{\-}morphism-embedding $\str{B}\to\str{A}$ (we call $\pi$ the \emph{projection}); 
\item every tuple in every relation of $\str{B}$ contains at most one element of each class $X^i_\str{B}$ (these tuples are called \emph{transversal} with respect to the partition).
\end{enumerate}
\begin{remark}
Our definition differs from the definition used in~\cite{Nevsetvril1989}. We do not treat the linear order separately
and we also assume the existence of a homomorphism-embedding $\str{B}\to \str{A}$ (which simplifies the proof of the partite lemma). (This formulation of the partite system does not lead directly to a proof of the Ne\v set\v ril--R\"odl Theorem~\ref{thm:NR} itself. We aim for simplicity here.)
\end{remark}
\end{definition}
\emph{Isomorphisms} and \emph{embeddings} of $\str{A}$-partite systems, say of $\str{B}_1$ into $\str{B}_2$, are defined as isomorphisms and embeddings of structures together with the condition that all parts 
are being preserved (the part $X^i_{\str{B}_1}$ is mapped to $X^i_{\str{B}_2}$  for every $i = 1,2,\ldots,a$).

\begin{lemma}[Partite lemma with closures]\label{partlem}
Let $L$ be a relational language, $\mathcal U$ be a closure description in $L$,
$\str{A}$ be a finite $\mathcal U$-closed $L$-structure and $\str{B}$ be a finite $\mathcal U$-semi-closed $\str{A}$-partite $L$-system such that every vertex of $\str{B}$ is contained in a copy of $\str{A}$.
Then there exists a finite $\mathcal U$-semi-closed $\str{A}$-partite $L$-system $\str{C}$
such that 
$$
\str{C}\longrightarrow (\str{B})^\str{A}_2.
$$
(Here we consider $\str{A}$ to be an $\str{A}$-partite system and thus all copies of $\str{A}$ in $\str{B}$ preserve parts.)

Moreover there exists a family $\mathcal B$ of copies of $\str{B}$ in $\str{C}$ such that:
\begin{enumerate} 
\item \label{lem:partitni:family}For every $2$-colouring of all substructures of $\str{C}$ which are isomorphic to $\str{A}$ there exists $\widetilde{\str{B}}\in \mathcal B$ such that all substructures of $\widetilde{\str B}$ which are isomorphic to $\str{A}$ are monochromatic (thus $\mathcal B$ is a Ramsey system of copies of $\str{B}$ in $\str{C}$).
\item \label{lem:partitni:usubstructure}Every $\widetilde{\str{B}}\in \mathcal B$ is a $\mathcal U$-substructure of $\str{C}$.
\end{enumerate}
Finally, if $\str{B}$ is $\mathcal U$-closed then $\str{C}$ is $\mathcal U$-closed, too, and if every $\mathcal U$-irreducible substructure of $\str{B}$ is transversal, then every $\mathcal U$-irreducible substructure of $\str{C}$ is transversal, too. 
\end{lemma}

\begin{remark}
Our proof is inspired by the proof of the partite lemma in~\cite{Nevsetvril1989} which uses the Hales--Jewett theorem~\cite{Hales1963}. 
We give an easy description of $\str{C}$ as a product.
 This simplification follows from the assumption that $\str{B}$ is an $\str{A}$-partite-system and thus has a homomorphism-embedding projection to $\str{A}$.  This easier description of $\str{C}$ allows us to verify the additional properties of $\str{C}$ needed to carry our later proofs.
  The key observation of our earlier paper~\cite{Hubivcka2014} is that unary closures can be preserved by the partite construction.
  We show this also for non-unary closures (by a different technique which uses the nested partite construction instead of free amalgamation) in Section~\ref{sec:partiteconstruction}.
\end{remark}

For completeness, we briefly recall the Hales--Jewett Theorem~\cite{Hales1963}:
Consider a family of functions $f\colon \{1,2,\ldots, N\} \to \Sigma$ for some finite alphabet $\Sigma$. Let $(\omega,h)$ be a pair where $\emptyset\neq\omega\subseteq \{1,2,\ldots, N\}$ and $h$ is a function from $\{1,2,\ldots, N\}\setminus \omega$ to $\Sigma$. A~\emph{combinatorial line} $\mathcal L$ given by $(\omega,h)$ is then the family of all those functions $f\colon\{1,2,\ldots, N\} \to \Sigma$ that are constant on $\omega$ and $f(i)=h(i)$ otherwise. The Hales--Jewett theorem guarantees, for a sufficiently large $N$, that for every $2$-colouring of functions $f\colon\{1,2,\ldots, N\} \to \Sigma$ there exists a monochromatic combinatorial line. For brevity, we identify the description $(\omega,h)$ with the combinatorial line itself.

\begin{proof}[Proof of Lemma~\ref{partlem}]
Assume without loss of generality that $A = \{1, 2,\ldots, a\}$  and denote by $\mathcal X_\str{B} = \{X^1_\str{B},X^2_\str{B},\ldots, X^a_\str{B}\}$ the parts of $\str{B}$.

Let $N$ be a sufficiently large integer (that will be specified later) and construct an $\str{A}$-partite $L$-system  $\str{C}$
with parts $\mathcal X_\str{C} = \{X^1_\str{C},X^2_\str{C},\ldots, X^a_\str{C}\}$  as follows:
\begin{itemize}
\item For every $1\leq i\leq a$ let $X_\str{C}^i$ be the set of all functions $$f\colon \{1,2,\ldots,N\}\to X_\str{B}^i.$$
\item For every relation $\rel{}{}\in L$, put $$(f_1,f_2,\ldots, f_{\arity{}})\in \rel{C}{}$$ if and only if for every $1\leq i\leq N$ it holds that $$(f_1(i),f_2(i),\ldots, f_{\arity{}}(i))\in \rel{B}{}.$$
\end{itemize}
This completes the construction of $\str{C}$.

\medskip
We now check that $\str{C}$ indeed is  a $\mathcal U$-semi-closed $\str{A}$-partite $L$-system
with parts $\mathcal X_\str{C} = \{X^1_\str{C},\allowbreak X^2_\str{C},\allowbreak \ldots,\allowbreak  X^a_\str{C}\}$.
Most of this follows immediately from the definition,
we only verify that $\str{C}$ is $\mathcal U$-semi-closed.  For a contradiction, assume the existence of a pair $(\rel{}{U},\str{R})\in \mathcal U$, an embedding  $f\colon \str{R}\to \str{C}$, and an $|R|$-tuple $\vv{r}=(r_1,r_2,\ldots, r_{\vert R\vert })$ of vertices of $f(\str{R})$ such that the $\rel{C}{U}$-out-degree of $\vv r$ is more than one. Denote by $m$ the number of vertices of $\str{R}$ and by $n$ the arity of $\rel{}{U}$.  Because the $\rel{A}{U}$-out-degree of $\vv r$ is more than one, we have two $(n-m)$-tuples $(f_{m+1},f_{m+2},\ldots, f_{n}) \neq (f'_{m+1},f'_{m+2},\ldots, f'_{n})$ such that: 
$$(r_1,r_2,\ldots, r_m, f_{m+1},f_{m+2},\ldots, f_{n})\in \rel{C}{U},\hbox{ and}$$
$$(r_1,r_2,\ldots, r_m, f'_{m+1},f'_{m+2},\ldots, f'_{n})\in \rel{C}{U}.$$
By the construction of $\str{C}$ we thus know that for every $1\leq j\leq N$:
 $$(r_1(j),r_2(j),\ldots, r_m(j), f_{m+1}(j),f_{m+2}(j),\ldots, f_{n}(j))\in \rel{B}{U},\hbox{ and}$$
 $$(r_1(j),r_2(j),\ldots, r_m(j), f'_{m+1}(j),f'_{m+2}(j),\ldots, f'_{n}(j))\in \rel{B}{U}.$$
Since the $\rel{A}{U}$-out-degrees are at most one in $\str{B}$, we know that $f_k(j)=f'_k(j)$ for every $m < k\leq n$ and $1\leq j\leq N$, a contradiction.
The second part of the definition of $\mathcal U$-semi-closed structure is trivially satisfied by the existence of the projection.

By a similar argument it follows that if $\str{B}$ is $\mathcal U$-closed then $\str{C}$ is $\mathcal U$-closed, too.

\medskip
Next we describe the Ramsey family $\mathcal B$ of copies of $\str B$.
Let $\widetilde{\str{A}}_1, \widetilde{\str{A}}_2,\ldots, \widetilde{\str{A}}_t$ be an enumeration of all $\str{A}$-partite subsystems of $\str{B}$ which are isomorphic to $\str{A}$. 
Put  $\Sigma=\{1,2,\ldots, t\}$ which we consider as an alphabet.
Each combinatorial line $\mathcal L=(\omega, h)$ in $\Sigma^N$ corresponds to an embedding $e_\mathcal L\colon\str{B}\to \str{C}$ which assigns to every vertex $v\in X^p_\str{B}$ a function $e_\mathcal L(v)\colon\{1,2,\ldots, N\}\to X^p_\str{B}$ (\ie{}\ a vertex of $X^p_\str{C}$) such that:
$$e_\mathcal L(v)(i)=
 \begin{cases} 
    \hbox{$v$ for $i\in \omega$, and}\\
    \hbox{the unique vertex in $\widetilde A_{h(i)}\cap X^p_\str{B}$ otherwise.}
   \end{cases}
$$
It follows from the construction of $\str{C}$, from the fact that $\str{B}$ has a projection to $\str{A}$ and from the assumption that every vertex of $\str{B}$ belongs to a copy of $\str{A}$ that $e_\mathcal L$ is an embedding.

Let the family $\mathcal B$ consist of all copies $e_\mathcal L(\str{B})$ for some combinatorial line $\mathcal L$.
First we check that every copy in $\mathcal B$ is a $\mathcal U$-substructure of $\str{C}$ (condition~\ref{lem:partitni:usubstructure} from the statement of the lemma).
Assume, to the contrary,
that there is $\widetilde{\str{B}} \in \mathcal{B}$ which corresponds to a combinatorial line $\mathcal L = (\omega, h)$, A pair $(\rel{}{U}, \str{R}) \in \mathcal U$ and
$\vv{t}=(f_1,f_2,\ldots, f_{\arity{U}}) \in \rel{C}{U}$ such that $\{f_1, f_2, \ldots, f_{\vert R\vert}\}\subseteq \widetilde{B}$ and there is a vertex
in $\vv{t}$ which is not in $\widetilde{B}$. 
By the construction of $\str C$ it holds that $(f_1(i),f_2(i),\ldots, f_{\arity{U}}(i))\in \rel{B}{U}$ for every $1\leq i\leq N$. Since $\str B$ is $\mathcal U$-semi-closed, we get that the $\rel{B}{U}$-out-degree of $(f_1(i),\allowbreak f_2(i), \ldots,\allowbreak f_{\vert R\vert}(i))$ is exactly one for every $i$. However, it is easy to see that this implies that in fact $f_j\in \widetilde{B}$ for every $1\leq j\leq \arity{U}$, which is a contradiction.

Now we check condition~\ref{lem:partitni:family} (\ie{}\ that $\mathcal B$ is a Ramsey system of copies of $\str{B}$).
 Let $N$ be the Hales--Jewett number guaranteeing a monochromatic line in any $2$-colouring of $\Sigma^N$.
Now assume that $\mathcal{A}_1, \mathcal{A}_2$ is a $2$-colouring of all copies of $\str{A}$ in $\str{C}$.
Using the construction of $\str{C}$ we see that  among the copies of $\str{A}$ are copies induced by the $N$-tuple $(\widetilde{\str{A}}_{u(1)}, \widetilde{\str{A}}_{u(2)},\ldots,\widetilde{\str{A}}_{u(N)})$ for every function $u\colon\{1,2,\ldots, N\}\to \Sigma$. However, such copies are coded by the elements of the cube $\Sigma^N$ and thus there is a monochromatic 
combinatorial line $\mathcal L$. The monochromatic copy of $\str{B}$ is then $e_\mathcal L(\str B)$ which belongs to $\mathcal B$.

\medskip

Finally, we verify that if every $\mathcal U$-irreducible subsystem of $\str{B}$ is transversal
then also every $\mathcal U$-irreducible subsystem of $\str{C}$ is transversal. Assume the contrary
and denote by $\str{D}$ a non-transversal $\mathcal U$-irreducible subsystem of $\str{C}$.
Let $f_1,f_2,\ldots, f_n$ be an enumeration of the vertices of $\str{D}$ such that $f_1$ and $f_2$ are in the same part.  For every $i\in 1,2,\ldots, N$
denote by $\str{D}_i$ the  substructure of $\str{B}$ on vertices $f_1(i),f_2(i),\ldots, f_n(i)$.
Because $\str{D}_i$ is a homomorphic image of $\str{D}$ it follows that $\str{D}_i$ is not transversal.
Consequently, $f_1(i)=f_2(i)$.  Because this holds for every choice of $i$, we have $f_1=f_2$. A contradiction.
\end{proof}
\subsection{Partite construction with closures}
\label{sec:partiteconstruction}
The main result of this section is the following Lemma.
\begin{lemma}\label{lem:closures}
Let $L$ be a relational language,
let $\mathcal U$ be a closure description in $L$, let $\str{A}$ and $\str{B}$ be finite $\mathcal U$-closed
$L$-structures and let $\str{C}_0$ be a finite $L$-structure such that
$$
\str{C}_0 \longrightarrow (\str{B})^{\str{A}}_2.
$$
Then there exists a finite $\mathcal U$-closed $L$-structure $\str{C}$ with a $\mathcal U$-homomorphism-embed{\-}ding $\str{C}\to \str{C}_0$ such that:
$$
\str{C} \longrightarrow (\str{B})^{\str{A}}_2
$$
and every $\mathcal U$-irreducible substructure of $\str{C}$ is isomorphic to a substructure of $\str{B}$.
\end{lemma}

First, we prove a weaker variant of Lemma~\ref{lem:closures} (the weakening consists in an additional assumption on $\str{C}_0$):
\begin{lemma}\label{lem:preclosures}
Let $L$ be a relational language,
let $\mathcal U$ be a closure description in $L$, let $\str{A}$ and $\str{B}$ be finite $\mathcal U$-closed
$L$-structures and let $\str{C}_0$ be a finite $L$-structure such that
$$
\str{C}_0 \longrightarrow (\str{B})^{\str{A}}_2.
$$
Further assume that every copy of $\str{A}$ in $\str{C}_0$ is in fact a $\mathcal U$-substructure of $\str{C}_0$.
Then there exists a finite $\mathcal U$-closed $L$-structure $\str{C}$ with a $\mathcal U$-homomorphism-embedding $\str{C}\to \str{C}_0$ such that:
$$
\str{C} \longrightarrow (\str{B})^{\str{A}}_2
$$
and every $\mathcal U$-irreducible substructure of $\str{C}$ isomorphic to a substructure of $\str{B}$.
\end{lemma}

\begin{proof}[Proof (an adaptation of~\cite{Nevsetvril1989}).]
Without loss of generality we can assume that $C_0=\{1,2,\ldots, c\}$.
Enumerate all copies of $\str{A}$ in $\str{C}_0$  as  $\{\widetilde{\str{A}}_1, \widetilde{\str{A}}_2,\ldots, \widetilde{\str{A}}_b\}$.
We shall define $\str{C}_0$-partite $\mathcal U$-closed structures $\str{P}_0, \str{P}_1, \ldots, \str{P}_b$
such that:
\begin{enumerate}[label=(\roman*)]
 \item for every $0\leq k<b$ and every 2-colouring of copies of $\str{A}$ in $\str{P}_{k+1}$ there is a copy of $\str{P}_{k}$ in $\str{P}_{k+1}$ such that all copies of $\str{A}$ with a projection to $\widetilde{\str{A}}_{k+1}$ are monochromatic,
 \item\label{1:ii} for every $0\leq k\leq b$ the projection of $\str{P}_k$ to $\str{C}_0$ is a $\mathcal U$-homomorphism-embedding and every $\mathcal U$-irreducible substructure of $\str{P}_k$ is isomorphic to a substructure of $\str{B}$.
\end{enumerate}
We denote the parts of $\str P_k$ as $\mathcal X_{\str{P}_k} = \{X_k^1, X_k^2, \dots,X_k^c\}$. As is usual in the Partite construction, the systems $\str{P}_k$ are called \emph{pictures} and will be constructed by induction on $k$. 

\begin{figure}[t]
\centering
\includegraphics{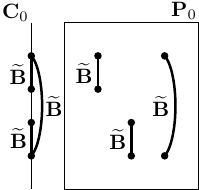}
\caption{The construction of picture $\str{P}_0$.}
\label{fig:picture0}
\end{figure}
\begin{enumerate}
\item
The picture $\str{P}_0$ is constructed as a disjoint union of  copies of $\str{B}$: For every copy $\widetilde{\str{B}}$ of
$\str{B}$ in $\str{C}_0$ we consider a new isomorphic and disjoint copy $\widetilde{\str{B}}'$ in $\str{P}_0$ which intersects the part $X_0^l$  if and only if $\widetilde{\str{B}}$ intersects such that the projection of $\widetilde{\str{B}}'$ is $\widetilde{\str{B}}$ (see Figure~\ref{fig:picture0}).
This is indeed $\mathcal U$-closed and satisfies~\ref{1:ii} as no tuples in any relations between copies are added.

\item

Let $\str{P}_k$ be already constructed.
Let $\str{B}_k$ be the substructure of  $\str{P}_k$ induced by $\str{P}_k$ on vertices which project to $\widetilde{\str{A}}_{k+1}$.
We can also assume that every vertex of $\str{B}_k$ belongs to a copy of $\str{A}$. (If this condition is not satisfied, it is possible to extend $\str{B}$ by free amalgamation with additional copies of $\str{A}$ over the $\mathcal U$-closure of every vertex which does not belong to a copy of $\str{A}$ already. The $\mathcal U$-closure of a vertex must be isomorphic to the $\mathcal U$-closure of its corresponding vertex of $\str{A}$ because the projection is a $\mathcal U$-homomorphism-embedding and $\mathcal U$-closures of vertices are $\mathcal U$-irreducible.)
By the assumption that  $\widetilde{\str{A}}_{k+1}$ is a $\mathcal U$-substructure of $\str{C}_0$ we also know that $\str{B}_k$ is $\mathcal U$-substructure of $\str{P}_k$.

Now we can use the partite Lemma~\ref{partlem} to obtain
a $\mathcal U$-closed $\widetilde{\str{A}}_{k+1}$-partite system $\str{D}_{k+1}$ and a Ramsey system $\mathcal B_{k+1}$ of copies of $\str{B}_k$ which are $\mathcal U$-substructures of $\str{D}_{k+1}$.
Now consider all copies in $\mathcal B_{k+1}$ and extend each of these structures to a copy of $\str{P}_k$ by a free amalgamation. These copies are disjoint outside $\str{D}_{k+1}$ and preserve the parts of all the copies.
The result of this multiple amalgamation is picture $\str{P}_{k+1}$. The construction is depicted in Figure~\ref{fig:picture}.
\begin{figure}
\centering
\includegraphics{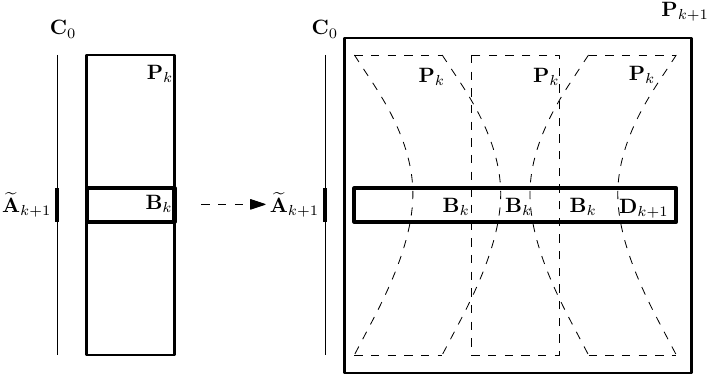}
\caption{The construction of picture $\str{P}_{k+1}$ from picture $\str{P}_k$.}
\label{fig:picture}
\end{figure}
By repeated application of Lemma~\ref{lem:uclosedamalgamation} we know that $\str{P}_{k+1}$ is $\mathcal U$-closed because it is a result of a sequence of free amalgamations of $\mathcal U$-closed structures over their $\mathcal U$-substructures.

It remains to be seen that~\ref{1:ii} holds.  This follows from the fact that
Lemma~\ref{partlem} preserves transversality of $\mathcal U$-irreducible substructures and $\str{D}_{k+1}$ has a $\mathcal U$-homomorphism-embedding to $\widetilde{\str{A}}_{k+1}$. Consequently, all $\mathcal U$-irreducible substructures of $\str{D}_{k+1}$ have an embedding to $\widetilde{\str{A}}_{k+1}$ (which we can assume to have an embedding to $\str{B}$ otherwise $\str{B}$ would be Ramsey for trivial reasons).
The subsequent free amalgamation does not introduce any new $\mathcal U$-irreducible and only copies those of $\str{P}_k$ which satisfy property~\ref{1:ii} by the induction hypothesis.
\end{enumerate}

Put $\str{C} = \str{P}_b$.  It  follows easily that  $\str{C}\longrightarrow
(\str{B})^{\str{A}}_2$. By a backward induction on $k$ one proves  that in every
$2$-colouring of ${\str{C}}\choose{\str{A}}$ there is a copy
$\widetilde{\str{P}}_0$ of $\str{P}_0$ such that the colour of a copy of
$\str{A}$ in $\str{P}_0$ depends only on its projection.  As this in turn induces
colouring of copies of $\str{A}$ in $\str{C}_0$, we obtain a monochromatic copy
of $\str{B}$ in $\widetilde{\str{P}}_0$.
\end{proof}

\begin{proof}[Proof of Lemma~\ref{lem:closures}]
We again apply the Partite construction as in the proof of Lemma~\ref{lem:preclosures}.
However, we repeatedly use Lemma~\ref{lem:preclosures} as a crucial step in the ``picture induction''.

Assume that $C_0=\{1,2,\ldots, c\}$.
Enumerate all copies of $\str{A}$ in $\str{C}_0$  as  $\{\widetilde{\str{A}}_1, \widetilde{\str{A}}_2,$ $\ldots, \widetilde{\str{A}}_b\}$.
We shall define $\str{C}_0$-partite $\mathcal U$-closed structures (pictures) $\str{P}_0, \str{P}_1, \ldots, \str{P}_b$ such that
\begin{enumerate}[label=(\roman*)]
\item for every $0\leq k<b$ and every 2-colouring of copies of $\str{A}$ in $\str{P}_{k+1}$ there is a copy of $\str{P}_{k}$ in $\str{P}_{k+1}$ such
that all copies of $\str{A}$ with projection to $\widetilde{\str{A}}_{k+1}$ are monochromatic,
\item\label{2:ii} for every $0\leq k\leq b$ the projection of $\str{P}_k$ to $\str{C}_0$ is a $\mathcal U$-homomorphism-embedding and every $\mathcal U$-irreducible substructure of $\str{P}_k$ is isomorphic to a substructure of $\str{B}$.
\end{enumerate}
Again we proceed by induction on $k$. 

\begin{enumerate}
\item
Picture $\str{P}_0$ is again constructed as a disjoint union of copies of $\str{B}$: For every copy $\widetilde{\str{B}}$ of
$\str{B}$ in $\str{C}_0$ we consider a new isomorphic and disjoint copy $\widetilde{\str{B}}'$ in $\str{P}_0$ which intersects the part $X_0^l$  if and only if $\widetilde{\str{B}}$ intersects it such that the projection of $\widetilde{\str{B}}'$ is $\widetilde{\str{B}}$. Clearly $\str{P}_0$ is $\mathcal U$-closed and satisfies~\ref{2:ii}.

\item

Let $\str{P}_k$ be already constructed.  
Let $\str{B}_{k+1}$ be the $\mathcal U$-semi-closed substructure of  $\str{P}_k$ induced by $\str{P}_k$ on the vertices of those copies of $\str{A}$ which project to $\widetilde{\str{A}}_{k+1}$.
Observe that in this setting $\str{B}_{k+1}$ is not necessarily $\mathcal U$-closed in $\str{O}_k$ for two reasons. First, $\widetilde{\str{A}}_{k+1}$ may not be an $\mathcal U$-substructure of $\str{C}_0$. Second, while constructing $\str{B}_{k+1}$ we did not include all vertices with a projection to $\widetilde{\str{A}}_{k+1}$.

In this situation, we use the partite Lemma~\ref{partlem} to obtain
a $\mathcal U$-semi-closed $\widetilde{\str{A}}_{k+1}$-partite system $\str{D}_{k+1}$ and a Ramsey system $\mathcal B_{k+1}$ of copies of $\str{B}_k$ which are all $\mathcal U$-substructures of $\str{D}_{k+1}$.
Now consider all copies in $\mathcal B_{k+1}$ and extend each of these structures to a copy of $\str{P}_k$, disjointly outside of $\str{D}_{k+1}$.
The result of this multiple amalgamation is denoted by $\str{O}_{k+1}$. ($\str{O}$ stands for Czech ``obr\'azek'' --- 
a little picture. At this moment  we further refine the partite construction: In the construction
of picture~$\str{P}_{k+1}$ from $\str{P}_k$ we sandwich $\str{O}_{k+1}$ which itself is a result of the partite construction.)
Note that $\str{O}_{k+1}$ is not necessarily $\mathcal{U}$-semi-closed, because $\str{B}_k$ is not necessarily a $\mathcal U$-substructure of $\str{P}_k$.

Denote by $\mathcal{A}_{k+1}$ the set of all copies of $\str{A}$ in $\str{O}_{k+1}$ with projection to $\widetilde{\str{A}}_{k+1}$. 
 We show that for every pair $(\rel{}{U} ,\str{R}) \in \mathcal U$ and every $\vert R\vert $-tuple $\vv{t}$ of vertices of $\str{O}_{k+1}$
such that the $\nbrel{\str{O}_{k+1}}{U}$-out-degree of $\vv{t}$ is more than one it holds that  $\vv{t}$ is never contained in a copy
of $\str{A}$ in $\mathcal A_{k+1}$. This follows from the fact that higher degrees can only be created when amalgamating (freely) $\str P_k$ on top of copies from $\mathcal B_{k+1}$, all copies of $\str{B}_{k+1}$ in $\mathcal B_{k+1}$ are $\mathcal U$-substructures of $\str{D}_{k+1}$,
$\str{D}_{k+1}$ is $\mathcal U$-semi-closed and $\str{P}_k$ is $\mathcal U$-closed.

To apply Lemma~\ref{lem:preclosures}, we turn the $\str{C}_0$-partite system $\str{O}_{k+1}$ to a relational structure $\str{O}^+_{k+1}$ in a lifted language $L^+$ which represents the parts using unary relations. Explicitly, we put $L^+=L\cup \{\rel{}{i}:i\in C_0\}$ and the arity of all new relations is one. The $L^+$-structure $\str{O}^+_{k+1}$ is constructed as follows:
\begin{enumerate}
\item $O^+_{k+1}=O_{k+1}$ (\ie{}\ $\str{O}^+_{k+1}$ has the same vertices as $\str{O}_{k+1}$),
\item for every relation $\rel{}{}\in L$ we put $\nbrel{\str{O}^+_{k+1}}{}=\nbrel{\str{O}_{k+1}}{}$ (\ie{}\ $\str{O}^+_{k+1}$ has the same original relations as $\str{O}_{k+1}$),
\item $v\in \nbrel{\str{O}^+_{k+1}}{i}$ if and only if $\pi(v)=i$.
\end{enumerate}
Analogously, we turn the $\str{C}_0$-partite $L$-system $\str{P}_{k}$ to an $L^+$-structure $\str{P}^+_k$ and $\widetilde{\str{A}}_{k+1}$ to $\str{A}^+$ (where $v\in \nbrel{\str{A}^+}{v}$ for every $v\in A^+$, remember that $\widetilde{\str{A}}_{k+1}\subseteq \str C_0$). Finally, construct a closure description $\mathcal U^+$ (in $L^+$) consisting of all pairs $(\rel{}{U},\str{S}^+)$ where $\rel{}{U}\in L$ and $\str{S}^+$ is an $L^+$-structure such that there exists $(\rel{}{U},\str{S})\in \mathcal U$ with $S=S^+$ and $\nbrel{\str{S}}{}=\nbrel{\str{S}^+}{}$ for every $\rel{}{}\in L$ (that is, $\mathcal U^+$ extends every root of $\mathcal U$ by unary relations in every possible way and thus represents the same closures).

We verify the conditions of Lemma~\ref{lem:preclosures} for these $L^+$-structures. Because the projection is explicitly represented by the unary relations in $L^+$, it follows that $$\str{O}^+_{k+1} \longrightarrow (\str{P}^+_k)^{\str{A}^+}_2.$$ (Here $\str{O}^+_{k+1}$, $\str{P}^+_k$ and $\str{A}^+$ are seen as structures, not partite systems.)
This holds because all copies of $\str{A}^+$ in $\str{O}^+_{k+1}$ correspond to copies of $\str{A}_{k+1}$ in $\mathcal A_{k+1}$ and the Ramsey property for those copies is given by Lemma~\ref{partlem}. We also verified that all such copies are $\mathcal U$-substructures of $\str{O}_{k+1}$ and consequently all copies of $\str{A}^+$ in $\str{O}^+_{k+1}$ are $\mathcal U^+$-substructures.

It follows, by an application of Lemma~\ref{lem:preclosures}, that
there exists a $\mathcal U^+$-closed $L^+$-structure $\str{P}^+_{k+1}$ such that $\str{P}^+_{k+1} \longrightarrow (\str{P}^+_k)^{\str{A}^+}_2$
with a homomorphism-embedding to $\str{O}^+_{k+1}$. 

The $\mathcal U$-closed $\str{C}_0$-partite $L$-system $\str{P}_{k+1}$ is then constructed by re-interpret\-ing $\str{P}^+_{k+1}$ back as a partite system: vertices of $\str{P}_{k+1}$ are the same as vertices of $\str{P}^+_{k+1}$ and all tuples in all relations in the language $L$ are also the same. The parts are determined by the unary relations $\rel{}{i}$ (for every $i\in C_0$ and $v\in P^+_{k+1}$ it holds that $v\in X^i_{k+1}$ if and only if $v\in \nbrel{\str{P}_{k+1}}{i}$).

It remains to verify~\ref{2:ii}. Because we extended the language by unary predicates naming the parts and because every $\mathcal U$-irreducible substructure of $\str{P}^+_{k+1}$ is isomorphic to a $\mathcal U$-irreducible substructure of $\str{P}^+_k$, we know that every $\mathcal U$-irreducible substructure of $\str{P}_{k+1}$ is transversal and isomorphic to a $\mathcal U$-irreducible substructure of $\str{P}_k$ and hence by the induction hypothesis to a substructure of $\str{B}$. This also verifies that the projection is a $\mathcal U$-homomorphism-embedding.
\end{enumerate}

Put $\str{C} = \str{P}_b$. 
Again, 
analogously to the proof of Lemma~\ref{lem:preclosures}, by a backward induction it follows that $\str{C}\longrightarrow (\str{B})^\str{A}_2$.
\end{proof}

\subsection{Iterated partite construction}
\label{sec:iterated}
To prove Theorems~\ref{thm:mainstrong} and~\ref{thm:mainstrongclosures}, we
need a generalisation of the iterated partite construction (in style of~\cite{Nevsetvril2007}).

\begin{lemma}[$j$ completion implies $j+1$ completion]
\label{lem:iteratedpartitestep2}
Let $L$ be a relational language, let $\mathcal U$ be a closure description in $L$, let $\mathcal K$ be a strong amalgamation class of finite irreducible $\mathcal U$-closed
$L$-structures which is hereditary for $\mathcal U$-closed substructures and let $j\geq 0$ be a non-negative integer.
Let $\str{A},\str{B}\in \mathcal K$
and let $\str{C}_0$ be a finite $\mathcal U$-closed $L$-structure such that $$\str{C}_0 \longrightarrow (\str{B})^{\str{A}}_2.$$
Further assume that either $j=0$ or $j>0$ and every $\mathcal U$-closed substructure of $\str{C}_0$ with $\mathcal U$-size at most $j$ has a $(\K,\mathcal U)$-completion. 
Then there exists a $\mathcal U$-closed $L$-structure $\str{C}$ satisfying the following conditions:
\begin{enumerate}[itemsep=0em]
\item There is a $\mathcal U$-homo{\-}morphism-embedding $\str{C}\to \str{C}_0$,
\item $\str{C} \longrightarrow (\str{B})^{\str{A}}_2$,
\item every $\mathcal U$-closed substructure of $\str{C}$ of $\mathcal U$-size at most $j+1$ has a $(\K,\mathcal U)$-completion, and
\item every $\mathcal U$-irreducible substructure of $\str{C}$ is isomorphic to a substructure of $\str{B}$. 
\end{enumerate}
\end{lemma}
\begin{proof}
For the fourth (and last) time we apply the partite construction. We proceed analogously to the proofs 
of Lemmas~\ref{lem:closures} and~\ref{lem:preclosures}.
Again, we enumerate all copies of $\str{A}$ in $\str{C}_0$  as  $\{\widetilde{\str{A}}_1, \widetilde{\str{A}}_2,\ldots, \widetilde{\str{A}}_b\}$.
We then define $\mathcal U$-closed $\str{C}_0$-partite systems $\str{P}_0, \str{P}_1, \ldots, \str{P}_b$ such that:
\begin{enumerate}[label=(\roman*)]
\item\label{3:i} every $\mathcal U$-closed substructure of $\str{P}_k$, $0\leq k\leq b$, of $\mathcal U$-size at most $j+1$ has a $(\K,\mathcal U)$-completion,
\item\label{3:ii} in every $2$-colouring of ${\str{P}_{k+1}}\choose{\str{A}}$, $0\leq k< b$, there exists a copy $\widetilde{\str{P}}_{k}$ such that all copies of $\str{A}$ with a projection to $\widetilde{\str{A}}_{k+1}$ are monochromatic,
\item\label{3:iii} for every $0\leq k\leq b$ the projection of $\str{P}_k$ to $\str{C}_0$ is a $\mathcal U$-homomorphism-embedding and every $\mathcal U$-irreducible substructure of $\str{P}_k$ is isomorphic to a substructure of $\str{B}$.
\end{enumerate}
As before, we get that putting $\str{C} = \str{P}_b$ we have the desired Ramsey property $\str{C}\longrightarrow (\str{B})^{\str{A}}_2$. It remains to prove \ref{3:i}, \ref{3:ii} and \ref{3:iii}.

Put explicitly $\mathcal X_{\str{P}_k} = \{X_k^1, X_k^2, \dots, X_k^c\}$.
We proceed by induction on $k$.

\begin{enumerate}
\item
The picture $\str{P}_0$ is constructed in the same way as in the proof of Lemma \ref{lem:preclosures} as a disjoint union of  copies of $\str{B}$: for every copy $\widetilde{\str{B}}$ of
$\str{B}$ in $\str{C}_0$ we consider a new isomorphic and disjoint copy $\widetilde{\str{B}}'$ in $\str{P}_0$ which intersects the part $X_0^l$  if and only if $\widetilde{\str{B}}$ intersects (so that the projection of $\widetilde{\str{B}}'$ is $\widetilde{\str{B}}$).
Clearly, $\str{P}_0$ has a $(\K,\mathcal U)$-completion (it can be constructed by a series of strong amalgamations over the empty set), which proves property~\ref{3:i}. Property~\ref{3:iii} also follows directly from the construction.

\item
Let $\str{P}_k$ be already constructed.  
 Let $\str{B}_k$ be the $\mathcal U$-substructure of  $\str{P}_k$ induced by $\str{P}_k$ on vertices which project to $\widetilde{\str{A}}_{k+1}$.
$\str{P}_{k+1}$ is constructed in the same way as in the proof of Lemma~\ref{lem:preclosures}:
We use the partite lemma~\ref{partlem} to obtain
a $\mathcal U$-closed $\widetilde{\str{A}}_{k+1}$-partite system $\str{D}_{k+1}$ and the Ramsey system $\mathcal B_{k+1}$.
Now consider all copies in $\mathcal {B}_{k+1}$ and extend each of these structures to a copy of $\str{P}_k$ (using free amalgamation). These copies are disjoint outside of $\str{D}_{k+1}$.  In this extension, we preserve the parts of all the copies.
The result of this series of amalgamations is $\str{P}_{k+1}$. Because $\str{D}_{k+1}\longrightarrow (\str{B}_k)^{\widetilde{\str{A}}_{k+1}}_2$, we know that $\str{P}_{k+1}$ satisfies \ref{3:ii}.

 Because $\str{P}_{k+1}$ is created by a series of free amalgamations of $\mathcal U$-closed structures over $\mathcal U$-substructures it follows that $\str{P}_k$ is $\mathcal U$-closed.
To see that $\str{P}_{k+1}$ satisfies \ref{3:iii} again observe that property \ref{3:iii} is preserved in the application of Lemma~\ref{partlem} as well as during the free amalgamation.

To finish the proof, we need to show~\ref{3:i} for $\str{P}_{k+1}$.
Assume the contrary and denote by $\str{F}$ a $\mathcal U$-substructure of $\str P_{k+1}$ of $\mathcal U$-size at most $j+1$ with no $(\K,\mathcal U)$-completion. Among all such counterexamples, pick $\str{F}$ to have the smallest $\mathcal U$-size (see Figure~\ref{fig:forbpicture2}). Because $\K$ has the strong amalgamation property (in particular over the empty set), we get that $\str{F}$ is connected.
\begin{figure}
\centering
\includegraphics{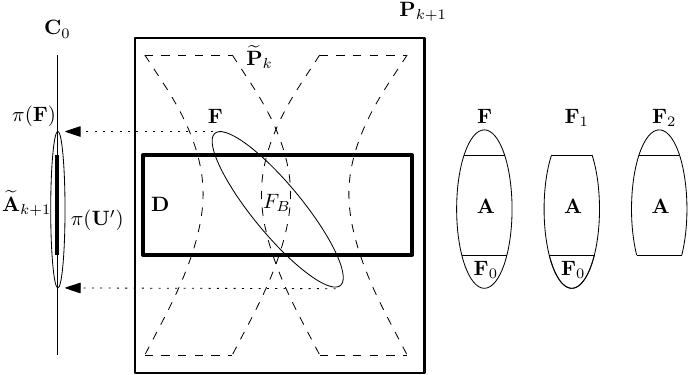}
\caption{The completion of $\str{F}$ by a strong amalgamation over a $\mathcal U$-closed irreducible substructure.}
\label{fig:forbpicture2}
\end{figure}

Consider the projection $\pi$ from $\str{F}$ to $\str{C}_0$ (which is an $\mathcal U$-homo{\-}morphism-embedding). Clearly the $\mathcal U$-size of $\str{F}$ is greater that or equal to the $\mathcal U$-size of $\pi(\str{F})$. First assume that the $\mathcal U$-size of $\pi(\str{F})$ is at most $j$. In this case, by the assumptions on $\str{C}_0$, there exists
 a structure $\str{F}'\in \K$ which is a  $\mathcal U$-completion of $\pi(\str{F})$.
 It follows that $\str{F}'\in \K$ is also a $\mathcal U$-completion of $\str{F}$.
 In the following we thus assume that both the $\mathcal U$-size of $\pi(\str{F})$ and the $\mathcal U$-size of $\str{F}$ are $j+1$.  

 Because $\str{D}_{k+1}$ is an $\widetilde{\str{A}}_{k+1}$-partite system and thus it has a projection to $\widetilde{\str{A}}_{k+1}$, we also know that $\str{A}\in \K$ is a $\mathcal U$-completion outside of $\str{D}_{k+1}$. In fact, by~\ref{3:i} it follows that $\str{F}$ contains vertices outside of $\str D_{k+1}$ which are from two different copies of $\str{P}_k$. Denote by $\widetilde{\str{P}}_k$ a copy
of $\str{P}_k$ which contains a vertex of $\str{F}$ with no projection to $\widetilde{\str{A}}_{k+1}$.
Denote by $F_B$ the set  $\widetilde{P}_k\cap F\cap D_{k+1}$. Because $\str{F}$ is connected, we know that the vertices of $F_B$ forms a vertex cut of $\str{F}$.
It is also easy to see that this vertex cut is $\mathcal U$-closed in $\str{F}$. (See schematic Figure~\ref{fig:forbpicture2}.)

 Denote by $\str{F}_0$ a connected component of $\str{F}$ after deleting the cut $F_B$. Denote by $\str{F}_1$ the structure induced by $\str{F}$ on vertices $F_0\cup F_B$ and $\str{F}_2$ the structure induced by $\str{F}$ on vertices $F\setminus F_0$. Both these structures are nonempty and $\mathcal U$-closed in $\str{F}$. Clearly $\str{F}$ is a free amalgamation of $\str{F}_1$ and $\str{F}_2$ over (the structure induced by $\str F$ on) $F_B$.
Denote by $\str{F}'_1$ the $\mathcal U$-closure of $\pi(\str{F}_1)$ in $\str{C}_{0}$ and by $\str{F}'_2$ the $\mathcal U$-closure of $\pi(\str{F}_2)$ in $\str{C}_{0}$. It follows that the $\mathcal U$-size of $\str{F}_1'$ and $\str{F}_2'$ is at most $j$. By the assumptions on $\mathcal U$-completions in $\str{C}_0$ it follows that there is $\str{F}''_1\in \K$ which is a $\mathcal U$-completion of $\str{F}'_1$ and $\str{F}''_2\in \K$ which is a $\mathcal U$-completion of $\str{F}_2'$.  
The strong amalgamation of $\str{F}''_1$ and $\str{F}''_2$ over the $\mathcal U$-closure of $\pi(F_B)$ in $\str{C}_0$ (which is a substructure of $\widetilde{\str{A}}_{k+1}$, and because $\widetilde{\str{A}}_{k+1}$ is irreducible, it must remain unchanged in both $\str{F}''_1$ and $\str{F}''_2$) is then a completion of $\str{F}$ in $\K$, which is a contradiction with $\str{F}$ having no $(\K,\mathcal U)$-completion. This finishes the proof of~\ref{3:i}.
\end{enumerate}
\end{proof}

\begin{lemma}
\label{lem:iteratedpartite2}
Let $L$ be a relational language, let $\mathcal U$ be a closure description in $L$, let $\mathcal K$ be a strong amalgamation class of finite irreducible $\mathcal U$-closed
$L$-structures which is hereditary for $\mathcal U$-closed substructures.
Let $\str{A},\str{B}\in \mathcal K$, let $n\geq 1$ be an integer,
and let $\str{C}_0$ be a finite $\mathcal U$-closed $L$-structure such that $$\str{C}_0 \longrightarrow (\str{B})^{\str{A}}_2.$$

Then there exists a finite $\mathcal U$-closed $L$-structure $\str{C}$ with a $\mathcal U$-homomorphism-embedding $\str{C}\to \str{C}_0$
such that
$$\str{C} \longrightarrow (\str{B})^{\str{A}}_2.$$
Moreover, every $\mathcal U$-substructure of $\str{C}$ with at most $n$ vertices has a $(\K,\mathcal U)$-completion
and every $\mathcal U$-irreducible substructure of $\str{C}$ is isomorphic to a substructure of $\str{B}$. 
\end{lemma}

\begin{proof}[Proof]
By a repeated application of Lemma~\ref{lem:iteratedpartitestep2}, starting from $\str{C}_0$,
we construct a sequence of $\mathcal U$-closed $L$-structures 
$\str C_0, \str{C}_1, \str{C}_2, \str{C}_3, \ldots, \str{C}_n$ such that:
\begin{enumerate}[label=(\roman*)]
\item $\str{C}_j \longrightarrow (\str{B})^{\str{A}}_2$ for every $0\leq j\leq n$,
\item there is a $\mathcal U$-homomorphism-embedding $\str{C}_j\to \str{C}_{j-1}$ for every $1\leq j\leq n$,
\item every $\mathcal U$-closed substructure of $\str{C}_j$ (for every $1\leq j\leq n$) of $\mathcal U$-size at most $j$ has a $(\K,\mathcal U)$-completion, and
\item every $\mathcal U$-irreducible substructure of $\str{C}_j$ (for every $1\leq j\leq n$) is contained in a copy of $\str{B}$.
\end{enumerate}

The statement of Lemma~\ref{lem:iteratedpartite2} then follows by putting $\str{C}=\str{C}_n$.
\end{proof}

\subsection{Proofs of Theorems~\ref{thm:mainstrong} and~\ref{thm:mainstrongclosures}}
\label{sec:mainres}

After all the preparations we are ready to complete the proofs of Theorems~\ref{thm:mainstrong} and~\ref{thm:mainstrongclosures}.
Combining the results of previous sections this takes the following easy form.
\begin{proof}[Proof of Theorem~\ref{thm:mainstrong}]
Given a language $L$, a class $\K$ of $L$-structures, $\str{A},\str{B}\in \K$, we use the Ramsey property of $\mathcal R$ (recall that $\mathcal K\subseteq \mathcal R$)
to obtain $\str{C}_0\in \mathcal R$ such that: $$\str{C}_0\longrightarrow (\str{B})^\str{A}_2.$$
Because $\K$ is locally finite in $\mathcal R$, we fix $n=n(\str{C}_0)$.

By an application of Propostion~\ref{prop:closures}, we obtain a relational language $L'$,
a closure description $\mathcal U$ and relational $L'$-structures $\str{A}'=\Rel(\str{A})$, $\str{B}'=\Rel(\str{B})$ and $\str{C}_0'=\Rel(\str{C}_0)$ and a class $\K'=\Rel(\K)$.
We have
$$\str{C}'_0\longrightarrow (\str{B}')^{\str{A}'}_2.$$

Next we apply Lemma~\ref{lem:iteratedpartite2} to obtain $\str{C}'$ such that
$$\str{C}'\longrightarrow (\str{B}')^{\str{A}'}_2$$
and which has a $(\K',\mathcal U)$-completion.

Now it remains to let $\str{C}$ to be a $\K$-completion of an $L$-structure corresponding to $\str{C}'$ in $\K$ (\ie{}\ structure $\Rel^{-1}(\str C'))$.
\end{proof}

\begin{proof}[Proof of Theorem~\ref{thm:mainstrongclosures}]
Thanks to Proposition~\ref{prop:closures}, we can assume that $L$ is a relational language with closure description $\mathcal U$.
Given $\str{A},\str{B}\in \K$ we use the Ramsey property of $\mathcal R$
to obtain $\str{C}_0\in \mathcal R$ such that:
$$\str{C}_0\longrightarrow (\str{B})^\str{A}_2.$$
Now fix $n=n(\str{B},\str{C}_0)$.
Next apply Lemma~\ref{lem:closures} to obtain a $\mathcal U$-closed $\str{C}_1$ such that
$$\str{C}_1\longrightarrow (\str{B})^\str{A}_2$$
and a $\mathcal U$-homomorphism-embedding $\str C\to\str{C}_0$.
Finally, apply Lemma~\ref{lem:iteratedpartite2} to obtain $\str{C}$ satisfying the following:
\begin{enumerate}[itemsep=0em]
\item $\str{C}\longrightarrow (\str{B})^\str{A}_2,$
\item there is a $\mathcal U$-homomorphism-embedding $\str C\to\str{C}_1$, and
\item every substructure of $\str{C}$ on at most $n$ vertices
has a $\K$-completion.
\end{enumerate}
We have verified the assumptions of the locally finite completion property (Definition~\ref{def:multiamalgamation}) for $\str{C}$. It follows that there is $\str{C}'\in \K$ which is a completion of $\str{C}$ with respect to copies of $\str{B}$. We thus obtained $\str{C}'\in \K$ such that:
$$\str{C}'\longrightarrow (\str{B})^\str{A}_2.$$
\end{proof}

\section{Construction of Ramsey lifts}
\label{sec:lifts}
In this section we focus on techniques for
lifting a class into a strong amalgamation class where we can apply Theorem~\ref{thm:mainstrong}.
We thus provide general tools for construction of Ramsey lifts.

\subsection{Ramsey classes and ultrahomogeneous structures}
\label{sec:fraisse}
Let $\K$ be a class of finite and/or countably infinite $L$-structures.  We say that a structure $\str{U}$ is \emph{embedding-universal} (or shortly \emph{universal}) for $\K$
if every structure in $\K$ embeds to $\str{U}$. We say that a \emph{class $\K$ contains an universal structure} if there exists structure $\str{U}\in \K$ which
is universal for $\K$.
One possible way of constructing universal objects is by iterated
amalgamations of finite objects, thereby obtaining the so-called \Fraisse{} limits. These are in fact \emph{$\K$-generic}: For a class $\K$, we say that an object $\str{H}$ is
$\K$-generic if it is both universal for $\K$ and it is \emph{ultrahomogeneous}, \ie{}, every isomorphism
$\varphi$ of two finite substructures $\str{A}$ and $\str{B}$ of
$\str{H}$ can be extended to an automorphism of $\str{H}$.  The notion of
ultrahomogeneous structure is one of the key notions of modern model theory and
it is the source of the well known classification programme of ultrahomogeneous structures~\cite{Lachlan1980,Cherlin1998,Cherlin2013}.

Recall that a structure $\str{A}$ is \emph{locally finite} if and only if the
$\str{A}$-closure of every finite subset of $A$ is finite. We focus on locally finite structures only. In
this context, (locally finite) ultrahomogeneous structures are characterised by the properties
of their finite substructures.  For a structure $\str{A}$, denote by
$\Age(\str{A})$ the class of all finite structures which embed to $\str{A}$.
For a class $\K$ of relational structures, we denote by $\Age(\K)$ the class
$\bigcup_{\str{A}\in \K} \Age(\str{A})$.

The following is one of the cornerstones of model theory.

\begin{theorem}[\Fraisse{}~\cite{Fraisse1986}, see \eg{}\ \cite{Hodges1993}]
\label{fraissethm}
Let $L$ be language and let $\K$ be a class of finite $L$-structures with only countably many non-isomorphic structures.

\begin{enumerate}[label=$(\alph*)$]
\item\label{fraisse:a} The class $\K$ is the age of a countable
locally finite ultrahomogeneous structure $\str{H}$ if and only if $\K$ is an amalgamation
class.
\item If the conditions of \ref{fraisse:a} are satisfied then the structure $\str{H}$ is
unique up to isomorphism. 
\end{enumerate}
\end{theorem}

Recall that a structure $\str{A}$ is \emph{$\omega$-categorical} if for every $n$ the
automorphism group of $\str{A}$ has only finitely many orbits on $n$-tuples. 
Ultrahomogeneous and $\omega$-categorical classes are
closely related to classes with Ramsey lifts as shown by the following 
proposition which exemplifies relevance of these model-theoretic
notions for Ramsey theory. 
\begin{prop}[\cite{Nevsetvril1989a}, see also \cite{Nevsetril2005,Kechris2005} for the general statement]
\label{prop:ramseyhomo}
Let $\mathcal K$ be a hereditary Ramsey class with the joint embedding property.
Then $\mathcal K$ is an amalgamation class.
\end{prop}
This (by now) easy observation, which was discovered in
order to characterise Ramsey classes of graphs, provided a link between the combinatorics of Ramsey
classes  and their model-theoretic properties. The link proved to
be vital and a decade later it led to the characterisation programme
of Ramsey classes~\cite{Nevsetril2005}  and to an important connection with
topological dynamics~\cite{Kechris2005}.

\subsection{Ramsey lifts and the Ramsey classification programme}
Ages of most ultrahomogeneous structures are not Ramsey for trivial reasons (most
frequently simply because they are not rigid enough) and one needs to add some extra information (such as a linear order) in order to make them Ramsey.

Let $L^+=L^+_\mathcal R\cup L^+_\mathcal F$ be a language containing language $L=L_\mathcal R\cup L_\mathcal F$. By this we mean $L_\mathcal R\subseteq L_\mathcal R^+$ and $L_\mathcal F\subseteq L^+_\mathcal F$ and the arities of the relations and functions which belong to both $L$ and $L^+$ are the same.
For every structure $\str{X}\in \Str(L^+)$ there is a unique structure $\str{A}\in \Str(L)$ satisfying $A=X$, $\rel{A}{}=\rel{X}{}$ for every $\rel{}{}\in L_\mathcal R$ and $\func{A}{}=\func{X}{}$ for every $\func{}{}\in L_\mathcal F$.
 We call $\str{X}$ a \emph{lift} of $\str{A}$ and $\str{A}$ is called the \emph{shadow} of $\str{X}$. With this notation, $\Str(L^+)$ is the class of all lifts of structures in $\Str(L)$, and conversely, $\Str(L)$ is the class of all shadows of structures from $\Str(L^+)$.
 Note that a lift is often called \emph{expansion} in the model-theoretic setting and a shadow is often called \emph{reduct}.
(Our terminology is motivated by the computer science context, see~\cite{Kun2008}, and for our purposes we find it both intuitive and natural.)
For a lift $\str X$ we denote by $\sh(\str{X})$  its shadow. ($\sh$
is a \emph{forgetful functor}.) Similarly, for a class $\K^+$
of lifted objects we denote by $\sh(\K^+)$ the class of all shadows
of structures in $\K^+$ (assuming the language $L^+$ of lifts is specified).
On the other hand for a class $\K$ of structures we often denote by $\K^+$ the class of lifted structures.

Given the large list of known ultrahomogeneous and $\omega$-categorical structures
(identified by the classification programme of ultrahomogeneous structures~\cite{Lachlan1980,Cherlin1998,Cherlin2013}) it is natural to ask if all those structures have Ramsey lifts.

The Ramsey classification programme~\cite{Nevsetril2005,Hubicka2005a} has been
completed for all ultrahomogeneous graphs~\cite{Nevsetvril1989a} and
digraphs~\cite{Jasinski2013}.  Motivated by this line of research, Cherlin also
recently extended the classification programme of ultrahomogeneous structures by the list of all
ordered graphs~\cite{Cherlin2013} which, in turn, also all lead to Ramsey
lifts.  This paper can be seen as a contribution to the Ramsey classification programme.

It is easy to see that every class $\K$ has a Ramsey lift. (For example, we may extend the
language by infinitely many unary relations and assign every vertex of every
structure in $\K$ a unique unary relation. Such a lift trivially prevents
any embeddings and the Ramsey statement becomes vacuously true.)
This is why we focus on Ramsey lifts using only finitely many additional relations
(where possible) or, more generally, on precompact lifts. This leads to the following definitions (introduced by Nguyen Van Th{\'e}, see~\cite{The2013}).
\begin{definition}
\label{defn:precompact}
Let $\mathcal K^+$ be a lift of $\K$.
We say that $\mathcal K^+$ is a \emph{precompact lift of $\mathcal K$} if for
every structure $\str{A} \in \mathcal K$ there are only finitely many
structures $\str{A}^+ \in \mathcal K^+$ such that $\str{A}^+$ is a lift of
$\str{A}$ (\ie{}\ $\sh(\str{A}^+)$ is isomorphic to $\str{A}$). 
\end{definition}
In the Ramsey setting the following is a natural property (called the \emph{expansion property} in~\cite{The2013}).
\begin{definition}
\label{defn:ordering}
Let $\mathcal K^+$ be a lift of $\K$. For $\str{A},\str{B}\in \K$ we say
that $\str{B}$ has the \emph{lift property} for $\str{A}$ if for every lift
$\str{B}^+\in \mathcal K^+$ of $\str{B}$ and for every lift $\str{A}^+\in \mathcal K^+$ of $\str{A}$ there is an embedding $\str A^+\to\str{B}^+$.

$\mathcal K^+$ has the \emph{lift property} with respect to $\K$ if for every $\str{A}\in \K$
there is $\str{B}\in \K$ with the lift property for $\str{A}$.
\end{definition}
In the special case where the lift adds only the order the lift property is
also called the \emph{ordering property} (which is one of the classical Ramsey theory definitions~\cite{Leeb,Nevsetvril1976}).

Lifts with the lift property are used to compute Ramsey degrees and universal
minimal flows \cite{Kechris2005}. Moreover, it can be shown that every class
has at most one precompact Ramsey lift up to bi-definability. Ramsey lifts with
the lift property can thus be considered to be the minimal lifts (see \eg{}\ \cite{The2013}).

In the Ramsey setting it is natural to work with classes that are not strong
amalgamation classes of ordered structures themselves but can be turned into
one by means of a precompact lift. 
A good candidate for a class with a precompact Ramsey lift is the age of an
$\omega$-categorical structure: every  $\omega$-categorical structure can
be turned to homogenous one by an appropriate precompact lift. 
This process is called the \emph{standard homogenisation}~\cite{Covington1990} and the corresponding lift is called the \emph{homogenising lift}.

More precisely, this is established as follows:
Given an age $\K$ of an $\omega$-categorical structure $\str{U}$, the homogenising
lift $\K^+$ can always be constructed by, for every $n\geq 1$, considering the
automorphism group of $\str{U}$ and adding lifted relations of arity $n$
denoting the individual orbits of $n$-tuples. The lift $\K^+$
is then the age of the ultrahomogeneous structure $\str{U}^+$ created this way.
However such a general description is rarely useful
for obtaining the Ramsey property.
We will focus on classes defined by forbidden homomorphism-embeddings because
these, when homogenised, turn into strong amalgamation classes which are at the heart
of our Ramsey argument.
First, we give an explicit homogenisation of these classes.
 This is done in a fully constructive way which
leads to an explicit description of Ramsey lifts (and therefore also to a practical way of
computing the Ramsey degrees and universal minimal flows). 

\subsection{Lifts of $\Forb(\F)$ with strong amalgamation}
Consider graphs (seen as relational structures in a language $L$ consisting of a single relation $E$) and the class $\K$ of all finite graphs
not containing $\str{C}_5$, the graph cycle with 5 vertices, and $\str{C}_3$, the graph cycle with 3 vertices, as (non-induced) subgraphs. Equivalently, this is the class of all finite graphs having no homomorphism-embedding from $\str{C}_5$. $\K$
is not an amalgamation class. However, it can be turned
to one by adding two binary relations~\cite{Komjath1988,Komjath1999,Cherlin1999}. This is done by considering language $L^+$ which extends $L$ by two binary relations $E_2$ and $E_3$.
For every graph $\str{A}\in \K$ one can construct its lift $\str{A}^+$ by putting $(u,v)\in E_2$ if and only if they are in distance 2 in $\str{A}$ and $(u,v)\in E_3$ if and only if they are in distance 3 or more. It can be easily checked that the class $\K^+$ of all substructures of such lifts forms an amalgamation class. (In fact, it is the class of all finite metric spaces with distances 1, 2 and 3 omitting triangles 1--1--1 and 1--2--2).
Similar results hold in general for classes given by forbidden homomorphism-embeddings.

Let $\F$ be a family of finite structures.  By $\Forb(\F)$ we denote the class
of all finite or countable structures $\str{A}$ such that there is no
homomorphism-embedding from any $\str{F}\in \F$ to $\str{A}$.  Analogously, by
$\Forbh(\F)$, $\Forbi(\F)$ and $\Forbm(\F)$ we shall denote the class of all finite or countable structures
$\str{A}$ such that there is no homomorphism, embedding and monomorphism from any $\str{F}\in \F$ to
$\str{A}$ respectively.

Generalising~\cite{Hubicka2013, Hubicka2009} we now give a way to turn every class $\Age(\Forb(\F))$
 into a lifted class $\Lifts_\F$ which has strong
amalgamation (and thus leads to a homogenisation of~$\Forb(\F)$ and in turn to a Ramsey class).
Here, $\Age(\Forb(\F))$ means the class consisting of all finite structures from $\Forb(\F)$.

\subsubsection{Pieces of structures}
Recall the notations of Gaifman graph, connected structure and cut introduced in Section~\ref{sec:perliminaries}.
In particular, cuts of an $L$-structure $\str{A}$ are vertex cuts of the Gaifman graph $\str{G}_\str{A}$ which are closed in $\str{A}$.
In this section we define a notion of piece which will be the basic building stone of our homogenizing lift.
First we review some standard graph-theoretic notions.

A~\emph{(connected) component} of $\str{A}$ with cut $R$ is any subset $C\subset A$ that is
a connected component of the graph created from $\str{G}_\str{A}$ by removing $R$.
Given a structure $\str{A}$ with cut $R$ and two subsets
$A_1$ and $A_2$ of $A$, we say that $R$ \emph{separates}
$A_1$ and $A_2$ if there are components $A'_1\neq
A'_2$ of $\str{A}$ with cut $R$ such that $A_1\subseteq A'_1$ 
and $A_2\subseteq A'_2$.
Given structure $\str{A}$ and a set $S\subseteq A$ of its vertices, the
\emph{neighbourhood of $S$ in $\str{A}$}, denoted by $N_\str{A}(S)$, is
the set of all vertices in $A\setminus S$ connected to a vertex $S$ by an edge in the
Gaifman graph of $\str{A}$.

We will make use of the following simple (geometrical) observation about the neighbourhoods and
components in structures.

\begin{observation}
\label{cuty}
Let $A_1$ be a component of a connected $L$-structure $\str{A}$ with a cut $R$. Then the neighbourhood $N_\str{A}(A_1)$ is a subset of $R$.

Moreover, $\Cl_\str{A}(N_\str{A}(A_1))$ is a cut contained in $R$, $A_1$ is one of the components of $\str{A}$ with cut $N_\str{A}(A_1)$
and $\str{A}$ induces a substructure on $\Cl_\str{A}(N_\str{A}(A_1))\cup A_1$.
\end{observation}
\begin{proof}Most of this is obvious, it is enough to remember that cuts are always substructures and thus it always holds $\Cl_\str{A}(N_\str{A}(A_1))\subseteq R$.\end{proof}

This correspondence between neighbourhoods and cuts lets us define the following notion.

\begin{definition}
\label{def:separating}
Let $R$ be a cut in a structure $\str{A}$. Let $A_1\neq A_2$
be two components of $\str{A}$ with cut $R$.  We call $R$ a \emph{minimal separating cut} for $A_1$ and $A_2$ in $\str{A}$ if
$R=\Cl_\str{A}(N_\str{A}(A_1))=\Cl_\str{A}(N_\str{A}(A_2))$.
\end{definition}
For brevity, we can omit one or both components when speaking about
a minimal separating cut:
We also call a cut $R$ minimal separating for $A_1$
in $\str{A}$ if there exists another component $B$ such that $R$
is minimal separating for $A_1$ and $B$ in $\str{A}$.  A cut
$R$ is minimal separating in $\str{A}$ if there exist components
$B_1$ and $B_2$ such that $R$ is minimal separating for
$B_1$ and $B_2$ in $\str{A}$.

\begin{example}
Observe that every inclusion minimal cut is also minimal separating, but not
vice versa.  An example of a minimal separating cut that is not inclusion minimal
is given in Figure~\ref{fig:minimalseparating}.
\begin{figure}
\centering
\includegraphics{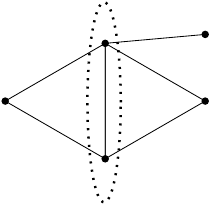}
\caption{A minimal separating cut that is not an inclusion minimal vertex cut.}
\label{fig:minimalseparating}
\end{figure}

\end{example}

The name ``minimal separating cut'' is justified by the following proposition.
\begin{prop}
\label{prop:sep}
Let $R$ be a cut in an $L$-structure $\str{A}$ with components $A_1\neq A_2$.
Then $R$ is a minimal separating cut for $A_1$ and $A_2$ if and only if it is an inclusion minimal cut of $\str{A}$ separating $A_1$ and $A_2$.
\end{prop}
\begin{proof}
Let $R$ be an inclusion minimal cut separating $A_1$ and $A_2$.
Clearly cuts $R_1=\Cl_\str{A}(N_\str{A}(A_1))\subseteq R$ and $R_2=\Cl_\str{A}(N_\str{A}(A_2))\subseteq R$
are both separating $A_1$ and $A_2$. From minimality of $R$ we have $R_1=R_2=R$.

To see the opposite direction, assume that $R$ is a minimal separating cut for $A_1$ and $A_2$, thus
$R=\Cl_\str{A}(N_\str{A}(A_1))=\Cl_\str{A}(N_\str{A}(A_2))$.  Assume, to the contrary, that  there
is a cut $R'\subset R$ which also separates $A_1$ and $A_2$. Denote by $A'_1$ the
component of $\str{A}$ with cut $R'$ containing $A_1$ and by $A'_2$
the component of $\str{A}$ with cut $R'$ containing $A_2$.  If $A'_1=A_1$
and $A'_2=A_2$ then $R'\supseteq
\Cl_\str{A}(N_\str{A}(A_1))=\Cl_\str{A}(N_\str{A}(A_2))=R$ which is a
contradiction.  By symmetry, we can thus assume that there is a vertex $v\in A'_1\setminus A_1$.
Because $v$ needs to be connected to $A_1$ we also know that there is a vertex $v'\in A'_1\cap R$.
 Because $v'\in \Cl_\str{A}(N_\str{A}(A_2))$ and $v'\notin A'_2$ we have that
$v'\in N_\str{A}(A'_2)$, a contradiction with $R'$ being cut separating $A'_1$ and $A'_2$.
\end{proof}

If $R$ is a set of vertices then $\ordclass{R}$ will denote a tuple (of
length $\vert R\vert $) formed by all the elements of $R$. Alternatively, $\ordclass{R}$
is an arbitrary linear ordering of $R$.
A~\emph{rooted $L$-structure} $\Piece$ is a pair $(\str{P},\ordclass{R})$
where $\str{P}$ is an $L$-structure and $\ordclass{R}$ is a
tuple consisting of distinct vertices of $\str{P}$. $\ordclass{R}$ is
called the \emph{root} of $\Piece$.

The following is our basic notion (generalizing \cite{Hubicka2013,Hubicka2009}).

\begin{definition}
\label{defn:piece}
Let $\str{A}$ be a connected $L$-structure and $R$ a minimal separating
cut for a component $A_1$ in $\str{A}$. A~\emph{piece} of an $L$-structure
$\str A$ is then a rooted $L$-structure $\Piece=(\str{P},\ordclass{R})$, where the
tuple $\ordclass{R}$ consists of the vertices of the cut $R$ in a (fixed)
linear order and $\str {P}$ is a structure induced by $\str{A}$ on $A_1\cup R$.
$\vert R\vert $ is called the \emph{width} of $\Piece$.
\end{definition}

Note that every piece is a connected structure.

All pieces are considered as rooted structures: a piece $\Piece$ is a structure $\str{P}$ rooted at $\ordclass{R}$. Accordingly, we say
 that pieces $\Piece_1=(\str{P}_1,\vv*{R}{1})$ and
$\Piece_2=(\str{P}_2,\vv*{R}{2})$ are \emph{isomorphic} if there is a function $\varphi\colon P_1\to P_2$ that is isomorphism of structures $\str{P}_1$ and $\str{P}_2$ and $\varphi$ restricted to $\vv*{R}{1}$ is a monotone bijection between $\vv*{R}{1}$ and $\vv*{R}{2}$ (we denote this as
$\varphi(\vv*{R}{1})=\vv*{R}{2}$).

\begin{figure}
\centering
\includegraphics{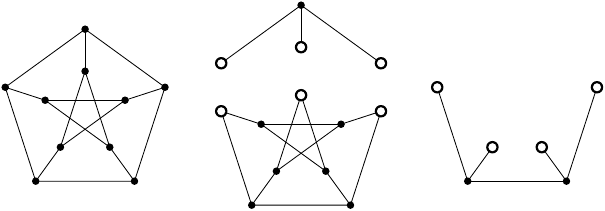}
\caption{Pieces of the Petersen graph up to isomorphisms (and up to a permutations of roots) with white vertices denoting the roots.
Observe that the complement of the last piece is an amalgamation of two identical pieces as depicted in Figure~\ref{Petersoni2}.}
\label{Petersoni}
\end{figure}
\begin{figure}
\centering
\includegraphics{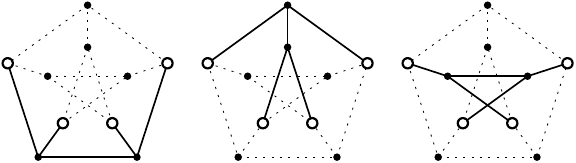}
\caption{Decomposition of the Petersen graph into 3 pieces using one cut with 4 vertices. Observe that all 3 pieces are isomorphic except for the ordering of the roots.}
\label{Petersoni2}
\end{figure}
\begin{example}
Figure~\ref{Petersoni} shows all isomorphism types of pieces of the Petersen graph (up to a permutation of roots).
\end{example}

\subsubsection{Regular families of structures}
\label{sec:regular}
Let $\F$ be a set of connected finite $L$-structures. For our construction of a universal structure for $\Forb(\F)$ we use special
lifts, called $\F$-lifts (introduced in~\cite{Hubicka2013,Hubicka2009} which also contain several examples).

Given rooted $L$-structures $(\str{P},\vv{R})$ and
$(\str{P}',\vv{R}')$ such that $\vert R\vert =\vert R'\vert $, denote by
$(\str{P},\vv{R})\oplus (\str{P}',\vv{R}')$ the
(possibly rooted) $L$-structure created as the free amalgamation of $\str{P}$ and $\str{P}'$
with the corresponding roots being identified (in the order of $\vv{R}$ and $\vv{R}'$).
Note that $(\str{P},\vv{R})\oplus (\str{P}',\vv{R}')$ is defined only
if the rooted structure induced by $\str{P}$ on $\vv{R}$ is isomorphic to the rooted
structure induced by $\str{P}'$ on $\vv{R}'$.

\begin{definition}
\label{def:incompatible}
A piece $\Piece=(\str{P},\vv{R})$ is \emph{incompatible} with
a rooted structure $\APiece$ if $\Piece \oplus \APiece$ is defined and there exists $\str{F}\in \F$
that is isomorphic to $\Piece \oplus \APiece$.  
(In  other words, there exists $\str{F}'$ isomorphic to some $\str{F}''\in \F$, such that $\Piece$ is a piece of $\str{F}'$ and $\APiece$ is a structure induced by $\str{F'}$ on $F'\setminus (P\setminus R)$
 rooted by $\vv{R}$.) 
\end{definition}
Thus $\Piece$ is incompatible with $\APiece$ if no amalgamation over their roots belongs to $\Forb(\F)$.
\begin{example}
As depicted in Figures~\ref{Petersoni} and \ref{Petersoni2}, for $\mathcal F$ being a Petersen graph there is up to isomorphism and permutation of roots a unique
piece $\Piece$ with 4 roots.  It is compatible with every other piece but incompatible with the structure $\APiece$ which is isomorphic to
the amalgamation of pieces $\Piece'$ and $\Piece''$ which are isomorphic to $\Piece$ except for the order of root vertices.
For this reason it is necessary to consider rooted structures rather than pieces in Definition~\ref{def:incompatible}.
\end{example}

Assign to each piece $\Piece$ a set $\Incompatible_\Piece$ of all
rooted structures that are incompatible with $\Piece$.  For two pieces
$\Piece_1$ and $\Piece_2$, we put $\Piece_1 \mathop{\sim_\F} \Piece_2$ if and only if
$\Incompatible_{\Piece_1}=\Incompatible_{\Piece_2}$. ($\mathop{\sim_\F}$ is called the \emph{piece equivalence}.)
Observe that every equivalence class of $\sim_\F$ contains pieces of the same
width $n$. We also call $n$ the \emph{width} of the equivalence class of $\sim_\F$.

\begin{definition}
\label{def:regular}
A family of finite structures $\F$ is  \emph{regular} if for every $n\geq 1$ the equivalence $\sim_\F$ has only
finitely many equivalence classes of width $n$.
  \end{definition}
\begin{remark}
The notion of regular
family of structures is a generalisation of the notion of a regular family of trees, introduced in
\cite{Erdos2012} and it is motivated by the similarity
to the characterisation of regular languages by the Myhill--Nerode Theorem.
Definition~\ref{def:regular} is a strengthening of the definition used in~\cite{Hubicka2013}
for classes without a bound on the size of the cut.
\end{remark}

\subsubsection{Maximal $\F$-lifts}
\label{sec:Flifts}
Now we are ready to explain the homogenising lift of the class $\Forb(\F)$. We
denote the language of the structures by $L$.

Fix an enumeration $\PieceEq_\F^1, \PieceEq_\F^2,\ldots$ of the equivalence
classes of all pieces with respect to $\sim_\F$ (the piece equivalence corresponding to $\F$).  If there are only finitely many equivalence
classes in $\sim_\F$, put $I=\{1,2,\ldots, N\}$, where $N$ denotes the number
of equivalence classes of $\sim_\F$. Otherwise put $I=\{1,2,\ldots\}$.  

The language $L^+$ extends $L$ by new relations $\ext{}{i}$, $i\in I$. The arity of $\ext{}{i}$ corresponds
to the width of $\PieceEq_\F^i$. (To make the distinction between languages more explicit, we use $\ext{}{i}$ to denote the lifted relations instead of $\rel{}{i}$.)
An \emph{$\F$-lift} $\str{X}$ of an $L$-structure $\str{A}$ is an $L^+$-structure $\str{X}$ such that $X=A$, $\rel{X}{}=\rel{A}{}$ for every $\rel{}{}\in L$, $\func{X}{}=\func{A}{}$ for every $\func{}{}\in L$ and with additional relations $\ext{X}{i}$, $i\in I$. Abusing notation, we will also write it briefly as:
$$\str{X}=(\str{A},(\ext{X}{i};i\in I)).$$
For an $L$-structure $\str A$, we define the \emph{canonical $\F$-lift}
of $\str A$ as follows:
$$L_\F({\str A})=(\str{A},(\extl{L_\F(\str{A})}{i}; i\in I))$$
by putting $(v_1,v_2,\ldots,v_l)\in
\extl{L_\F(\str{A})}{i}$ if and only if $v_i\neq v_j$ for every $1\leq i<j\leq l$ and there is a piece
$\Piece=(\str{P},\vv{R})\in \PieceEq_\F^i$ of width $\l$ and a homomorphism-embedding $f\colon \str{P}\to\str{A}$ such that
$f(\vv{R})=(v_1,v_2,\ldots,v_l)$ (thus, in particular, $f$ is injective on the vertices of $\vv{R}$).
We also say that the canonical $\F$-lift is induced on $\str{A}$ by $L_\F(\str{A})$.

We will use the following notion:
\begin{definition}
\label{defn:maximal}
The canonical $\F$-lift $L_\F({\str A})$ of $\str A$ is \emph{maximal} on $B\subseteq A$ if for every  $\str{C}\in \Forb(\F)$
such that $\str{C}$ contains $\str{A}$ as substructure, the $\F$-lift induced on $B$ by $L_\F(\str{A})$ is
the same as the $\F$-lift induced on $B$ by $L_\F(\str{C})$.
We say that an $\F$-lift $\str{X}$ is \emph{maximal} if there exists $\str{A}\in
\Forb(\F)$ such that $\str{X}$ is induced on $X$ by $L_\F(\str{A})$ and the
canonical $\F$-lift $L_\F(\str{A})$ of $\str A$ is maximal on $X$.
\end{definition}
Intuitively, a maximal $\F$-lift contains all possible relations from all extensions. Because the extensions are not
always compatible with each other, a maximal $\F$-lift is not necessarily unique. 
\begin{example}
\label{example:5cycle}
Consider $\mathcal F_5=\{\str{C}_5\}$. The graph $5$-cycle $\str{C}_5$ has up
to isomorphism two pieces: a path of length two and a path of length three both
rooted in the endpoints.  Each path forms its own equivalence class.  The
language $L^+$ will thus extend the language of graphs by two binary relations which
we denote by $\ext{}{2}$ and $\ext{}{3}$.  The canonical lift $L_{\mathcal F_5}(\str{B})$ of a graph $\str{B}$ adds a pair of vertices $(x,y)$ to $\ext{}{2}$
if and only if $x$ and $y$ are in distance two and to $\ext{}{3}$ if and only if
$x$ and $y$ are in distance three. 

Maximal $\F_5$-lift of $\str{B}$ can be constructed by extending $\str{B}$ to
$\str{A}$ by adding a new path of length three connecting every pair of vertices
which is not in distance one, two or three. It can be easily checked that $\str{A}\in
\Forb(\F)$ and that $L(\str{A})$ induces a maximal lift $\str{X}$ of $B$: for
every canonical lift of a structure in $\Forb(\F_5)$ it holds that every pair of
vertices is in at most one of relations $E$, $\ext{}{2}$ or $\ext{}{3}$ and in
$\str{X}$ this holds for every pair. Note also that ${L_{\F_5}}(\str{A})$ may not be maximal
for $\str{A}$ because some pairs of newly introduced vertices may be in distance greater than 3.

To see that maximal lifts are not necessarily unique consider the graph consisting of two vertices and no edges.
The procedure described above will connect them by relation $\ext{}{3}$, however it is also possible
to connect them by relation $\ext{}{2}$.
\end{example}
\begin{example}
Consider $\mathcal F_\infty$ to be the family of all odd graph cycles.  Here
the pieces are all paths of length 2 or more rooted in the endpoints.  There
are however only two equivalence classes. One containing all paths of even
length and the other containing all paths of odd length.  The class
$\Forb(\F_\infty)$ is the class of all bipartite graphs and the $\F_\infty$-lift
adds a binary relation $\ext{}{e}$ for even distances and a binary relation $\ext{}{o}$ for
odd distances. In a maximal $\F_\infty$-lift, every pair of vertices will either
be connected by an edge or be in one of $\ext{}{e}$ or $\ext{}{o}$.
\end{example}
Maximal $\F$-lifts form the homogenisation we are looking for. Before stating the main result of this section we recall several notions.

Recall that a structure $\str{A}\in\K$ is \emph{existentially complete in a class $\K$} if
for every structure $\str{B} \in \K$ such that the identity mapping (of $A$)
is an embedding $\str{A}\to\str{B}$, every existential statement $\psi$
which is defined in $\str{A}$ and true in $\str{B}$ is also true in
$\str{A}$.

We say that a homomorphism-embedding $f$ from an $L$-structure $\str{A}$ to an $L$-structure $\str{B}$ is \emph{surjective} if $f(A)=\{f(x); x \in A\}=B$. Homomorphism-embedding $f$ is \emph{tuple-surjective}
if for every $\rel{}{}\in L$ and every $\vv{u}\in \rel{B}{}$ there exists $\vv{v}=(v_1,v_2,\ldots, v_n)\in \rel{A}{}$ such
that $f(\vv{v})=(f(v_1),f(v_2),\ldots, f(v_n))=\vv{u}$ and for every $\func{}{}\in L$ and every $\vv{u}\in \dom(\func{B}{})$ there exists $\vv{v}=(v_1,v_2,\ldots, v_n)\in \dom(\func{A}{})$ such that $f(\vv{v})=(f(v_1),f(v_2),\ldots, f(v_n))=\vv{u}$.

We say that a class $\F$ is \emph{closed for homomorphism-embedding images} if
for every $\str{F}\in \F$, and every tuple-surjective homomorphism-embedding
$f\colon \str{F}\to \str{F}'$, there exists a substructure $\str{F}''$ of $\str{F}'$ such that  $\str{F}''\in \F$.
We shall prove the following result about the existence of homogenisations of classes $\Forb(\F)$:

\begin{theorem}
\label{mainthm}
Let $\F$ be a family of finite connected $L$-structures which is
closed for homomorphism-embedding images.  Denote by $\Lifts_\mathcal F$ the
class of all finite maximal $\F$-lifts.  Then $\Age(\Lifts_\mathcal F)$ is an amalgamation class
with strong amalgamations whose shadows are free amalgamations.
If $\F$ is a regular family, then the $\F$-lift adds only finitely many new relations
of every arity and therefore is precompact.

If $\Lifts_\mathcal F$ is countable, denote by $\str{U}'$ the \Fraisse{} limit of $\Age(\Lifts_\mathcal F)$.
If $\F$ is regular and $L$ is a finite relational language, then the shadow
$\str{U}=\sh(\str{U}')\in\Forb(\F)$ is the $\omega$-categorical existentially complete
structure universal for $\Forb(\F)$.
\end{theorem}

We can also show that the construction is tight. (This may be of independent model-theoretic interest.)
Family $\F$ is \emph{upwards closed} if for every $\str{F}\in \F$ we also have $\str{F}'\in \F$
provided that $\str{F}'$ is connected and there is a homomorphism-embedding $\str{F}\to \str{F}'$.
\begin{theorem}
\label{mainthm2}
Let $L$ be a finite relational language.
Let $\F$ be a upwards closed family of finite connected $L$-structures.
Then the following conditions are equivalent:
\begin{enumerate}[label=$(\alph*)$]
 \item\label{mainthm2:a} $\F$ is a regular family of connected structures.
 \item\label{mainthm2:b} There is a lift $L^+$ which extends $L$ only by finitely many relations of any given arity, no functions, and an ultrahomogeneous $L^+$-structure $\str{U}^+$ such that the shadow $\sh(\str{U}^+)\in \Forb(\F)$ is universal for $\Forb(\F)$.
 \item\label{mainthm2:c} $\Forb(\F)$ contains an $\omega$-categorical universal structure.
\end{enumerate}
\end{theorem}
Observe that the assumption about the language being relational is necessary in Theorem~\ref{mainthm2}
as shown in the following example.
\begin{example}
Consider language $L$ containing one unary function and $\F=\emptyset$.
While $\F$ is regular, there is no $\omega$-categorical universal structure for all structures with
a single unary function because there are infinitely many non-isomorphic vertex closures and thus also infinitely many
orbits of vertices.
\end{example}

Theorems~\ref{mainthm} and~\ref{mainthm2} are proved in Sections~\ref{sec:homogenization} and~\ref{sec:classification}  of this paper.

\subsection{Existence of precompact Ramsey lifts}
\label{sec:NR}
In this section we give a strengthening of the following classical result:
\begin{theorem}[Ne\v set\v ril--R\"odl Theorem~\cite{Nevsetvril1977}]
\label{thm:NRoriginal}
Let $\str{A}$ and $\str{B}$ be ordered hypergraphs, then there exists an ordered hypergraph
$\str{C}$ such that $\str{C}\longrightarrow (\str{B})^\str{A}_2$.

Moreover, if $\str{A}$ and $\str{B}$ do not contain an irreducible hypergraph
$\str{F}$ (as an non-induced sub-hypergraph) then $\str{C}$ may be chosen
with the same property.
\end{theorem}
In this original formulation (see~\cite{Nevsetvril1989}) the theorem speaks of hypergraphs (or set systems)
with additional linear order on vertices. This linear order has no further constraints (it is free) and is treated specially throughout the
proof. In other words, the theorem states that for every family $\mathcal {E}$ of finite irreducible hypergraphs the lift of the class of all
finite hypergraphs in $\Forbm(\mathcal E)$ adding a free linear order on vertices is a Ramsey class.

In this section we first
give a re-formulation of this theorem in the language of
relational structures with a small strengthening stated as Theorem~\ref{thm:NR} below.  Then we proceed with the main result of this section (Theorem~\ref{thm:main}) which strengthens
the Ne\v set\v ril--R\"odl Theorem for classes with forbidden homomorphisms and closures.

In this section, the language of every Ramsey class $\K$ will always contain a binary relation
$\leq$. Most often $\leq$ will (in every structure of the class $\K$) represent a linear
order. However, we will also work with structures where $\leq$ is not a linear
order. To distinguish this we say that a structure $\str{A}$ is \emph{ordered} if the relation
$\leq$ forms a linear order on $A$.  If there is no further restriction on $\leq$ then it
is called a \emph{free ordering}.
We say that a relational $L$-structure $\str{F}$ is \emph{irreducible without order} if the shadow of $\str{F}$ removing the relation $\leq_\str{F}$ is irreducible. 
An $L$-structure $\str{F}$ is an \emph{ordered irreducible} structure if it is both ordered and
irreducible without order.

Now we are ready to formulate Theorem~\ref{thm:NRoriginal} in our language:
\begin{theorem}[Ne\v set\v ril--R\"odl Theorem for relational structures]
\label{thm:NR}
Let $L$  be a relational language containing a binary relation $\leq$
and $\mathcal E$ be a (possibly infinite) family of ordered irreducible $L$-struc\-tures.
Then the class of all finite ordered structures in $\Age(\Forbi(\mathcal E))$ is a Ramsey class.
\end{theorem}
There are two differences compared to the original formulation.
First, structures in class $\mathcal E$ are ordered (we thus do not speak of a lift of the class adding a free order, but rather a constrained relation $\leq$).  This allows us to use
Theorem~\ref{thm:NR}  to show, for example, the Ramsey property of acyclic graphs as shown in~\cite{Nevsetvril1984} (see Corollary~\ref{cor:acyclic}).  Second, we speak of forbidden embeddings (and
thus substructures). However both these strengthening follow by the same proof as presented in~\cite{Nevsetvril1989}.

\medskip

From now on, we again consider languages involving both functions and relations.
The linear order will continue to be special in our results, too. The following
notion captures the properties of structures that can be forbidden as homomorphic images:
\begin{definition}
\label{defn:weakorder}
Let $L$ be a language containing a binary relation $\leq$.
An $L$-structure $\str{F}$ is \emph{weakly ordered} if 
\begin{enumerate}
\item $\leq_\str{F}$ can be completed to linear order (in other words, it forms a reflexive acyclic digraph), and
\item for every pair of distinct vertices $a, b\in F$, either $(a,b)\in \leq_\str{F}$ or $(b,a)\in \leq_\str{F}$
if and only if $a,b$ is contained in an irreducible substructure of $\str{F}^-$ where $\str{F}^-$ is a shadow of $\str{F}$
in the language $L^-=L\setminus \{\leq\}$.
\end{enumerate}
(In other words, $\leq_\str{F}$ is an acyclic orientation of the Gaifman graph of the shadow of $\str{F}$ removing the relation $\leq_\str{F}$.)
\end{definition}
Note that in a weakly ordered structure $\str{F}$ the relation $\leq_\str{F}$ may be neither a partial order nor a
linear order, it is just a (special --- reflexive but otherwise acyclic) digraph. Weakly ordered structures typically arise as free
amalgamations of ordered structures.  In a weakly ordered structure $\str{F}$,
$\leq_\str{F}$ is a linear order if and only if $\str{F}$ is irreducible without order.

Sufficient conditions for the existence of a precompact Ramsey lift can now be formulated
as follows.

\begin{theorem}
\label{thm:main}
Let $L$ be a language containing a binary relation  $\leq$ and let $\mathcal F$ be a (possibly infinite) regular family of finite connected weakly ordered $L$-structures which is closed for homomorphism-embedding images.
Further assume that the class of all finite ordered structures in $\Forb(\mathcal F)$ is a locally finite subclass of the class of all finite ordered $L$-structures.
Then the class $\mathcal K$ of all finite ordered structures in $\Forb(\F)$ has a precompact Ramsey lift.

More specifically:
The class $\mathcal K_\F$ of all ordered maximal $\F$-lifts of structures in $\mathcal K$ is a Ramsey class
and
for every pair of struc\-tures $\str{A}, \str{B}\in\mathcal K_\mathcal F$ there exists a 
structure $\str{C} \in \mathcal K_\mathcal F$  such that
$$
\str{C} \longrightarrow (\str{B})^{\str{A}}_2.
$$
If $L$ is a relational language then the lift $\mathcal K_\mathcal F$ has the lift property
with respect to $\K$.
\end{theorem}
\begin{remark}
The condition on the class of ordered structures in $\Forb(\F)$ being locally finite subclass of the class of all ordered structures can be re-formulated as follows: For every structure
$\str{C}_0$ there exists $n(\str{C}_0)$ such that for every structure $\str{C} \notin \Forb(\F)$ with a homomorphism-embedding to $\str{C}_0$ there
is $\str{F}\in\F$ with at most $n(\str{C}_0)$ vertices and a homomorphism-embedding $\str{F}\to \str{C}$.
\end{remark}

 For relational languages and finite families $\F$ of finite connected relational structures, we arrive to the following characterisation:

\begin{corollary}[Ramsey classes with forbidden homomorphism-embeddings]
\label{cor:forbH}
Let $L$ be a relational language and let $\F$ be a finite family of finite connected $L$-structures.
Then the class $\Age(\Forb(\F))$ has a precompact Ramsey lift with the lift property.
\end{corollary}
\begin{proof}
Expand the language $L$ to $L'$ by adding a binary relation $\leq$ and consider the class $\F'$ consisting of 
all weakly ordered structures $\str{F}$ such that their $L$-shadow is a homomorphic image of a structure in $\F$.
Because $\F$ is finite, one can verify that $\F'$ is also finite and thus a regular family.
In this setting we can apply Theorem~\ref{thm:main}.
\end{proof}

\subsection{Proof of Theorem~\ref{mainthm}}
\label{sec:homogenization}
  In this section we prove the essentially model-theoretic Theorem~\ref{mainthm}
which gives an explicit homogenisation of classes $\Forb(\F)$.
This
extends the construction in~\cite{Hubicka2013} for the case of regular infinite families $\mathcal
F$ and particularly for families without an upper bound on the size of minimal separating cuts (thus completing our
techniques to all classes with a precompact homogenisation). Note also that we use  homomorphism-embeddings, instead of homomorphisms.  Thanks to the use of maximal
lifts we not only simplify the argument of~\cite{Hubicka2013}, but
more importantly, we obtain an existentially complete homogenising lift, which, in
turn, gives the lift property of the resulting Ramsey lift. We also, for the first time in this context, consider functions. 
All in all, this part may be seen as a generalisation of~\cite{Hubicka2009, Hubicka2013, Hartman2014}.

Recall that for every $\str{X}\in \Lifts_\mathcal F$ there exists a structure $\str{A}\in\Forb(\F)$ such that $A\supseteq X$, $\str{X}$ is induced on $X$ by $L_\F(\str{A})$
and $L_\F(\str{A})$ is maximal on $X$. We will call such $\str A$ a \emph{witness} of the fact that $\str{X}$ belongs to $\Lifts_\mathcal F$ and denote it by $W(\str{X})=\str A$. Note that the choice of $W(\str{X})$ is not unique.

Given a piece $\Piece=(\str{P},\vv{R})$ of a structure
$\str{F}$, we call $\Piece'=(\str{P}',\vv{R}')$ a {\em
sub-piece} of $\Piece$ if $\Piece'$ is a piece of $\str{F}$ and $P'\subseteq P$.

The key technical part of our construction (and of the proof of Theorem~\ref{mainthm}) is expressed in the following lemma.

\begin{lemma}
\label{lem:amalgamationstr}
Let $L$ be language and let $\mathcal F$ be a family of connected $L$-structures closed for homo{\-}morphism-embedding images.
Let $\str{A}$ and $\str{B}$ be both witnesses of the fact that $\str{X}\in\Lifts_\F$. Then the free amalgamation of $\str{A}$ and $\str{B}$ over the structure induced on $X$ by both $\str{A}$ and $\str{B}$ is also a witness of $\str{X}\in\Lifts_\F$.
\end{lemma}
The main idea of the proof of Lemma~\ref{lem:amalgamationstr} is to use the maximality of $\str{A}$ and $\str{B}$ on $\str{X}$ to show that if
there is an homomorphism-embedding from $\str{F}\in \mathcal F$ to the free amalgamation of $\str{A}$ and $\str{B}$ over $X$ then there is also
an homomorphism-embedding from $\str{F}'\in \mathcal F$ to both $\str{A}$ and $\str{B}$ obtaining a contradiction with $\str{A},\str{B}\in \Forb(\mathcal F)$. 

Towards this direction we define the following flip operation.
\begin{definition}[flip operation]
\label{defn:flip}
Let $\str{A}$ and $\str{B}$ be both witnesses of the fact that $\str{X}\in\Lifts_\F$, let $\str{D}$ be the free amalgamation of $\str{A}$ and $\str{B}$ over $X$, let
$f$ be a homomorphism-embedding from $\str{F}\in \F$ to $\str{D}$ and let $\Piece$ be a piece of $\str{F}$ whose image under $f$ is in $\str{A}$ such that the image of its root is in $X$. Then the \emph{flip} of the piece $\Piece$ from $\str{A}$ to $\str{B}$ is a structure $\str{F}'$ created from $\str{F}$ by replacing $\Piece$ by $\Piece_2$ which is $\sim_F$ equivalent to $\Piece$, along with a homomorphism $f'\colon\str{F}'\to \str{D}$ such that $\Piece_2$ is mapped to $\str{B}$ and $f'$ agrees with $f$ otherwise.
\end{definition}
The flip is schematically depicted in Figure~\ref{fig:flip}.
It follows directly from definition of $\sim_\mathcal F$ that $\str{F}'$ is isomorphic to a structure in $\mathcal F$.
Before proving Lemma~\ref{lem:amalgamationstr} we first prove the fact that flips are always possible.
\begin{figure}
\centering
\includegraphics{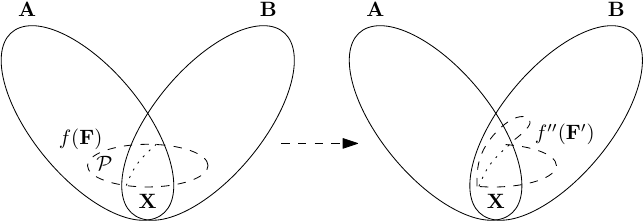}
\caption{The flip operation.}
\label{fig:flip}
\end{figure}
\begin{lemma}
\label{lemma:flip}
Let $\str{A}$, $\str{B}$, $\str{D}$, $\str{F}$, $f$ and $\Piece$ be as in Definition~\ref{defn:flip}. Then there exists $\str{F}'$, $\Piece'$ and $f'$ forming a flip of $\Piece$ from $\str{A}$ to $\str{B}$.
\end{lemma}
\begin{proof}
Denote by $i$ the index such that $\Piece=(\str{P},\vv{R})$ belongs to the equivalence class $\PieceEq_\F^i$ of $\sim_\F$ (see Section~\ref{sec:Flifts}).
Because $f$ is a homomorphism of $\Piece\to \str{A}$ we know that 
$f(\vv{R})\in \ext{X}{i}$ and thus also
$f(\vv{R})\in \ext{L_\F(A)}{i}$ and $f(\vv{R})\in
\ext{L_\F(B)}{i}$. Consequently, there exists a piece $\Piece_2=(\str{P}_2,\vv*{R}{2})$ and a homomorphism-embedding $f_2\colon\str{P}_2 \to \str{B}$
such that $\Piece_2\sim_\F \Piece$ and $f_2(\vv*{R}{2})= f(\vv{R})$. Consider $\str{F}'$
created from $\str{F}$ by replacing $\Piece$ by $\Piece_2$ and a function
$f'\colon F'\to D$ defined as follows:
$$f'(x)=
\begin{cases}
 f_2(x)\hbox{ for }x\in P_2,\\
 f(x)\hbox{ otherwise.}
\end{cases}
$$
Clearly $f'$ is a homomorphism-embedding $\str{F}'\to \str{D}$.
\end{proof}
\begin{proof}[Proof of Lemma~\ref{lem:amalgamationstr}]
\begin{figure}
\centering
\includegraphics{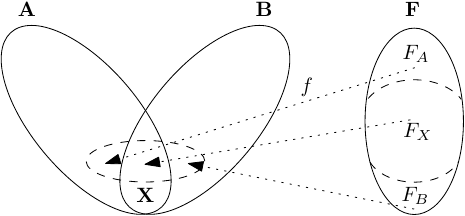}
\caption{An amalgamation of maximal $\F$-lifts.}
\label{fig:liftamal}
\end{figure}
Denote by $\str{D}$ the free amalgamation of $\str{A}$ and $\str{B}$ over $X$.
From the maxi\-mality of $\str{X}$ in both $\str{A}$ and $\str{B}$ we know that $\str{D}$ is a witness of $\str{X}$ if
$\str{D}\in \Forb(\F)$. Assume, to the contrary, that $\str{D}\notin \Forb(\F)$ and thus
there is $\str{F}\in \F$ and a homomorphism-embedding $f\colon\str{F}\to\str{D}$. Because $\F$ is closed for homomorphism-embedding images, we can also assume $f$ to be injective.
Then $f$ partitions the vertex set of $\str{F}$ into three sets defined as follows: $F_X$
are vertices with image in $X$, $F_A$ are vertices with image in
$A\setminus X$ and $F_B$ are vertices with image in $B\setminus X$.
Without loss of generality we can assume that $\str{F}$ and $f$ were chosen so that $\vert F_A\vert $ is minimal. Clearly $\vert F_A\vert \geq 1$. The situation is depicted in Figure~\ref{fig:liftamal}.  Observe that $F_X$ is a cut of $\str{F}$ separating $F_A$ and $F_B$. 

 Denote by $\Piece=(\str{P},\vv{R})$ a piece of $\str F$ with its root contained in $F_X$ which contains a vertex of $F_A$ (such a piece can be obtained by Proposition~\ref{prop:sep}). If $P\subseteq F_A\cup F_X$, by Lemma~\ref{lemma:flip} we can use the flip operation for piece $\Piece$ of $\str{F}$ from $\str{A}$ to $\str{B}$ contradicting the minimality of $F_A$.
We thus conclude: 
\begin{claim}\label{claim:1}
Every piece of $\str{F}$ with root in $F_X$ containing a vertex of $F_A$ must also contain a vertex of $F_B$.
\end{claim}

Choose $\Piece'=(\str{P}',\vv{R}')\in \PieceEq_\F^j$ to be a piece containing vertices of both $F_A$ and $F_B$ with the minimal number of non-root vertices among all pieces with this property.
If $\Piece'$ contains a sub-piece with root in $F_X$ which is contained in $F_X\cup F_B$, we can perform the flip of $\Piece'$ from $\str{B}$ to $\str{A}$.
If this procedure eliminates all vertices of $P'\cap F_B$ we get
a homomorphism-embedding $f'\colon\str{P}'\to \str{A}$, $f'(\vv{R}')=f(\vv{R}')$, and therefore $f(\vv{R}')\in \ext{L_\F(A)}{j}$ which contradicts
the minimality of $\vert F_A\vert $ as in Claim~\ref{claim:1}. It follows that:
\begin{claim}
\label{claim:2}
Every piece $\Piece'$ of $\str{F}$ containing vertices of $F_A$ and minimising number of non-root vertices contains a component $B'$ of $\str{F}$ with cut $F_X$ which cannot be eliminated from $\Piece'$ by a flip operation of a piece of $\str{F}$ contained in $\Piece'$ from $\str{B}$ to $\str{A}$.
Consequently, every piece of $\str{F}$ containing $B'$ contains also some vertices of $F_A$.
\end{claim}

Denote by $A'$ a component of $\str{F}$ with cut $F_X$ contained in $\Piece'$ consisting of vertices of $F_A$ (such a component must exist because $\Piece'$ was chosen to contain a vertex of $F_A$).
In the following we separate $A'$ and $B'$ within $\Piece'$. This cannot be done by a direct application of
Proposition~\ref{prop:sep} because the minimal separating cut that separates $A'$ and $B'$ in $\str{F}$ may 
contain some vertices of $F\setminus P'$.

Denote by $F'$ the set of vertices of any connected component of $\str{F}\setminus P'$ such that $R'\subseteq \Cl_\str{F}(N_\str{F}(F'))$ (such a component exists because $R'$ is a minimal separating cut). 
By an application of Proposition~\ref{prop:sep} on cut $F_X\cap P'$ and components $F'$
and $B'$ one gets that $R'\subseteq \Cl_\str{F}(N_{\str{F}}(B'))$, otherwise one would obtain
a sub-piece which would contradict the minimality of $\Piece'$ or Claim~\ref{claim:2}.

Again using Proposition~\ref{prop:sep} on cut $F_X$ and
components $A'$ and $B'$ we obtain a minimal separating cut $C$. Clearly $R'\subseteq C$ because $R'\subseteq \Cl_\str{F}(N_\str{F}(A'))\cap \Cl_\str{F}(N_\str{F}(B'))$. $C$ must contain some additional vertices of $F_X\cup (P'\setminus R')$ because $\str{P}'\setminus R'$ is connected and $F_X$ separates $A'$ and $B'$. The pieces obtained are thus proper sub-pieces
of $\Piece'$ that either contain both vertices of $F_A$ and $F_B$ or they can be used for the flip operations. 
In all these cases this yields a contradiction.
\end{proof}

\begin{proof}[Proof of Theorem~\ref{mainthm}]
The class $\Lifts_\mathcal F$ of all maximal $\F$-lifts is clearly hereditary, isomorphism closed and has the joint
embedding property. 
Thus to show that $\Lifts_\mathcal F$ is an amalgamation class it remains to
verify that $\Lifts_\mathcal F$ has the amalgamation property.

Consider $\str{X},\str{Y},\str{Z}\in \Lifts_\mathcal F$. Assume that
$\str{Z}$ is a substructure of both $\str{X}$ and
$\str{Y}$ and without loss of generality assume that $X\cap Y=Z$.

Put
\begin{eqnarray*}
\str{A}&=&W(\str{X}),\\
\str{B}&=&W(\str{Y}),\\
\str{C}&=&\sh(\str{Z}).
\end{eqnarray*}

Now consider $\str{D}$, the free amalgamation of $\str{A}$ and $\str{B}$
over $\str{C}$.  As shown by Lemma~\ref{lem:amalgamationstr}, $\str{D}$ is a witness of
$\str{Z}$ and also a witness of $\str{A}$ and $\str{B}$.  Now find
$\str{E}\in \Forb(\F)$ containing $\str{D}$ as a substructure such that
$L_\F(\str{E})$ is maximal on $D$.
It follows that the structure induced on $D$ by $L_\F(\str{E})$ is a strong amalgamation of
$\str{X}$ and $\str{Y}$ over $\str{Z}$.

By the maximality condition it also follows that the shadow of the \Fraisse{} limit of $\Lifts_\F$
is existentially complete in $\Forb(\F)$.
\end{proof}

\subsection{Proof of Theorem~\ref{mainthm2}}
\label{sec:classification}

Theorem~\ref{mainthm2} gives a characterisation of those families $\F$ such that $\Forb(\F)$ contains
an $\omega$-categorical universal structure.
This is related to (and generalises) forbidden homomorphism theorem of~\cite{Cherlin1999}. This is also in contrast with forbidden monomorphisms where the corresponding characterisation is a well known problem conjectured to be undecidable~\cite{Cherlin2011}.
For a family $\F$ of finite connected $L$-structures denote by $\overline{\F}$ the (complementary) class
of all connected $L$-structures not isomorphic to any structure in
$\F$.  First we show that regular families are closed for complements:
\begin{lemma}
\label{doplnek}
Let $L$ be a language.
For every class of finite connected $L$-structures $\F$ it holds that $\F$ is regular if and only if $\overline{\F}$ is regular.
\end{lemma}
\begin{proof}
Clearly it suffices to show only one implication. Assume that $\F$ is regular.
Now consider $\overline{\Piece}$, a piece of some $\overline{\str{F}}\in
\overline{\F}$.  Denote by $\overline{\Incompatible}_{\overline{\Piece}}$ 
the set of all rooted structures incompatible with $\overline{\Piece}$ with
respect to $\overline{\F}$ (see Definition~\ref{def:incompatible}).  There are two cases:
\begin{enumerate}
\item $\overline{\Piece}$ is not isomorphic to any piece $\Piece$ of any structure $\str{F}\in \F$.
In this case for every rooted structure $\APiece$ such that $\APiece\oplus\overline{\Piece}$ is defined we have
that $\APiece\oplus\overline{\Piece}$ is not isomorphic to any structure in $\F$, consequently $\APiece\oplus\overline{\Piece}\in \overline{\F}$ and thus $\APiece \in \overline{I}_{\overline{\Piece}}$.
\item  $\overline{\Piece}$ is isomorphic to some piece $\Piece$ of some $\str{F}\in \F$.
In this case for every rooted structure $\APiece$ such that $\APiece\oplus\overline{\Piece}$ is defined we have that
$\APiece\oplus\Piece$ is isomorphic to some structure in $\F$ if and only if $\APiece\oplus\overline{\Piece}$ is
not isomorphic to any structure in $\overline{\F}$.  It follows that $\APiece\in \Incompatible_\Piece$ if and only if
$\APiece \notin \overline{\Incompatible}_{\overline{\Piece}}$.
\end{enumerate}
We have shown that the sets $\overline{\Incompatible}_{\overline{\Piece}}$ are,
in a certain sense, complements of the sets $\Incompatible_\Piece$ and thus by regularity of $\F$ there
are only finitely many different sets
$\overline{\Incompatible}_{\overline{\Piece}}$ of pieces of $\overline{\F}$ with any given width $n\geq 1$.
It follows that $\overline{\F}$ is regular.
\end{proof}
\begin{proof}[Proof of Theorem~\ref{mainthm2}]
\ref{mainthm2:a}$\implies$\ref{mainthm2:b} follows from Theorem~\ref{mainthm} for the class $\F$.

\ref{mainthm2:b}$\implies$\ref{mainthm2:c} is immediate. The shadow of every ultrahomogeneous structure
with finitely many relations of a given arity is $\omega$-categorical.

To see that \ref{mainthm2:c}$\implies$\ref{mainthm2:a} we first observe that for every $\omega$-categorical
structure $\str{U}$ the family $\mathcal C$ 
consisting of all connected structures in
$\Age(\str{U})$ is a regular family.  Fix $n\geq 1$ and consider two pieces
$\Piece=(\str{P},\vv{R})$ and
$\Piece'=(\str{P},\vv{R})$. Denote by $O_\Piece$ the set of all
orbits of $n$-tuples of the automorphism group of $\str{U}$ such that there
exists a homomorphism-embedding $f\colon\str{P}\to \str{U}$ with the tuple
$f(\str{\vv{R}})$ being in the orbit.  It is easy to see that
$O_\Piece=O_{\Piece'}$ implies $\Piece\sim_\F\Piece'$.  This gives the 
regularity of $\F$.

Consider an upwards closed family $\F$ and such that $\Forb(\F)$ contains an
$\omega$-categorical universal structure $\str{U}$.  It is easy to see that
$\overline{\F}$ is precisely the family of all connected structures in
$\Age(\str{U})$.  Because the family $\overline{\F}$ is regular, the family
$\F$ is regular by Lemma~\ref{doplnek}. 
\end{proof}

\subsection{Proof of Theorem~\ref{thm:main}}
\label{sec:mainresult}
By now this is an easy application of our lift construction together with the proof of Theorem~\ref{thm:mainstrongclosures}.
\begin{proof}[Proof of Theorem~\ref{thm:main}]
By Theorem~\ref{mainthm} we obtain a class $\Lifts_\mathcal F$ which is a lift of $\Forb(\F)$ with strong amalgamation. 
The class $\K_\mathcal F$ is then the subclass of $\Lifts_\mathcal F$ consisting of all maximal
lifts of structures in $\K$.

Given $\str{A}, \str{B}\in \K_\mathcal{F}$, denote by $\overline{\str{B}}$ a maximal
lift of a witness of $\str{B}$ (which is finite, because $\F$ is regular) and by
Theorem~\ref{thm:NR} we obtain $\str{C}'_0$ such that
$$\str{C}'_0\longrightarrow(\overline{\str{B}})^{\str{A}}_2.$$
By an application of Lemma~\ref{lem:closures} obtain an ordered structure $\str{C}_0$ such that
$$\str{C}_0\longrightarrow(\overline{\str{B}})^{\str{A}}_2$$
and moreover we have  a homomorphism-embedding $\str{C}_0\to \str{C}'_0$.

Now by the regularity of $\F$ there exists a finite $\F_0$ such that every structure $\str{A}\in
\Forb(\F_0)$ with a homomorphism-embedding to $\str{C}'_0$ is also in
$\Forb(\F)$.  Denote by $n$ the size of the largest structure in $\F_0$ and
construct $\str{C}_1, \str{C}_2, \ldots, \str{C}_n$ by the repeated application of
Lemma~\ref{lem:iteratedpartitestep2} such that for every $1\leq j\leq n$ the following holds:
\begin{enumerate}
\item $\str{C}_j\longrightarrow (\overline{\str{B}})^{\str{A}}_2$, 
\item $\str{C}_j$ has a homomorphism-embedding to $\str{C}'_0$,
\item every substructure of $\str{C}_j$ with at most $j$ vertices has a completion in $\Lifts_\mathcal F$.
\end{enumerate}

We obtain  $\str{C}_n$ where the shadow of every substructure with at most $n$ vertices has a completion in $\Forb(\F)$.
We conclude that the shadow of $\str{C}_n$ is in $\Forb(\F_0)$ and because there is also a homomorphism-embedding
from $\str{C}_n$ to $\str{C}'_0$ we know that the shadow of $\str{C}_n$ is in $\Forb(\F)$. 

Let $\str{C}$ be a maximal lift of the shadow of $\str{C}_n$ with $\leq_\str{C}$ completed to linear order. Because $\F$ is a family of weakly ordered structures we know that the shadow of $\str{C}$ is in $\Forb(\F)$.
By the maximality of $\str{B}$ in $\overline{\str{B}}$ it follows that every copy of $\str{B}$ which is maximal in a copy of $\overline{\str{B}}$ in $\str{C}_n$ is preserved in $\str{C}$. It follows that
$$\str{C}\longrightarrow (\str{B})^\str{A}_2.$$

The lift property of $\mathcal K_\F$ follows from the maximality of lifts:
given $\str{A}\in \K_\mathcal F$ we construct $\str{B}$ as the disjoint union of witnesses
of all maximal lifts of $\str{A}$ and apply the above proof.
\end{proof}
\begin{remark}The second part of the proof (after the lift is constructed) is essentially the same as the proof of Theorem~\ref{thm:mainstrongclosures}.
It is however more convenient to give the proof by means of Lemma~\ref{lem:closures} and~\ref{lem:iteratedpartitestep2} because we do
not need to go into a further analysis of the homogenising lift.
\end{remark}

\section{Examples of Ramsey classes}
\label{sec:examples}
We believe that Theorems~\ref{thm:mainstrong} and~\ref{thm:mainstrongclosures} generalise most proofs used in the structural Ramsey theory.
It is however often not obvious how to verify that a given class is a locally finite subclass of a known Ramsey class (which is needed to apply Theorem~\ref{thm:mainstrong})
or that a given class is a multiamalgamation class (for Theorem~\ref{thm:mainstrongclosures}). In this section we give multiple examples which are chosen to demonstrate different techniques used to verify conditions needed to apply our main results.

We start by recalling some classical corollaries of the Ne\v set\v ril-R\" odl Theorem (in Section~\ref{sec:graphs}) and then we show the Ramsey property of
several classes (old and new) and thus illustrate the versatility of applications of Theorem~\ref{thm:mainstrong} (in Sections~\ref{sec:examplesstrong} and~\ref{sec:examplesclosure}), Theorem~\ref{thm:mainstrongclosures} (in Section~\ref{sec:manyorders}) and Theorem~\ref{thm:main} (in Section~\ref{sec:exampleslifts}).

Unless explicitly
stated, all our examples of lifts are precompact and have the lift property. 

\subsection {Ramsey lifts of free amalgamation classes}
\label{sec:graphs}
\label{sec:bipgraphs}
\label{sec:acyclic}

Recall the Ne\v set\v ril--R\"odl theorem (Theorem~\ref{thm:NR}).  In the
model-theoretic context it is often understood as a theorem about Ramsey lifts
of free amalgamation classes of relational structures. Since our techniques deal also with different
kinds of amalgamation (free amalgamation, strong amalgamation and amalgamation
with closures) let us first state a variant of Theorem~\ref{thm:NR}
in the more refined setting.
\begin{definition}[Ordered free amalgamation property]
  Let $L$ be a language containing a binary relation $\leq$.  We say that a class
  $\K$ of ordered $L$-structures has the \emph{ordered free amalgamation property} if for
  every $\str{A},\str{B}_1,\str{B}_2\in \K$ every ordered structure $\str{C}$ created
  as a free amalgamation of $\str{B}_1$ and $\str{B}_2$ over $\str{A}$ with
  $\leq_\str{C}$ completed arbitrarily to a linear order of $C$ is in $\K$.
\end{definition}
Note that it is not true that every class $\K$ with the ordered free amalgamation
property would become a free amalgamation class if order is removed. For example,
the class of all finite acyclic graphs with a linear extension has the ordered free
amalgamation property but the class of all finite acyclic graphs is not an
amalgamation class. 
\begin{corollary}[of Theorem~\ref{thm:NR}]
\label{cor:freeamalg}
  Let $L$ be a relational language containing a binary relation $\leq$ and let $\K$ be an
  amalgamation class of ordered $L$-structures with the ordered
  free amalgamation property. Then $\K$ is a Ramsey class.
\end{corollary}
\begin{proof}
Let $\K$ be a class of ordered $L$-structures with the ordered
free amalgamation property.  Denote by $\mathcal E$ the class of all ordered
$L$-structures $\str F$ such that $\str F\notin \K$ and every proper
substructure of $\str F$ is in $\K$ (\ie{}\ the family of minimal obstacles).

Clearly $\K=\Age(\Forbi(\mathcal E))$.
To show that $\K$ is Ramsey, we apply Theorem~\ref{thm:NR}. For that, we only
need to verify that $\mathcal E$ is a family of ordered irreducible structures.

Consider, to the contrary, that there is $\str{F}\in \mathcal R$ such that it is an
ordered free amalgamation of two of its proper substructures $\str{F}_1$ and $\str{F}_2$
By the construction of $\mathcal E$ we have that $\str{F}_1,\str{F}_2\in \K$. But then $\str{F}\notin \K$ gives
a contradiction with $\K$ satisfying ordered free amalgamation property.
\end{proof}

As a warmup for many examples below, let us give several other special cases of Theorem~\ref{thm:NR} and Corollary~\ref{cor:freeamalg}. In the process, we will also demonstrate two techniques to overcome the limitations of Theorem~\ref{thm:NR} and Corollary~\ref{cor:freeamalg}.

The easiest class we discuss is the class of all finite directed graphs (digraphs) $\mathcal D$ and its lift $\ordclass {\mathcal D}$ adding a linear order
on vertices. More precisely, $\mathcal D$ consists of structures in the language consisting of a single binary relation $E$ with no restrictions. $\ordclass{\mathcal D}$
extends the language by a binary relation $\leq$ and consists of all structures $\str{A}$ where $\leq_\str{A}$ is a linear order.
The class $\mathcal G$ of all (undirected) graphs may be viewed as a subclass of $\mathcal D$ consisting of those structures $\str{A}$ where $E$ is symmetric and irreflexive.
$\ordclass{\mathcal G}$ is a lift of $\mathcal G$ adding a free order on vertices.
We immediately obtain:
\begin{corollary}
\label{cor:digraphs}
The class $\ordclass{\mathcal {D}}$ of all finite directed graphs with a free ordering of vertices is a Ramsey class.
The class $\ordclass{\mathcal {G}}$ of all finite (simple) graphs with a free ordering of vertices is a Ramsey class.
\end{corollary}
\begin{proof}
Follows from Corollary~\ref{cor:freeamalg} as both $\ordclass{\mathcal {D}}$ and $\ordclass{\mathcal {G}}$ are amalgamation classes with the ordered
free amalgamation property.
\end{proof}
\begin{remark}
The same technique can be also used for the class of all finite digraphs omitting a given set of tournaments
(these are called \emph{Henson graphs}) or the class of all finite graphs omitting $K_n$ for a fixed $n$.
Up to complementation this exhausts all ultrahomogeneous undirected graphs where the lift adding a free order on vertices forms a Ramsey class~\cite{Nevsetvril1989a}.
\end{remark}

Equivalently we can say that $\ordclass{\mathcal {D}}$ is a Ramsey lift of class $\mathcal D$
and $\ordclass{\mathcal {G}}$ is a Ramsey lift of class $\mathcal G$. It is a classical result that the lift $\ordclass{\mathcal {G}}$ has the lift
property (Definition~\ref{defn:ordering}). However, $\ordclass{\mathcal {D}}$
does not: we can order every directed graph such that all vertices with loops come before all vertices without loops.

Consider the class
$\ordclasssup{\mathcal D}{0}$ of all finite directed graphs ordered in a way that
vertices with loops are before vertices without loops.
It is easy to see that $\ordclasssup{\mathcal {D}}{0}$ is also a Ramsey class: Given pairs of
ordered directed graphs $\str{A},\str{B}\in \ordclasssup{\mathcal {D}}{0}$ and an ordered directed graph $\str{C}'\in \ordclass{\mathcal {D}}$
such that $\str{C}'\longrightarrow (\str{B})^\str{A}_2$ (given by Corollary~\ref{cor:digraphs}), we can construct an ordered directed graph
$\str{C}\in \ordclasssup{\mathcal {D}}{0}$ such that $\str{C}\longrightarrow (\str{B})^\str{A}_2$ by reordering vertices of $\str{C}'$
so that all vertices with loops come first without breaking any of the embeddings of $\str{B}$ into $\str{C}$.

Clearly both $\ordclass{\mathcal{D}}$ and $\ordclasssup{\mathcal{D}}{0}$
are Ramsey lifts of $\mathcal D$. One could claim that $\ordclasssup{\mathcal {D}}{0}$ is better
because there are fewer ways to lift a given directed graph. In fact,
it can be shown that $\ordclasssup{\mathcal{D}}{0}$ has the lift property with respect to
$\mathcal {D}$~\cite{Jasinski2013}.

We generalise this observation by the following concept of admissible ordering:
\begin{definition}
Let $L$ be a language containing a binary relation $\leq$. Denote by $\mathcal O_L$ the class of all isomorphism types of $L$-structures with one vertex
and let $\leq_L$ denote a fixed linear order on $\mathcal O_L$.
 Given  an ordered $L$-structure $\str A$, we say that its order is \emph{$\leq_L$-admissible} if for every pair of distinct vertices $u,v\in \str{A}$ it holds that whenever $\str{O}_u<_L\str{O}_v$ then $u \leq_\str{A},v$.
Here $\str{O}_u$ and $\str{O}_v$ are the structures in $\mathcal O_L$ isomorphic to structure inducted by $\str{A}$ on $\{u\}$ and $\{v\}$ respectively.
\end{definition}
The order $\leq_L$ will be usually understood from the context and thus we will just speak of an admissible order of the structure.
\begin{prop}
\label{prop:admisible}
Let $L$ be a language and $\K$, be a class of $L$-structures, and $\leq_L$ be a linear order of $\mathcal O_L$. If the lift $\ordclass \K$ of $\K$ adding a free order on vertices is
a Ramsey class then the lift $\ordclasssup \K{0}$ of $\K$ adding an $\leq_L$-admissible order on vertices is also a Ramsey class.
\end{prop}
\begin{proof}
Let $\str{A}, \str{B}$ be structures in $\ordclasssup \K{0}$ and $\str{C}'\in \ordclass \K$ such that $\str{C}'\longrightarrow(\str{B})^\str{A}_2$.
The $\leq_L$-admissibly ordered structure $\str{C}\longrightarrow(\str{B})^\str{A}_2$ is constructed by re-ordering the vertices of $\str{C}'$ without breaking any of the desired
embeddings of $\str{B}$ (which is always possible).
\end{proof}

The phenomenon of admissible orderings is observed already in~\cite{Nevsetvril1989a} and~\cite{Kechris2005} in the context of bipartite graphs where the Ramsey lift adds a unary relation $R$ which
denotes one of the bipartitions (here $R$ may come from ``right''). This representation of bipartite graphs forms a free amalgamation class and thus the lift adding a free order on vertices is Ramsey (by Corollary~\ref{cor:freeamalg}).
Again, this lift does not have the lift property, however it can be obtained by means of Proposition~\ref{prop:admisible}.
Here the admissible ordering can be chosen in such a way that all vertices in the unary relation $R$ are before the remaining vertices. Such an order, which respects the bipartition, is also sometimes called a \emph{convex ordering}~\cite{Kechris2005}.
These observations can be further generalised.
\begin{corollary}
\label{cor:Hcolor}
Let $\str{H}$ be a finite ordered relational structure. Denote by $\CSP(\str{H})$ the class
of all structures with a homomorphism into $\str{H}$. Then the class
all finite ordered structures in $\CSP(\str{H})$ has a Ramsey lift adding $\vert H\vert $ unary relations.
\end{corollary}

\begin{proof}
Given $\str{H}$, we  lift the language by unary relations $\ext{}{i}$, $i\in H$ (those are special cases of relations used in Section~\ref{sec:Flifts} and thus we use letter $\ext{}{}$).
For a finite ordered structure $\str{A}\in \CSP(\str{H})$ we construct a lift $\str{A}^+$ by choosing a homomorphism $c\colon\str{A}\to \str{H}$ arbitrarily and putting $(v)\in \ext{A}{i}$ if and only if $c(v)=i$. (Our lifts explicitly fix the homomorphism to $\str{H}$.) It is easy to see that the lifted class is a free amalgamation class and the Ramsey property follows by Corollary~\ref{cor:freeamalg}.
\end{proof}
It is easy to check that the described lift has the lift property with admissible orderings
  whenever every homomorphism $\str{H}\to\str{H}$ is an automorphism (\ie{}\ $\str{H}$ is a core~\cite{Hell2004}).

In some special cases it is possible, for a given $\F$, construct a finite structure $\str{H}$ such that $\Forbh(\F)=\CSP(\str{H})$.
In such situations $\str{H}$ is called the \emph{homomorphism dual} of $\F$.
All homomorphism dualities have been characterised in~\cite{Nesetril2000} and~\cite{Erdos2012}, see also~\cite{Kun2008}.
In the context of universal structures, this can be further generalised to the notion of
\emph{monadic lifts} (\ie{}\ homogenising lifts which add only finitely many unary relations). Classes $\Forbh(\F)$ with
monadic lift are discussed in~\cite{Hubicka2009}: even if there is no homomorphism dual, 
every monadic homogenising lift (see Section~\ref{sec:homogenization}) is an amalgamation class of ordered structures with amalgamation which is free in all relations except for $\leq$ and thus Theorem~\ref{thm:NR} can still be applied.  

\begin{corollary}
\label{cor:duals}
Let $\F$ be a regular family of finite connected weakly ordered structures such that all minimal separating cuts consist of one vertex.
Then there exists a Ramsey lift of $\Forbh(\F)$ adding only finitely many unary relations.
\end{corollary}
\begin{proof} This follows as a combination of Corollary~\ref{cor:Hcolor} with~\cite{Hubicka2009} (as indicated above).
\end{proof}

Analogous proofs also give corresponding corollaries for homomorphism-embeddings:
\begin{corollary}
Let $\str{H}$ be a finite ordered structure. Denote by $\CSP_{\mathrm {he}}(\str{H})$ the class
of all finite ordered structures with a homomorphism-embedding to $\str{H}$. Then the class
$\CSP_{\mathrm {he}}(\str{H})$ has a unary Ramsey lift adding only $\vert H\vert $ unary relations.
\end{corollary}
\begin{proof}
Again it is easy to show that the lift fixing a homomorphism embedding to $\str{H}$ leads to a free amalgamation class.
\end{proof}
The following corollary represents the special (and easy) case of Theorem~\ref{thm:main}:
\begin{corollary}
Let $\F$ be a regular family of finite connected weakly ordered structures such that all minimal separating cuts consist of one vertex.
Then there exists a Ramsey lift of $\Forb(\F)$ adding only finitely many unary relations.
\end{corollary}
\begin{proof} This follows by a combination of Corollary~\ref{cor:freeamalg} and Theorem~\ref{mainthm}.
\end{proof}

Structures with minimal separating cuts (see Definition~\ref{def:separating}) of size one generalise graph trees~\cite{Nesetril2000}: every such structure can be constructed from a graph tree by replacing edges by arbitrary ordered irreducible structures (or, in other words,
every two-connected component of its Gaifman graph is a complete
graph).  We know, by Theorem~\ref{mainthm2}, that regularity (Definition~\ref{def:regular}) is a necessary condition
for the existence of an $\omega$-categorical universal structure in $\Forb(\F)$.

It may seem that by considering monadic lifts  we exhausted all
possible applications of Theorem~\ref{thm:NR}.  Yet there is another case:
\cite{Nevsetvril1984} gives an example of an application of Theorem~\ref{thm:NR} 
which uses order to give a Ramsey lift of the class of all finite directed acyclic graphs.
 Because cycles are not irreducible structures, it is necessary to use other
 means to describe the acyclicity.  Instead of forbidding directed cycles, we
 (dually) use the fact that every directed acyclic graph has a linear extension. Finite directed acyclic graphs with linear extensions form a class with the ordered free amalgamation property and we thus immediately obtain:
\begin{corollary}[\cite{Nevsetvril1984}]
\label{cor:acyclic}
The class $\ordclass{\mathcal A}$ of all finite directed acyclic graphs with a linear extension is a Ramsey class.
\end{corollary}
 One can verify the lift property and show that every Ramsey lift of the
 class of directed acyclic graphs always fixes a linear extension. This
shows that this technical looking trick (of adding a linear extension) is in fact necessary. 
The infinite linear order may also be seen as an infinite dual of the class of all directed acyclic graphs.

\subsection {Ramsey classes with strong amalgamation}
\label{sec:examplesstrong}
In this section we focus on strong amalgamation classes of relational structures which are Ramsey by an application of Theorem~\ref{thm:mainstrong}.
Recall that Theorem~\ref{thm:mainstrong} states that local finiteness  is
essentially the only condition which prevents us from showing the Ramsey
property of every strong amalgamation class of ordered structures.

\subsubsection {Partial orders with linear extension}
\label{sec:posets}
We start with an example of a Ramsey class with non-trivial local finiteness.
This serves as a warm-up example introducing all necessary concepts.

Let $L_P$ be the language with two binary relations $\sqsubseteq$ and $\leq$.
We consider partial orders $(A,\sqsubseteq_A)$ with a fixed linear extension denoted by $\leq_A$. The class of such finite $L_P$-structures $\str{A}$ will be denoted  by $\ordclass{\mathcal P}$.
By Theorem~\ref{thm:NR} we know that the class $\ordclass{\Str}(L_P)$ (that is, the class of all finite $L_P$-structures $\str{A}$ where $\leq_A$ is a linear order of vertices) is Ramsey.
We aim to prove that $\ordclass{\mathcal P}$ is a locally finite subclass of $\ordclass{\Str}(L_P)$.
By an application of Theorem~\ref{thm:mainstrong} we then obtain that $\ordclass{\mathcal P}$ is Ramsey.

\medskip

Recall the notions of irreducible structures (Definition~\ref{def:irreducible}), homomorphism-embedding (Definition~\ref{def:homembed}) and completion (Definition~\ref{defn:completion}).
As in the definition of locally finite subclass (Definition~\ref{def:localfinite}), let $\str{C}_0\in
\ordclass{\Str}(L_P)$ be fixed and consider an $L_P$-structure $\str{C}$ such that:
\begin{enumerate}
\item $\str{C}_0$ is a $\ordclass{\Str}(L_P)$-completion of $\str{C}$ (equivalently, there exists a homo\-mor\-phism-embedding $f\colon\str{C}\to\str{C}_0$), and
\item every irreducible substructure of $\str{C}$ is in $\ordclass{\Str}(L_P)$.
\end{enumerate}
When does $\str{C}$ have no $\ordclass{\mathcal P}$-completion?

First observe that since every irreducible substructure of $\str{C}$ is in $\ordclass{\mathcal
P}$ we get that for every $v\in C$ it holds that $v \leq_\str{C} v$ and
$v \sqsubseteq_\str{C} v$. We also get that $u \sqsubseteq_\str{C} u'$ implies
$u \leq_\str{C} u'$ for every $u,u'\in C$ and thus relation $\sqsubseteq_\str{C}$ is a subset of $\leq_\str{C}$.

Because $\str{C}$ is not necessarily an ordered structure, $\leq_\str{C}$ may not be a linear order (for example, it is not necessarily transitive).
However, since $\leq_{\str{C}_0}$ is a linear order and $f$ is a
homomorphism-embedding we know that $\leq_\str{C}$ is an acyclic graph extended by loops on all vertices. Thus there is a linear order $\leq'_\str{C}$ of $C$ such that $\leq_\str{C}$ (and consequently also $\sqsubseteq_\str{C}$) is a subset of $\leq'_\str{C}$. 

Let $\str{C}'$ be a structure with vertex set $C$ where $\leq_{\str{C}'}$ is
an arbitrary completion of $\leq_\str{C}$ to a linear order and $\sqsubseteq_{\str{C}'}$
is the transitive closure of $\sqsubseteq_{\str{C}}$.
From the discussion above we know that
$\str{C}'\in \ordclass{\mathcal P}$. The identity on $C$ is a homomorphism $h\colon\str{C}\to\str{C}'$.
However, $h$ is not necessarily a homomorphism-embedding (it may not be an embedding on irreducible substructures) and thus
$\str{C}'$ is not necessarily a completion of $\str{C}$.

If $h$ is not a homomorphism-embedding $\str{C}\to\str{C}'$, $L_P$ being binary implies that there exists a pair of vertices $u, v\in C$ such that $\str{C}$ induces an irreducible
substructure $\str{D}$ on $\{u,v\}$ and $h$ restricted to $\str{D}$ is not embedding.  Because we know that $\str{D}\in \ordclass{\mathcal P}$ it follows that there is only one such $\str D$: $u \leq_\str{C} v$, $u \not\sqsubseteq_\str{C} v$ and $u \leq_{\str{C}'}$, $u \sqsubseteq_{\str{C}'} v$ in $\str{C}'$.
In this case $u \sqsubseteq_{\str{C}'} v$ is in the transitive closure of $\str{C}$ and thus there is a 
sequence $u=u_1,u_2,\ldots, u_\ell=v$ of distinct vertices of $\str{C}$ such that
\begin{enumerate}
\item $u_i\sqsubseteq_\str{C} u_{i+1}$ for every $1\leq i<\ell$ (and consequently also $u_i\leq_\str{C} u_{i+1}$),
\item $u \not\sqsubseteq_\str{C} v$, and
\item $u \leq_\str C v$.
\end{enumerate} We call such a structure a \emph{quasi-cycle}.

Clearly no quasi-cycle has a $\ordclass{\mathcal P}$-completion and thus they
are the only obstacles. We thus conclude that:
\begin{claim}
$\str{C}$ satisfying our assumptions (given by Definition~\ref{def:localfinite} of a locally finite subclass) has a $\ordclass{\mathcal P}$-completion if and only if it contains no quasi-cycle.
\end{claim}
Local finiteness is usually shown by giving an upper bound on the size of such an obstacle.
Here it follows from the observation
that $f$ has to be injective on every quasi-cycle because 
$u_i \leq_C u_{i+1}$ for every $1\leq i< \ell$ and $\leq_{C_0}$ is a linear order. We get that $\ell\leq |C_0|$.

Putting $n(\str{C}_0)=|C_0|$, we can make sure that all structures $\str{C}$ considered by the definition
of locally finite subclass contains no quasi-cycle.
We have thus verified that $\ordclass{\mathcal P}$ is a locally finite subclass of
$\ordclass{\Str}(L_P)$ and thereby proved the following theorem.

\begin{theorem}[\cite{Nevsetvril1984,Trotter1985}]
\label{thm:posets}
The class $\ordclass{\mathcal P}$ of all finite partial orders with a linear extension is Ramsey.
\end{theorem}
Another way to prove Theorem~\ref{thm:posets}  is to deduce it from Theorem~\ref{cor:acyclic}
(\ie{}\ use $\ordclass{\mathcal A}$ instead of $\ordclass{\Str}(L_P)$ as the base Ramsey class).
The analysis of obstacles in this setting is almost equivalent to the one above.

Many applications of Theorem~\ref{thm:mainstrong} follow the scheme described here:
\begin{enumerate}
\item First it is important to understand the completion procedure for the given class.
\item Based on the analysis of the completion procedure one can derive the class of minimal obstacles
\item Finally the bound on the size of obstacles is given. 
\end{enumerate}
A useful tool in this analysis is also Proposition~\ref{prop:strongcompletion}
which shows that instead of completions we can in many cases study strong completions only which
are easier to understand.

\begin{remark}
Note that for local finiteness it is critical to use the linear extension.
In fact, the class of all finite partial orders with a free linear order is not Ramsey~\cite{Fouche1997} and it is possible
to verify the lift property of $\ordclass{\mathcal P}$ and thus show that the
linear extension is necessary.
\end{remark}
\subsubsection {$S$-metric spaces with no jumps}
\label{sec:metric}
In this section we strengthen results of~\cite{Nevsetvril2007} which establishes the Ramsey
property of the class of ordered finite rational metric spaces with a free ordering on vertices, \cite{Dellamonica2012} which establishes the Ramsey property of the class of finite ordered graphs with a free ordering of vertices with respect to metric embeddings, and~\cite{The2010} which studies Ramsey expansions for special choices of $S$.
Using the results of~\cite{Sauer2013}, we characterise, in a surprisingly simple way, Ramsey classes of ordered metric spaces which only use distances from a given closed set $S$ (in Theorem~\ref{thm:Smetric} and Corollary~\ref{cor:Smetricfin} proved in Section~\ref{sec:metric2}).

In this section,
we start by recalling the basic properties of $S$-metric spaces. It appears that it is useful to consider two main types of a distance set $S$: without jumps (defined in Definition~\ref{defn:jump} and treated in this section) and with jumps (treated in Section~\ref{sec:metric2} by means of closures).

\begin{definition}
\label{defn:smetric}
Given $S\subseteq \mathbb R_{>0}$ (that is, a subset of positive reals) an \emph{$S$-metric
space $\str{A}$} is a pair $(A,d_\str{A})$ where $A$ is the vertex set and $d$ is
a binary function $d_\str{A}\colon A^2\to S\cup \{0\}$ (the \emph{distance function}) such that:
\begin{enumerate}
 \item $d_\str{A}(u,v)=0$ if and only if $u=v$,
 \item $d_\str{A}(u,v)=d_\str{A}(v,u)$, and
 \item $d_\str{A}(u,w)\leq d_\str{A}(u,v)+d_\str{A}(v,w)$ for every $w\in A$ (the \emph{triangle inequality}).
\end{enumerate}
We denote by $\mathcal M_S$ the class of all finite $S$-metric spaces.
\end{definition}

\begin{definition}
\label{defn:lmetric}
We interpret an $S$-metric space as a
relational structure $\str{A}$ in the language $L_S$ with (possibly infinitely many) binary relations $\rel{}{s}$, $s\in S$, where
we put, for every $u\neq v\in A$, $(u,v)\in \rel{A}{\ell}$ if and only if $d(u,v)=\ell$. We do not explicitly represent that $d(u,u)=0$ (\ie{}\ no loops are added).

Every $L_S$-structure where all relations are symmetric and irreflexive and every pair of vertices is in at most one relation is called an \emph{$S$-graph}, which we may alternatively view as a graph with edges labelled by $S$.
Every non-induced substructure of an $S$-metric
space is an \emph{$S$-metric graph} ($S$-metric graphs are structures with have
a strong completion to $S$-metric space in the sense of Definition~\ref{defn:completion}).
Every $S$-graph that is not an $S$-metric graph is a \emph{non-$S$-metric graph}.
\end{definition}

  Not every choice of
$S$ leads to an amalgamation class $\mathcal M_S$ (see
Figure~\ref{fig:nonmetric}).  Those that do satisfy the following condition:
\begin{figure}
\centering
\includegraphics{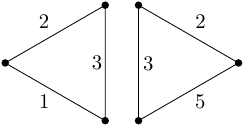}
\caption{$\{1,2,3,5\}$-metric spaces do not have the amalgamation property.}
\label{fig:nonmetric}
\end{figure}

\begin{definition}[\cite{Delhomme2007}]
A subset $S\subseteq \mathbb R_{>0}$ satisfies the \emph{4-values condition}, if
for every $a,b,c,d\in S$ such that there exists $x\in S$ such that triangles with distances $a$--$b$--$x$ and $c$--$d$--$x$ satisfy
the triangle inequality there exists $y\in S$ such that triangles with distances $a$--$c$--$y$ and
$b$--$d$--$y$ satisfy the triangle inequality.
\end{definition}

\begin{figure}
\centering
\includegraphics{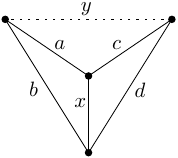}
\caption{The 4-values condition.}
\label{fig:4values}
\end{figure}
The 4-values condition describes a strong amalgamation of two 3-point metric
spaces over a common 2-point subspace, see Figure~\ref{fig:4values}. Clearly, this is a necessary condition for the amalgamation
property of $\mathcal{M}_S$. In fact, one can prove that under certain conditions (for closed sets $S$), this is also a sufficient condition (see Theorem~\ref{thm:sauercharacterisation}).

The universal ultrahomogeneous metric space was constructed by Urysohn \cite{Urysohn1927, Katetov1986} (by a \Fraisse{}-type construction, predating \Fraisse{} by more than 20 years).
Following~\cite{Sauer2013} we say that
an $S$-metric space $\str{U}$ is an \emph{$S$-Urysohn metric space} if it is homogeneous, separable, complete and if it isometrically embeds every separable $S$-metric space $\str{A}$.
 Because sets $S$ may be uncountable (and
Theorem~\ref{fraissethm} cannot be directly applied), the existence of a Urysohn
$S$-metric space requires $S\cup \{0\}$ to be closed and to have $0$ as a limit point:
\begin{theorem}[\cite{Sauer2013}]
\label{thm:sauercharacterisation}
Let $S\subseteq \mathbb R_{>0}$ be a set with 0 as a limit of $S\cup \{0\}$. Then there 
exists a Urysohn $S$-metric space if and only if $S\cup \{0\}$ is
a closed subset of $\mathbb R$ satisfying the 4-values condition.

Let $S\subseteq \mathbb R_{>0}$ which does not have 0 as a limit of $S\cup
\{0\}$.  Then there exists a Urysohn $S$-metric space if and only
if $S$ is a countable subset of $\mathbb R$ satisfying the 4-values
condition.

Any two Urysohn metric spaces having the same set of distance
$S$ are isometric.
\end{theorem}

We state the observations about strong amalgamation as follows (and for completeness we include their proofs):
\begin{corollary}[\cite{Sauer2013}]
\label{cor:sauer2}
Let $S\subseteq \mathbb R_{>0}$ be a subset of positive reals. $S$ satisfies the 4-values condition if and only if $\mathcal M_S$ has the strong amalgamation property.
\end{corollary}
\begin{proof}
We show that for every $S$ satisfying the 4-values condition the class $\mathcal M_S$ has strong amalgamation.
Let $\str{A},\str{B}_1,\str{B}_2\in \mathcal M_S$ such that the identity is
an embedding of $\str{A}$ to both $\str{B}_1$ and $\str{B}_2$. We will construct
a strong amalgamation of $\str{B}_1$ and $\str{B}_2$  over $\str{A}$.
Because $\mathcal M_S$ is hereditary, we can assume without loss of generality
that $B_1\setminus A=\{u\}$ and $B_2\setminus A=\{v\}$. Let $w\in A$ be a vertex such
that $a=d_{\str{B}_1}(u,w) + d_{\str{B}_2}(w,v)$ is minimised and let $w'\in A$ be a vertex
such that $b=\vert d_{\str{B}_1}(u,w')-d_{\str{B}_2}(w',v)\vert $ is maximised. By the triangle-inequality, $a$ is an upper
bound on the distance of $u$ and $v$ while $b$ is a lower bound.
If $w\neq w'$, the  4-values condition says precisely that there is $s\in S$ such that we can put $u$ and $v$ to be in distance $s$. If
$w=w'$ then we put $u$ and $v$ to be in distance $\max \{d_{\str{B}_1}(u,w), d_{\str{B}_2}(w,v)\}\in S$.

If $S$ does not satisfy the 4-values condition then we know that there are already two triangles which do not amalgamate over a common edge.
\end{proof}
For certain sets $S$, the 4-values condition can be expressed in the following neat algebraic way due to Sauer~\cite{Sauer2013b}.
For $a,b\in S$ denote by $a\oplus_S b=\sup\{x\in S; x\leq a+b\}$.
The algebraic characterisation of sets with the 4-values
condition allows us to easily complete $S$-metric graphs to $S$-metric spaces.
\begin{theorem}
\label{thm:sauer2}
Assume that $S\subseteq \mathbb R_{>0}$ is a subset of positive reals on which the $\oplus_S$ operation is defined.
Then $S$ satisfies the 4-values condition if and only if the operation $\oplus_S$
is associative.
\end{theorem}
Note that Theorem~\ref{thm:sauer2} is proved in~\cite{Sauer2013b} only for those $S$ such that $S\cup\{0\}$ is closed. An important example of a non-closed $S$ where $\oplus_S$ is defined and associative is the set of all positive rational numbers. For this reason we state it in this generalised form which still follows by the same argument.
\begin{proof}
Assume that $S$ satisfies the 4-values condition. Given $a,b,c\in S$ we aim to show that $a\oplus_S (b\oplus_S c)=(a\oplus_S b)\oplus c$.

By the definition of $\oplus_S$ we 
 know that the triangle with distances $a,b,a\oplus_S b$ is metric. Similarly the triangle with distances $c,a\oplus_S b,(a\oplus_S b)\oplus_S c$ is metric.
Now we apply the 4-values condition for $a,b,(a\oplus_S b)\oplus_S c,c$ (we have verified that we can put $x=a\oplus_S b$)
to obtain $y$ such that the triangles with distances $b,c,y$ and $a,(a\oplus_S b)\oplus_S c,y$ respectively are metric.
This implies that $$(a\oplus_S b)\oplus_S c\leq a\oplus_S y\geq a\oplus_S (b\oplus_S c)$$ and by a symmetrical argument we get the other inequality which proves the first implication.

Now assume that $\oplus_S$ is associative and $a,b,c,d\in S$ are chosen such that there exists $x$
such that $a,b,x$ and $x,c,d$ are metric. Put $y=\min (a\oplus_S c,b\oplus_S d)$.
Without loss of generality we can assume that $a\oplus_S c\leq b\oplus_S d$. It remains to verify that the
triangle $b,d,a\oplus_S c$ is metric. Clearly $b+d\geq b\oplus_S d \geq a\oplus_S c$. To see that
$(a\oplus_S c)+b\geq d$ observe that $(a\oplus_S c)\oplus_S b=c\oplus_S (a\oplus_S b)\geq c\oplus x\geq d$.
Analogously we get that $a\oplus_S c+d\geq b$.
\end{proof}

Let $\str{G}$ be an $S$-metric graph and let $\vv{w}=(w_1,w_2,\ldots, w_n)$ be a sequence
of vertices forming a walk in $\str{G}$ (that is, for every $1\leq i<n$ we
have $w_i\neq w_{i+1}$ and their distance is defined in $\str{G}$; \ie{}\ $(w_i, w_{i+1})\in \rel{G}{l}$ for some $l\in S$). The \emph{$S$-length} of $\vv{w}$ is defined as $d_\str{G}(w_1,w_2)\oplus_S d_\str{G}(w_2,w_3)\oplus_S \cdots \oplus_S d_\str{G}(w_{n-1},w_n)$.

The following corollary is a small strengthening some results of~\cite{Sauer2013} and~\cite{Sauer2013b} (considering not only closed sets) which are important for our construction.
\begin{prop}
\label{prop:sauer1}
Assume that $S\subseteq \mathbb R_{>0}$ is a subset of positive reals on which the $\oplus_S$ operation is defined and associative.
Then the following holds.
\begin{enumerate}
\item \label{cor:suaer1:1} Let $\str{G}$ be a finite $S$-metric graph.
Denote by $d'(u,v)$ the minimal $S$-length of a walk from $u$ to $v$ and by
$\vv{W}(u,v)$ the corresponding walk.
Then $\str{G}$ can be completed to an $S$-metric space $\str{A}$ by putting $d_\str{A}(u,v)=d'(u,v)$ for every pair $u\neq v\in G$.
\item \label{cor:suaer1:2} A finite $S$-graph $\str{G}$ is an $S$-metric graph if and only if all of its cycles are $S$-metric.
\end{enumerate}
\end{prop}
In this setting we will call the metric graph corresponding to $(G,d')$ the \emph{shortest path completion} of $\str{G}$.
\begin{proof}
Both statements can be seen as easy consequences of the associativity of $\oplus_S$:

\medskip

$\it \ref{cor:suaer1:1}.$ 
First assume that $\str{G}$ is $S$-metric. We show that the completion
described will give an $S$-metric space.
  First we verify that $d'$ satisfies the triangle inequality. Assume, to the contrary, the existence of
vertices $u,v,w$ such that $d'(u,v)> d'(u,w)+d'(w,v)$. Concatenate the walks
$\vv{W}(u,w)$ and $\vv{W}(w,v)$ to get a walk $\vv{p}$. By associativity of $\oplus_S$ we get that the $S$-length of $\vv{p}$
is $d'(u,w)\oplus_S d'(w,v)\geq d'(u,v)$. It follows that $d'(u,v)\leq d'(u,w)\oplus_S d'(w,v)\leq d'(u,w)+d'(w,v)$ which is a contradiction.

Now it remains to check that $d_\str{G}(u,v)=d'(u,v)$ whenever $d_\str{G}(u,v)$ is defined.
 We show a stronger claim: Let $\str{B}$ be a completion of $\str{G}$ to an $S$-metric space
then $d_\str{B}(u,v)\leq d'(u,v)$ for every $u\neq v\in G$.

We proceed by induction on the length of the $S$-walk $\vv{W}(u,v)$ which we call $n$.  For $n=3$ this follows from the triangle
inequality.  For $n>3$ denote by $(p_1,p_2,\ldots ,p_{n-1},p_n)$ the vertices of $\vv{W}(u,v)$.
By the induction hypothesis we know that $d_\str{B}(u,p_{n-1})\leq d'(u,p_{n-1})$.
The inequality then follows from associativity of $\oplus_S$ and the
triangle inequality. This finishes the proof of~\ref{cor:suaer1:1}.

\medskip

$\it \ref{cor:suaer1:2}.$
Assume that $\str{G}$ is non-$S$-metric.  In this case we have a pair of vertices
$u$ and $v$ with their distance defined such that $d'(u,v) < d_\str{A}(u,v)$. Because $\oplus_S$ is monotone,
it is easy to see that the walk $\vv{W}(u,v)$ can be turned into a path. A non-metric
cycle is then induced on vertices of this path. For the other implication, clearly if a graph is $S$-metric, then in particular all of its cycles are.
\end{proof}

In Example~\ref{example:13} we showed that the class of all finite $\{1,3\}$-metric spaces is not locally finite in the class of all $\{1,3\}$-graphs.
This is an important example. It indicates that ``large gaps'' in the distance set $S$ have to be treated with care. In the rest of this section we consider 
only those sets $S$ where such a scenario does not happen, that is, $a\oplus_S b > \max(a,b)$.
(Sets with gaps will be treated in Section~\ref{sec:metric2}.)
Such sets are characterised by the absence of jump numbers:

\begin{definition}
\label{defn:jump}
Given $S\subseteq \mathbb R_{>0}$ and $a\in S$, we say that $a$ is a \emph{jump
number} if $a$ is not the maximum element of $S$ and there is no $b\in S$ with $a < b \leq 2a$.
\end{definition}
We shall observe below that $\mathcal M_S$ is locally finite for sets $S$ with no jump numbers:

\begin{lemma}
\label{lem:archimedean}
Let $S\subseteq \mathbb R_{>0}$ be a set which has no jump numbers such that $\oplus_S$ is well-defined and associative. Then for every $a, b\in S$ there is an integer $n$ such that $n\times a \geq b$, where $n\times a$ is defined as
$$\underbrace{a\oplus a\oplus \cdots \oplus a}_{n\text{ times}}.$$
\end{lemma}
\begin{proof}
For a contradiction assume the existence of $a,b\in S$ such that $n\times a < b$ for every $n$. Note that $S$ containing no jump numbers implies that $a\oplus_S a = a$ if and only if $a$ is the maximum element of $S$. This means that the sequence $a_1, a_2, \ldots$ defined as $a_i = i\times a$ is strictly increasing and bounded by $b$, hence has a limit which we call $\ell$. There is $m$ such that $\ell-a_m < a$. From the definition of $\oplus_S$ we have that $a_{m+1} = \sup\{s\in S : a_m + a \geq s\}$. It follows that $a_{m+1} \geq \ell$, which is a contradiction. 
\end{proof}

Operations satisfying the condition in Lemma~\ref{lem:archimedean} are often called \emph{archi\-medean}.

\begin{lemma}
\label{lem:nonSmetric}
Let $S\subseteq \mathbb R_{>0}$ be a set which has no jump numbers such that $\oplus_S$ is well-defined and associative. Then for every finite $S'\subseteq S$ there is $n=n(S')$ such that every non-$S$-metric cycle which only uses distances from $S'$
has at most $n$ vertices.
\end{lemma}
\begin{proof}
Put $m=\min(S')$, $M=\max(S')$ and pick $n$ such that $n\times m\geq M$ (such $n$ exists by Lemma~\ref{lem:archimedean}). Let $\str C$ be a non-$S$-metric cycle which uses only distances from $S'$. By Proposition~\ref{prop:sauer1} we know that $\str{C}$
contains a pair of vertices whose distance is longer than the $S$-length of the path connecting them.
Enumerate the vertices of $C$ as $v_1,v_2,\ldots,v_k$
such that $d_\str{C}(v_1,v_k)>d_\str{C}(v_1,v_2)\oplus_S d_\str{C}(v_2,v_3)\oplus_S\cdots \oplus_S d_\str{C}(v_{k-1},v_k)$.

Clearly $M\geq d_\str C(v_1, v_k)$ and also $d_\str{C}(v_i,v_{i+1})\geq m$ for every $1\leq i < k$. Putting this together (and using monotonicity of $\oplus_S$ we get that $M > k\times m$) and from Lemma~\ref{lem:archimedean} it follows that $m < n$, which is what we wanted.
\end{proof}

\begin{corollary}
\label{cor:finite4value}
Let $S\subseteq \mathbb R_{>0}$ be a set which has no jump numbers such that $\oplus_S$ is well-defined and associative. Then the class of all finite $S$-metric
spaces with free, \ie{}\ arbitrary, ordering of vertices, $\ordclasssup{\mathcal M}{S}$, is a
Ramsey class.
\end{corollary}

\begin{proof}
By Proposition~\ref{prop:sauer1} the weakly ordered structure such that every irreducible substructure is in $\ordclasssup{\mathcal M}{S}$ (and thus it is an $S$-graph) has a strong completion in
$\ordclasssup{\mathcal M}{S}$ if and only if all its cycles are $S$-metric.
By Proposition~\ref{prop:strongcompletion} we know that the existence of a strong completion is equivalent to the existence of a completion. Let $\str C_0$ be as in the definition of local finiteness (Definition~\ref{def:localfinite}) and put $S'$ to be the set of distances occurring in $\str C_0$. Clearly, every structure $\str C$ with a homomorphism to $\str C_0$ only contains distances from $S'$. Since, by Lemma~\ref{lem:nonSmetric}, the set of non-$S$-metric cycles using only distances from $S'$ is finite, the statement follow by Theorem~\ref{thm:mainstrong}.
\end{proof}
Note that already this corollary implies~\cite{Nevsetvril2007} and~\cite{Dellamonica2012}.
Before extending our construction to sets $S$ with jump numbers (which we will give in Section~\ref{sec:metric2}) we first need to overcome the problem explained in Example~\ref{example:13}. This will be done by adding functions into the language.

\subsubsection {Ramsey classes with a locally finite interpretation}
\label{sec:imaginaries}
One of key elements of the proofs of Theorems~\ref{thm:mainstrong} and~\ref{thm:mainstrongclosures} is the iterated partite
construction (Lemma~\ref{lem:iteratedpartite2}) where the local finiteness condition
gives a finite bound on the number of iterations.
However, this gives no methods on achieving local finiteness for specific classes.
Many classes are locally finite in a suitable Ramsey class by themselves (such as $S$-metric spaces without jumps). Nonetheless, there are examples of Ramsey classes which are not locally finite in any ``reasonable'' base Ramsey class. Sometimes, one can turn them into locally finite classes by means of a suitable interpretation.

In particular, we will be interested in the following standard (model-theoretic) method of elimination of imaginaries~\cite{Hodges1993,Shelah1978}.

Let $\str{A}$ be a relational structure.  An \emph{equivalence formula} is a
first order formula $\phi(\vv{x},\vv{y})$ which is symmetric and transitive on the set of
all $n$-tuples $\vv{a}$ of vertices of $\str{A}$ where
$\phi(\vv{a},\vv{a})$ holds (the set of such $n$-tuples is called the \emph{domain} of
the equivalence formula $\phi$). An \emph{imaginary} element $\vv{a}/\phi$ of
$\str{A}$ is an equivalence formula $\phi$ together with a representative
$\vv{a}$ of some equivalence class of $\phi$.

Structure $\str{A}$ \emph{eliminates  imaginary $\vv{a}/\phi$} if there
exists a first order formula $\Phi(\vv{x},\vv{y})$ such that there is a unique tuple $\vv{b}$
such that $\phi(\vv{x}, \vv{a}) \iff \Phi(\vv{x}, \vv{b})$.

\begin{example}
In the Urysohn $\{1,3\}$-metric space $\mathbb U_{\{1,3\}}$  there is an equivalence formula $\phi(x,y)$ which is satisfied for a pair of
vertices if and only if their distance is at most one.  The imaginary element $a/\phi$ then corresponds to the set of vertices in distance at most one from $a$. There is no way to eliminate
these imaginaries.

In order to turn $\mathbb U_{\{1,3\}}$ to a structure eliminating imaginaries defined by 
$\phi$, one can add a unary function $\func{}{}$ to the language 
and lift $\mathbb U_{\{1,3\}}$ to a structure $\mathbb U_{\{1,3\}}^+$ by
choosing precisely one vertex in every equivalence class defined by $\phi$ and putting, for every vertex $v$,
$F(v)$ to be the chosen vertex in the equivalence class of $v$.

Observe that the same formula $\phi$ is not an equivalence formula in the Urysohn $\{1,2,3\}$-metric space.
\end{example}

For a given ordered structure $\str{U}$ we say that $\phi$ is an \emph{equivalence formula on copies of $\str{A}$}
if and only if $\phi$ is an equivalence formula and moreover $\phi(\vv{a},\vv{a})$ holds if and only
if the structure induced by $\str{U}$ on $\vv{a}$ is isomorphic to $\str{A}$ and moreover order of vertices in $\vv{a}$
agrees with the order $\leq_\str{U}$.

\begin{prop}
\label{prop:imaginary}Let $\mathcal K$ be a hereditary Ramsey class of ordered $L$-structures, $\str{U}$ its \Fraisse{} limit,
$\str{A}$ be a finite substructure of $\str{U}$ and $\phi$ an equivalence formula on copies of $\str{A}$.
Then $\phi$ has either one or infinitely many equivalence classes.
\end{prop}
\begin{proof}
Assume to the contrary that $\phi$ is an equivalence formula on copies of $\str{A}$ which defines $k$ equivalence classes, $\infty>k>1$.
 It is well known that from ultrahomogeneity we can assume that $\phi$ is quantifier free.
Consequently, there is a finite substructure $\str{B}\subseteq \str U$ that contains two such copies of $\str{A}$ which belong to two different equivalence classes of $\phi$.
Partition $\str U\choose \str A$ to $k$ equivalence classes of $\phi$. Since $\phi$ is quantifier free, we get that there is no $\widetilde{\str B}\in {\str U \choose \str B}$ such that $\widetilde{\str B}\choose \str A$ would lie in a single equivalence class. Clearly, this implies that there is no $\str C\in \mathcal K$ such that $\str{C}\longrightarrow (\str{B})^\str{A}_k$, hence contradicting the Ramsey property.
\end{proof}
\begin{remark}
Recall that tuples $\vv{a}$ and $\vv{b}$ have the same \emph{strong type} if
$\phi(\vv{a},\vv{b})$ holds for every equivalence formula $\phi$ with finitely many
equivalence classes.
By the above observation it follows that the automorphism group of the \Fraisse{}  
limit of a Ramsey class must also fix strong types (such automorphisms are considered, for example, in~\cite{Ivanov1997}).
\end{remark}

For a given equivalence formula $\phi$ with finitely many equivalence classes
it is possible to lift the language by explicitly adding relations representing the
individual equivalence classes. This will be demonstrated on two examples in
this section.

Our first example is  a simple class with a perhaps surprising Ramsey lift.
The important property of this example is that equivalences are definable on pairs of vertices rather than singletons
(for which we have already discussed $\{1,3\}$-metric spaces as an example).

Consider structures with a single quaternary relation $\rel{}{\mathbb E}$. We say that
a structure $\str{A}$ is a \emph{fat bipartite graph} if there exists a
bipartite graph $G=(V,E)$ with $V={A\choose 2}$ and
$$(a,b,c,d)\in \rel{A}{\mathbb E} \hbox{ if and only if } a\neq b, c\neq d\hbox{, and } \left\{\{a,b\},\{c,d\}\right\}\in E.$$

It is easy to see that an $\omega$-categorical universal fat bipartite graph $\str{U}_{\mathcal F\mathcal B}$ can be constructed
by assigning bipartitions to pairs at random and producing a random bipartite graph
spanning these partitions. There is an equivalence $\sim$ defined on the unordered pairs of vertices of $\str{U}_{\mathcal F\mathcal B}$ as follows:
$\{u,v\}\sim \{u',v'\}$ if and only if they are connected by a fat path of length two.  By Proposition~\ref{prop:imaginary}
we know that every Ramsey lift will thus have a binary relation denoting the bipartition.
Consequently, we can introduce binary relation $\rel{}{\mathbb{L}}$ (denoting the class of bipartition) explicitly into our lifted language
which yields the following:
\begin{theorem}
The class $\mathcal F\mathcal B$ of all finite fat bipartite graph has the following precompact Ramsey lift $\mathcal F\mathcal B^+$ with the lift property:

The language $L^+_\mathcal{FB}$ is extended by two binary relations $\leq$ and $\rel{}{\mathbb{L}}$, where $\leq$ is a free ordering of vertices and $\rel{}{\mathbb{L}}$ denotes one of the two bipartitions of pairs.
\end{theorem}
\begin{proof}
The class  $\mathcal F\mathcal B^+$ is locally finite subclass of all finite ordered
$L^+_\mathcal{FB}$-struc\-tures: If a structure $\str{A}$ has no $\mathcal F\mathcal B^+$-completion then it either contains a tuple $(a,a)$ in $\rel{}{\mathbb{L}}$ or analogously a tuple with incorrectly duplicated vertices in $\rel{A}{\mathbb E}$ or a tuple $(a,b,c,d)\in \rel{A}{\mathbb E}$ such that either $(a,b),(c,d)\notin \rel{A}{\mathbb{L}}$ or $(a,b),(c,d)\in \rel{A}{\mathbb{L}}$.
\end{proof}
(In fact, Theorem~\ref{thm:mainstrong} is not necessary here, the Ramsey property also follows by an application of Theorem~\ref{thm:NR}.)
What is interesting about this lift? If one considers the shadow of $\str{U}_{\mathcal F\mathcal B}$ in the language containing only the relation $\rel{}{\mathbb{L}}$,
it will form the Rado graph. 
This shows that the precompact lifts with the lift property may give rise to rich structures and not only to orders and unary relations (as in most cases mentioned so far).
It is easy to generalise this example further (giving fat analogies to Corollary~\ref{cor:Hcolor}, forbidding a homomorphism from a graph in the language $\rel{}{L}$, or introducing a fat linear order as in Theorem~\ref{thm:fatorder}).

As our second example, consider structures with a single ternary relation $\rel{}{\mathbf E}$. We say that a structure $\str{A}$
is a \emph{neighbourhood bipartite graph} if for every vertex $v\in A$ the digraph $\str{G}_v$ is a bipartite graph, where $\str{G}_v$ is defined on the vertex
set $\str{A}\setminus\{v\}$ and $(a,b)\in E_{\str{G}_v}$ if and only if $(v,a,b)\in \rel{A}{\mathbf E}$.

The class of all neighbourhood bipartite graphs is not a locally finite subclass of the class of all relational structures with a single ternary relation $\rel{}{\mathbf E}$, since there is a definable equivalence on 2-tuples of vertices of the generic
neighbourhood bipartite graph: $(u,v)\sim (u,v')$ if $v$ and $v'$ are connected by
a path of length two in $\str{G}_u$.  This time, however, the number of
equivalence classes is not finite and we cannot apply
Proposition~\ref{prop:imaginary} directly.

It is easy to observe that if $\str U$ is a homogeneous structure with a Ramsey age, then so is $\str U^+$ which differs from $\str U$ by distinguishing one vertex by a special unary relation (\ie{}\ the automorphism group of $\str{U}_{\mathcal N\mathcal B}$
is forced to fix the vertex)~\cite{Bodirsky2015}.
Let $\str{U}_{\mathcal N\mathcal B}$ be the $\omega$-categorical universal neighbourhood bipartite graph. If we distinguish one vertex $u$ then the equivalence $v\sim_u v'\iff v,v'\text{ are connected by a path of length 2 in }\str G_u$ become definable in the sense of Proposition~\ref{prop:imaginary}.
Consequently, every Ramsey lift of $\str{U}_{\mathcal N\mathcal B}$ must already explicitly represent one of the two equivalence classes (in the model-theoretic setting this corresponds to the elimination of imaginaries with a parameter). We can eliminate these imaginaries at once by means of a binary relation:

\begin{theorem}
The class $\mathcal N\mathcal B$ of all finite neighbourhood bipartite graph has the following precompact Ramsey lift $\ordclass{\mathcal {NB}}$ with the lift property:

The language is extended by two binary relations $\leq$ and $\rel{}{\mathbb L}$. The order $\leq$ is free. Relation $\rel{}{\mathbb L}$
has the property that for every vertex $v$ the set of vertices connected to $v$ by $\rel{}{\mathbb L}$ is one of the bipartitions of the graph
$\str{G}_v$.
\end{theorem}
This time the shadow of the Ramsey lift of the generic neighbourhood bipartite graph produces the generic digraph.
We further develop ``neighbourhood structures'' in Section~\ref{sec:manyorders}. For this we however need to deal with
closures and functions.

\subsection {Ramsey classes with closures}
\label{sec:examplesclosure}
A lot of more complex Ramsey classes contain equivalences defined on vertices (and even tuples of vertices) which are not explicitly present in the language and which have infinitely many equivalence classes (\eg{}\ $S$-metric spaces with jump numbers~\cite{Sauer2013} or bowtie-free graphs~\cite{Hubivcka2014}
which we shall handle in Section~\ref{sec:metric2} and~\ref{sec:CSS} respectively). This means that one cannot assign labels to them and use Theorem~\ref{thm:mainstrong} in a relational language. In such situations one can make use of an interpretation in a language with functions. 
We first explain this on a simple example and later apply this technique in more complex situations.

\medskip

Let $\sim$ be an equivalence on set $A$. To every equivalence class $E$ of $\sim$ we assign a vertex $v_E$ and define a choice function $\func{}{}\colon A\to A$ which maps every vertex $v$ to $v_E$ where $E$ is the $\sim$-equivalence class containing $v$. What we obtain is a structure $\str{A}(\sim)$ in the language $L_{\mathcal {PE}}$ consisting from a unary function $\func{}{}$.
The class of all structures $\str{A}(\sim)$ is denoted by $\mathcal {PE}$. Explicitly,
class $\mathcal{PE}$ contains all finite $L_\mathcal{PE}$-structures $\str{A}$ where for every $u\in A$
it holds that $(u)\in\dom(\func{A}{})$ and $\func{A}{}(\func{A}{}(u))=\func{A}{}(u)$.

$\mathcal{PE}$ stands for \emph{pointed equivalences}: in every equivalence class we selected a special vertex (thus obtaining a ``pointed set'').
Clearly, embeddings of pointed equivalences $\str{A}(\sim_1)$ to $\str{A}(\sim_2)$ correspond to embeddings of $\sim_1$ into $\sim_2$ (as relations) with the additional property that special vertices are mapped to special vertices. Thus we have an interpretation of the class of equivalences and their embeddings. Combining this with Theorem~\ref{thm:mainstrong} and Theorem~\ref{thm:models} we obtain:
\begin{theorem}
\label{thm:pointed}
The lift $\ordclass{\mathcal{PE}}$ of $\mathcal{PE}$ which adds a free order on vertices is a Ramsey class.
\end{theorem}
\begin{proof}
Denote by $L^+_\mathcal{PE}$ the lift of language $L_\mathcal{PE}$ adding binary relation $\leq$.
Apply Theorem~\ref{thm:models} to show that $\ordclass{\Str}(L^+_\mathcal{PE})$ is Ramsey.
To apply Theorem~\ref{thm:mainstrong} we verify that $\ordclass{\mathcal{PE}}$ is a locally finite subclass of $\ordclass{\Str}(L^+_\mathcal{PE})$.
This follows easily for $n=1$: Every structure $\str{A}\in \ordclass{\Str}(L^+_\mathcal{PE})$ such that every irreducible substructure of $\str{A}$ is in $\ordclass{\mathcal{PE}}$
is itself in $\ordclass{\mathcal{PE}}$.
\end{proof}
Note that to obtain a lift with the lift property it is necessary to order vertices in
a convex way where every equivalence class forms an interval and in each such interval the special vertex must be the first one.

Theorem~\ref{thm:pointed} is only the tip of the iceberg and there are many applications of this technique. We give several examples in the next section.

\subsubsection{Unary functions (only) are easy}
\label{sec:unaryfunctions}

First we consider unary functions (of which Theorem~\ref{thm:pointed} is a particular example).  Despite the seeming complexity (as exemplified
by~\cite{Sokic2016}) the basic result here is deceptively easy and can be formulated
as follows. 

Consider a structure $\str{A}$ with (unary) function symbols, $\str{A}=(A,\func{A}{1}, \func{A}{2},\allowbreak \ldots, \func{A}{m})$,
where each $\func{}{i}$ is a function $A\to A$.
Such structures represent the most natural example of a
class with a closure. For example, given a structure $\str{B}=(\{u,v\},\func{B}{1})$ where $\func{B}{1}(u)=v$ and $\func{B}{1}(v)=v$,
the closure of $u$ in $\str{B}$ is $\str{B}$ itself, there is no structure induced by $\str{B}$ on $\{u\}$.

Denote by $\mathcal F^m_1$ the class of all finite structures with $m$ unary functions. As usual, by ordered structures with $m$ unary functions we will mean structures from $\mathcal F^m_1$ together with a linear order on vertices.
The class of all finite ordered structures with $m$ unary functions will be denoted by $\ordclasssup{\mathcal F}{1}^m$. Recall that, given a vertex $v$ of a structure $\str{A}$, its \emph{vertex closure} is the smallest substructure of $\str{A}$ containing~$v$. 

The Ramsey property of $\ordclasssup{\mathcal F}{1}^m$ follows by a simple direct
argument:

\begin{theorem}
\label{thm:unaryfunc}
Let $\str{A}$ be a finite ordered structure with $m$ unary functions and let
$\str{B}$ be a finite or countably infinite ordered structure with $m$ unary functions.
If $\str{B}$ is infinite, assume moreover that $\leq_\str{B}$ is isomorphic to the order of natural numbers.
Then there exists an ordered structure with $m$ unary functions $\str{C}$ such that
$\str{C}\longrightarrow (\str{B})^\str{A}_2$.

Moreover if $\str{B}$ is finite, then
$\str{C}$ is finite, too. If all vertex closures of vertices of $\str{B}$ are finite, then $\str{C}$
is countable.
\end{theorem}
\begin{proof}
Fix ordered structures with $m$ unary functions $\str{A}$ and $\str{B}$.
Without loss of generality assume that $B=\{1,2,\ldots, b\}$ or $B=\mathbb N$ and is ordered naturally by $\leq_\str{B}$.
 Obtain $N\longrightarrow (b)^{\vert A\vert }_2$ by the Ramsey Theorem.
Consider a lifted language adding a unary relation $\rel{}{i}$ for every $1\leq i\leq N$.

Now construct a structure $\str{P}$ as follows: For each $b$-tuple $\vv{v}=(v_1,v_2,\ldots v_b)$ of elements of $\{1,2,\ldots,N\}$ such that $v_1< v_2<\cdots< v_b$ add a disjoint copy $\str{B}_{\vv{v}}$ of $\str{B}$ to $\str{P}$ and for every $n$, $1\leq n\leq b$, put the $n$-th smallest vertex of $\str B_{\vv{v}}$ into $\rel{P}{v_n}$. Order vertices of $\str{P}$ linearly such that for every $1\leq i< j\leq N$ every vertex $v\in \rel{P}{i}$ is before every vertex $v'\in \rel{P}{j}$. (Note that this is essentially picture zero of the partite construction, cf. Section~\ref{sec:partiteconstruction}.)

From $\str P$, construct a structure by identifying every pair of vertices of $\str{P}$ with isomorphic vertex closures (isomorphic including the unary relations). Finally remove the unary relations and call the resulting structure $\str{C}$ (that is, $\str{C}$
is an ordered structure with $m$ unary functions.)  From the construction it follows that there is a homomorphism from $\sh(\str{P})$ to $\str{C}$ which is an embedding on every $\sh(\str{B}_{\vv{v}})$.

It is easy to verify that $\str{C}\to (\str{B})^\str{A}_2$: A colouring of copies of $\str{A}$ in $\str{C}$ gives a colouring of $\vert A\vert $-tuples of $\{1,2,\ldots, N\}$ (note that there is only one copy of $\str{A}$ for every $\vert A\vert $-tuple of elements of $\{1,2,\ldots,N\}$) and the Ramsey Theorem gives a monochromatic $b$-tuple which corresponds to a copy of $\str{B}$ in $\str{P}$ and thus also to a copy of $\str{B}$ in $\str{C}$.
\end{proof}
As a consequence we obtain the Ramsey property of $\mathcal F^m_1$.
\begin{corollary}[\cite{Sokic2016}]
$\ordclasssup{\mathcal F}{1}^m$ is a Ramsey lift of $\mathcal F^m_1$.
\end{corollary}
\begin{remark}
Note that $\ordclasssup{\mathcal F}{1}^m$ does not have lift property with respect to $\mathcal F^m_1$.
If one views structures in $\mathcal F^1_1$ as oriented graphs (with edges
 pointing from $v$ to $\func{}{1}(v)$) then these graphs form a forest of ``graph trees''
oriented towards a root where the root may be an oriented cycle. To obtain the
lift property the order needs to be convex with respect to the individual
connected components, it needs to order the cycles of a given size in a unique
way and the vertices of trees need to be ordered convexly level-wise with
children of a vertex forming a linear interval~\cite{Sokic2016}.
  The lift property becomes even more involved for classes $\mathcal F^m_1$, $m>1$.
 A precise description of this lift will appear in~\cite{Evans2}. 
\end{remark}

\begin{remark}
Unary functions can be seen as a generalisation of structures
with unary relations: Every unary relation $\rel{}{}$ can be represented by a unary function $\func{}{}$ and two artificial vertices $0$, $1$ by putting $\func{}{}(v)=1$ if $(v)\in \rel{}{}$ and $f(v)=0$ otherwise.  This gives an intuition why the Ramsey property of
classes with unary functions follows by a simple argument and why this argument cannot be easily generalised to non-unary functions.
Still, the proof of Theorem~\ref{thm:unaryfunc} can be seen as a basic case of the partite construction where the partite lemma
is replaced by an identification of all copies of $\str{A}$ with a given projection to one.
\end{remark}

In a way, structures with unary functions are a misleading (easy) example.
The proof of Theorem~\ref{thm:unaryfunc} should be contrasted with the situation for function symbols of higher arities where we need our main theorem. 

\subsubsection {$S$-metric spaces}
\label{sec:metric2}
We are now ready to further develop results of Section~\ref{sec:metric} and
complete the study of Ramsey properties of general $S$-metric spaces (\ie{}\ even for sets $S$ containing jump numbers).
This generalises results of~\cite[Theorem 25]{The2010} where the Ramsey property of $S$-metric spaces was shown for all sets $S$ containing at most
3 distances. This also confirms the conjecture stated in~\cite{The2010} that every $S$-metric space with $S$ finite has a precompact Ramsey lift.

We use the notation introduced in Section~\ref{sec:metric} (in particular Definitions~\ref{defn:smetric} and \ref{defn:lmetric} which introduce classes $\mathcal M_S$ and $S$-metric graphs, the 4-values condition and the operation $\oplus_S$).
Our analysis is based on (and refines)~\cite{Sauer2012} which gives a family of definable
equivalences on $S$-metric spaces when $S$ contains jump numbers. The following definition is a generalisation of a definition from~\cite{Sauer2012} for not-necessarily finite sets.

\begin{definition}[\cite{Sauer2012}]
Let $S\subseteq \mathbb R_{>0}$ be a subset satisfying the 4-values condition where $\oplus_S$ is well-defined.
A~\emph{block} of $S$ is any inclusion maximal subset $B$ of $S$ satisfying the 4-values condition
that has no jump number (see Definition~\ref{defn:jump}).
\end{definition}
In other words, blocks are maximal sets on which $\oplus_S$ is archimedean. It is shown in~\cite{Sauer2012} that finite $S$ satisfying the 4-values condition can be decomposed to mutually disjoint blocks and that for every block $B\subseteq S$ other than the block containing the largest numbers, the value of $\max(B)$ is defined and it is a jump number. In turn, this gives equivalences on $S$-metric spaces:

\begin{definition}[\cite{Sauer2012}]
Let $S\subseteq \mathbb R_{>0}$ be a subset satisfying the 4-values condition, let $\str{A}$ be an $S$-metric space and let $B$ be a block of $S$. We define a \emph{block equivalence $\sim_B$}
on vertices of $\str{A}$ by putting $u\sim_B v$ whenever there is $b\in B$ such that $d(u,v)\leq b$.
\end{definition}
Note that if $\max(B)$ is defined, one can always put $b=\max(B)$. It is easy to see that $\sim_B$ is indeed an equivalence relation. By Proposition~\ref{prop:imaginary}
it is thus necessary to lift $\mathcal{M}_S$ to represent these
equivalences explicitly.

In this section we aim to show that the following class is Ramsey.

\begin{definition}[\cite{The2010}]
\label{defn:Js}
Given $S\subseteq \mathbb R_{>0}$ satisfying the 4-value condition on which the $\oplus_S$ operation is defined.
We say that for an ordered $S$-metric space $\str A$ the order $\leq_\str{A}$ is \emph{convex with respect
to block equivalences} if every equivalence class of every $\sim_B$, for every block $B$ of $S$ is an interval of $\leq_\str{A}$.
An ordered $S$-metric space whose order is convex with respect to block equivalences is also called a \emph{convexly ordered $S$-metric space}.

The class of all finite convexly ordered $S$-metric spaces will be denoted by $\ordclasssup{\mathcal M}{S}$.
\end{definition}

We first focus on finite $S\subseteq \mathbb R_{>0}$ satisfying the 4-values
condition. In this setting the operation $\oplus_S$ is always well-defined (and associative) and
every block contains a maximal element which is either a jump number or
$\max(S)$. We will denote by $J_S$ the set of all jump numbers. For every jump number $j\in J_S$ we will denote by $B_j$ the corresponding block
containing $j$.

Now we are ready to describe a construction extending $S$-metric spaces by new vertices
representing all definable equivalences which will lead to a class where Theorem~\ref{thm:mainstrong} can be applied.
For this we lift the language $L_S$ to $L^+_S$ by adding the order $\leq$ and unary functions $\func{}{j}$ for every
$j\in J_S$.
\begin{definition}
\label{defn:L}
For a given convexly ordered metric space $\str{A}\in \ordclasssup{\mathcal M}{S}$ we denote by $L(\str{A})$ the $L^+_S$-structure defined by means of the following procedure:
\begin{enumerate}[label=(\roman*)]
\item\label{step1} For every $j\in J_S$ enumerate the equivalence classes of $\sim_{B_j}$ in $\str{A}$ as $E^1_j, E^2_j,\allowbreak \ldots, E^{n_j}_j$ such that for $v\in E^i_j$ and $v'\in E^{i'}_j$ we have $v<_\str{A} v'$ whenever $1\leq i<i'\leq n_j$.
\item\label{step2} For every $j\in J_S$ and $1\leq i\leq n_j$ add a new vertex $v^i_j$.
\item\label{step3} For every $j\in J_S$, $1\leq i\leq n_j$ and $u\in E^i_j$ put $\nbfunc{L(\str{A})}{j}(u)=v^i_j$ and $(u,v^i_j),(v^i_j,u)\in \nbrel{L(\str{A})}{j}(u)$ (so $d_{L(\str{A})}(u,v^i_j)=j$).
\item\label{step4} Complete the remaining distances by the shortest path completion (Proposition~\ref{prop:sauer1}).
\item\label{step5} Extend the order of $\str{A}$ to $\leq_{L(\str{A})}$ by considering every $j\in J_S$ (from smallest to largest) and for every $1\leq i\leq n_j$ putting vertex $v^i_j$ immediately after the last vertex of $E^i_j$.
\end{enumerate}
\end{definition}
Thus every $\sim_j$ equivalence class $E^i_j$ in $L(\str{A})$ has a unique vertex $v^i_j\notin A$ such that all other vertices are linked to it by means of
functions $\nbfunc{L(\str{A})}{j}$.  We will call this special vertex the \emph{closure} vertex corresponding to $E^i_j$.
All vertices in $A$ are \emph{original} vertices.

Clearly vertices $v^i_j$ may be regarded as added imaginaries and we still consider $L(S)$ as an $L^+_S$-lift of $\str{A}$ despite the fact that the $L_S$-shadow of $L(\str{A})$ contains added vertices $v^i_j$.

To simplify the notation below we will denote the trivial equivalence by $\sim_{B_0}$, that is, $u\sim_{B_0} j\iff u=j$.
\begin{lemma}
\label{lem:samedist}
Let $S\subseteq \mathbb R_{>0}$ be a finite subset satisfying the 4-values condition, let $\str{A}\in \mathcal M_S$ be an $S$-metric space, $j_1,j_2\in
J_S\cup \{0\}$, let $E_1$ be a $\sim_{B_{j_1}}$ equivalence class (in $A$) and let $E_2$ be a 
$\sim_{B_{j_2}}$ equivalence class such that $E_1\cap E_2=\emptyset$.
Then there exists $\ell$ such that for every $u\in E_1$ and $v\in E_2$ it holds
that $\ell=j_1\oplus_S d(u,v)\oplus j_2$.
\end{lemma}
\begin{proof}
Assume to the contrary that there are $\str{A}$, $j_1$, $j_2$, $E_1$ and $E_2$ as
in the statement of the lemma and moreover there are $u,u'\in E_1$, $v,v'\in E_2$ such
that $j_1\oplus_S d_\str{A}(u,v)\oplus_S j_2<j_1\oplus_S d_\str{A}(u',v')\oplus_S j_2$.

Construct a metric space $\str{B}$ on vertices $a,b,c,d$ such that $d_\str{B}(a,b)=j_1$,
$d_\str{B}(c,d)=j_2$, $d_\str{B}(a,c)=d_\str{B}(u',v')$ and the other distances are given  
by the shortest path completion (Proposition~\ref{prop:sauer1}). Observe that $d_\str{B}(b,d)=j_1\oplus_S d_\str{B}(u',v')\oplus_S j_2$.

Because $\mathcal M_S$ is a strong amalgamation class, it contains an
amalgamation $\str{C}$ of $\str{A}$ and $\str{B}$ unifying vertices $a,u'$ and $c,v'$.
Here the distance of $b$ and $u$ at most $j_1$ and the distance of $d$ and $v$
is at most $j_2$. Now $d_\str{C}(b,d)=j_1\oplus_S d_\str{C}(u',v')\oplus_S j_2
>j_1\oplus_S d_\str{C}(u,v)\oplus j_2\geq d_\str{C}(b,u)\oplus_S d_\str{C}(u,v)\oplus_S
d_\str{C}(v,d)$ and thus the cycle $b,u,v,d$ is non-metric.  A contradiction.
\end{proof}
\begin{lemma}
\label{lem:Smetricemb}
Let $S\subseteq \mathbb R_{>0}$ be a set satisfying the 4-values condition on which the $\oplus_S$ operation is
defined and let $\str{A}\in \ordclasssup{\mathcal M}{S}$. Then the $L_S$-shadow of $L(\str{A})$ is a metric
space (which includes both the original and closure vertices) and the order $\leq_{L(\str{A})}$ is convex with respect to block equivalences.
 Moreover, for every $\str{B}\in \ordclasssup{\mathcal M}{S}$ it holds that every embedding $f\colon\str{B}\to \str{A}$ extends uniquely
to an embedding $L(\str{B})\to L(\str{A})$.
\end{lemma}
\begin{proof}
The first part is a consequence of Proposition~\ref{prop:sauer1} and the fact
that step~\ref{step3} of Definition~\ref{defn:L} did not introduce any non-metric triangles or cycles.

To see the second part, consider convexly ordered $S$-metric spaces $\str{A}$ and $\str{B}$ and, for simplicity,
assume that $\str{B}$ is a substructure of $\str{A}$. Put $\str{A}^+=L(\str{A})$ and $\str{B}^+=L(\str{B})$.
It is easy to see that every new vertex of $\str B^+$ introduced in step~\ref{step2} of Definition~\ref{defn:L} 
has a unique corresponding vertex of $\str A^+$ introduced in step~\ref{step2} of Definition~\ref{defn:L} and that
the order of $\str{A}$ and $\str{B}$ was extended same way.  We can
thus define $f\colon B^+\to A^+$ to be the identity on $B$ and to map closure vertices of $\str{B}^+$ to the
corresponding closure vertices of $\str{A}^+$. It remains to verify that $f$ preserves distances
introduced in step~\ref{step3} of Definition~\ref{defn:L}. This is a consequence of Lemma~\ref{lem:samedist}.
\end{proof}

\begin{definition}
\label{defn:Mp}
Let $S\subseteq \mathbb R_{>0}$ be a finite set satisfying the 4-values condition.
Denote by $\ordclasssup{\mathcal M}{S}^+$ the class of all finite $L^+_S$-structures $\str{A}$ satisfying:
\begin{enumerate}
 \item\label{cond:1} The $L_S$-shadow of $\str{A}$ is an $S$-metric space.
 \item\label{cond:2} $\leq_\str{A}$ is a linear order convex with respect to block equivalences.
 \item\label{cond:3} For every $j\in J_S$ and every $\sim_{B_j}$-equivalence class $E$ whose last vertex (in the order $\leq_\str{A}$) is $v$ it holds
    that for every $u\in E$, $u\neq v$ the distance of $u$ and $v$ is $j$ and $\func{A}{j}(u)=v$. $\func{A}{j}(v)$ is undefined.
We will again call such $v$ the \emph{closure} vertex of the equivalence class $E$.
\end{enumerate}
\end{definition}

\begin{lemma}
Let $S\subseteq \mathbb R_{>0}$ be a finite set satisfying the 4-value condition. Then $\ordclasssup{\mathcal M}{S}^+$ is a lift of $\ordclasssup{\mathcal M}{S}$ with the strong amalgamation property.
\end{lemma}
\begin{proof}
It is easy to see that for every $\str{D}\in \ordclasssup{\mathcal M}{S}$ it holds that $L(\str{D})\in \ordclasssup{\mathcal M}{S}^+$.

We verify the strong amalgamation property.
Consider $\str{A},\str{B}_1,\str{B}_2\in \ordclasssup{\mathcal M}{S}^+$ such that $\str{A}$ is a substructure of both $\str{B}_1$ and $\str{B}_2$.
For simplicity assume that $B_1\cap B_2=A$.
The strong amalgamation $\str{C}$ of $\str{B}_1$ and $\str{B}_2$ over $\str{A}$ is then constructed from the free amalgamation 
by completing all missing distances by the shortest path completion:
For $v_1 \in B_1\setminus A$ and $v_2\in B_2\setminus A$ we put $d_\str{C}(v_1,v_2)=\min_{c\in A}(d_{\str B_1}(v_1,c)\oplus_S d_{\str B_2}(c,v_2))$.

Consider jump number $j\in J_S$ and a $\sim_{B_j}$-equivalence class $E$ in $\str{C}$
such that there are vertices $v_1\in E\cap B_1$ and $v_2\in E\cap B_2$.
Because distances between vertices in $\str{C}$ are completed by means of the $\oplus_S$ operation we know that it
means that there is a vertex $c\in \str{A}$ such that $d_{\str{B}_1}(v_1,c)\leq j$ and $d_{\str{B}_2}(v_2,c)\leq j$.
We can also choose $c\in \str{A}$ to be the unique closure vertex of the  $\sim_{B_j}$ equivalence class  of $\str{B}_1$
containing $v_1$ and of the $\sim_{B_j}$ equivalence class  of $\str{B}_2$ containing $v_2$. Consequently, 
it is possible to complete $\leq_\str{C}$ to a linear order convex with respect to block equivalences
satisfying the additional assumption about the closure vertex being last in its equivalence class.
\end{proof}

\begin{remark}
It may seem more natural to define $\ordclasssup{\mathcal M}{S}^+$ as the class of all $L(\str{A})$, $\str{A}\in \ordclasssup{\mathcal M}{S}$.
This class is however not hereditary and we would not be able to apply Theorem~\ref{thm:mainstrong} directly.

It may also be tempting to not define distances on closure vertices (that is, omit step \ref{step4} of Definition~\ref{defn:L}).
This would however lead to problems with the amalgamation property.
To see that, let $S=\{1,3,5\}$ and consider metric spaces $\str{B}_1$ consisting of two vertices $B_1=\{u,v\}$ in distance 3
and $\str{B}_2$ consisting of two vertices $B_2=\{u',v'\}$ in distance 5.
Omitting step~\ref{step4} of Definition~\ref{defn:L} would make it possible to consider amalgamation of
$L(\str{B}_1)$ and $L(\str{B}_2)$ identifying $u$ with $u'$ and $\nbfunc{L(\str{B}_1)}{1}(v)$ with $\nbfunc{L(\str{B}_2)}{1}(v')$.
However, identifying the closure vertices for $v$ and $v'$ means that in this amalgamation $v\sim_{B_1} v'$ and thus $d(v,v')=1$ which gives a non-metric triangle.
By defining the additional distances by means of Proposition~\ref{prop:sauer1} we solve this problem because: 
\begin{align*}
d_{L(\str{B}_1)}(u, \nbfunc{L(\str{B}_1)}{1}(v))&=3,\\
d_{L(\str{B}_2)}(u', \nbfunc{L(\str{B}_2)}{1}(v'))&=5.
\end{align*}
\end{remark}

The following definition and technical lemma are the key to obtaining a locally finite
description of $\mathcal{M}_S$ needed for Theorem~\ref{thm:mainstrong}:

\begin{definition}
Let $S\subseteq \mathbb R_{>0}$ be a finite set satisfying the 4-values condition and
let $\str{P}$ be a path with distances in $S$. Denote by $B(\str{P})$ the block
of $S$ containing the maximal distance of an edge in $\str{P}$.
Let  $\str{P}'$ be any $S$-metric path.  We say that
$\str{P}'\preceq_S \str{P}$ if all the distances in $\str{P}'$ are
bounded from above by a member of $B(\str P)$ (for example by $\max(B(\str{P}))$).
\end{definition}

\begin{figure}
\centering
\includegraphics{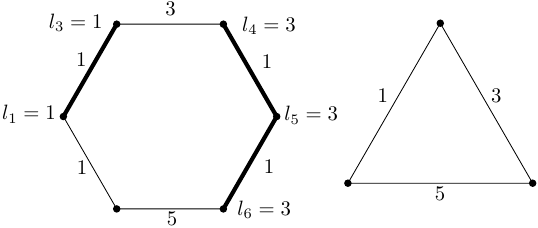}
\caption{The family of unimportant paths in a non-$\{1,3,5\}$-metric cycle (left), the non-$\{1,3,5\}$-metric
cycle created after the concatenation of unimportant paths (right).}
\label{fig:nonSmetric}
\end{figure}
\begin{lemma}
\label{lem:Sequivalence}
Let $S\subseteq \mathbb R_{>0}$ be a finite set satisfying the 4-values
condition and let $\str{C}$ be a non-$S$-metric cycle. Then there exist 
disjoint paths $\str{P}^i$, $1\leq i\leq p$, in $\str{C}$ such that
the cycle created by identifying all vertices of each path into a single vertex is a non-$S$-metric
cycle with at most $\vert S\vert $ vertices, and moreover every cycle created from
$\str{C}$ by replacing each of the paths $\str{P}^i$ by arbitrary path
$\str{P}'^i$, $\str{P}'^i\preceq_S \str{P}^i$ is non-$S$-metric.
\end{lemma}
We will call the paths $\str{P}^i$ \emph{unimportant}.
An example is given in Figure~\ref{fig:nonSmetric}.
\begin{proof}
Let $\str{C}$ be a non-$S$-metric cycle
with $n$ vertices.  By Proposition~\ref{prop:sauer1} we know that $\str{C}$
contains a pair of vertices whose distance is larger than the $S$-length of the path connecting them.

Enumerate the vertices of $C$ as $v_1,v_2,\ldots,v_n$
such that $$d_\str{C}(v_1,v_n)>d_\str{C}(v_1,v_2)\oplus_S d_\str{C}(v_2,v_3)\oplus_S\cdots \oplus_S d_\str{C}(v_{n-1},v_n).$$
For every $1< j\leq n$ denote by $$l_j=d_\str{C}(v_1, v_2)\oplus_S d_\str{C}(v_2, v_3)\oplus_S \cdots \oplus_S d_\str{C}(v_{j-1}, v_j)$$ the $S$-length of
the walk formed by the initial segment on $j$ vertices. We know that $l_j\leq l_{j+1}$ for every $1< j< n$.

We say that a path induced by $\str{C}$ on vertices $(v_j,v_{j+1},\ldots, v_k)$ is unimportant if $l_j=l_k$ and let the paths $\str{P}^i$ be all inclusion maximal unimportant paths. As a special case, if there
is only one inclusion maximal unimportant path on vertices $v_2,v_3,\ldots, v_n$,
put $\str{P}^1$ to be the path on vertices $v_2,v_3,\ldots, v_{n-1}$ so that the result of identification
is a triangle.

Because there are only finitely many values in $S$, we know that there  are
at most $\vert S\vert-1$ choices of $j$ such that $l_j<l_{j+1}$ ($\vert S\vert-1$, because we know that $l_n < d_\str{C}(v_1,v_n)\in S$).  We thus know that there
are at most $\vert S\vert-1$ pairs $v_j,v_{j+1}$ which do not belong to an unimportant
path.
 It follows that the graph created by identifying each unimportant path to a vertex
has at most $\vert S\vert $ vertices (and, because of the special case, at least three vertices).  It is also easy to verify that the resulting graph
is non-$S$-metric cycle because the $S$-length of the walk connecting $v_1$ and $v_n$
was not affected by the identifications. To prove the last part, observe that if path $\str{P}^i$ is unimportant and we replace it with a path $\str{P}'^i$ such that
$\str{P}'^i\preceq_S \str{P}^i$ then $\str P'^i$ is also unimportant.
\end{proof}
The following result is a direct analogy of Corollary~\ref{cor:finite4value} (this time with jump numbers).
\begin{lemma}
\label{lem:finite4value2}
Let $S\subseteq \mathbb R_{>0}$ be a finite set satisfying the 4-values
condition.  Then the class
$\ordclasssup{\mathcal M}{S}$ of all convexly ordered $S$-metric spaces  is a Ramsey class.
\end{lemma}
\begin{proof}
By Lemma~\ref{lem:Smetricemb} $L(\str{A})$ is a mapping lifting every $\str{A}\in \ordclasssup{\mathcal M}{S}$ to $L(\str{A})\in \ordclasssup{\mathcal M}{S}^+$ preserving substructures. Moreover every
$L_S\cup\{\leq\}$-shadow of a structure in $\ordclasssup{\mathcal M}{S}^+$ is a structure in  $\ordclasssup{\mathcal M}{S}$.
It follows that the Ramsey property of $\ordclasssup{\mathcal
M}{S}^+$ implies the Ramsey property of $\ordclasssup{\mathcal M}{S}$.
To apply Theorem~\ref{thm:mainstrong} and obtain that $\ordclasssup{\mathcal M}{S}^+$  is Ramsey
we verify
that $\ordclasssup{\mathcal M}{S}^+$ is a locally finite subclass of
$\ordclass{\Str}(L^+_S)$ which is Ramsey by Theorem~\ref{thm:models}.

\medskip

We show that for all $\str{C}_0\in \ordclass{\Str}(L^+_S)$ we can put $n=2\vert S\vert(\vert J_s\vert+1)$.
Let $\str{C}'$ be a structure with a homomorphism-embedding to $\str{C}_0$ such
that every $n$-element substructure of $\str{C}'$ has a (strong) completion $\str{C}$ in $\ordclasssup{\mathcal M}{S}^+$.
It follows that the $L_S$-shadow $\str{G}$ of $\str{C}'$ is an $S$-graph. We verify that 
$\str{G}$ can be completed to $S$-metric space by application of Proposition~\ref{prop:sauer1}.
For that we only need to verify that all cycles in $\str{G}$ are $S$-metric.
Assume, to the contrary, that there is non-$S$-metric cycle $\str{K}$ in $\str{G}$.
 Consider the family of
 unimportant paths in $\str{K}$ given by Lemma~\ref{lem:Sequivalence}. Let $\str{P}$ be an unimportant path and let $j$ be the smallest jump number of $S$ such that all distances in $\str{P}$ are at most $j$.  Then we know that there exists vertex $c\in C'$ such that $\nbfunc{\str{C}'}{j}(u)=c$ or $u=c$ for every $u\in P$. We call $c$ the \emph{common closure of the path $\str{P}$}.
Create $\str{K}'$ as the structure induced by
 $\str{C}'$ on the set of all vertices of $\str{K}$ which are not in unimportant paths, all initial and terminal vertices
 of unimportant path and the common closures of unimportant paths. This structure has at most $2\vert S\vert$ original vertices and thus at most $2\vert S\vert\vert J_S\vert$ closure vertices are added, hence it has at most $n$ vertices. By Lemma~\ref{lem:Sequivalence} there is no completion of $\str{K}'$. A contradiction.

We can thus create $\str{C}$ from $\str{C}'$ by completing all missing distances by means of the shortest walks as done in Proposition~\ref{prop:sauer1}.
It follows that $\str{C}$ satisfies condition~\ref{cond:1} of Definition~\ref{defn:Mp}.

\medskip

Next we verify that for every $j\in J_S$ and every $\sim_{B_j}$ equivalence class $E$ in $\str{C}$ it holds that $E$ contains precisely one vertex $c\in E$ such that $\func{C}{j}(c)$ is undefined and for every $v\in E$, $v\neq c$ it holds that $\func{C}{j}(v)=c$. This follows from the fact that for every $u$ and $v$ in $E$ there exists a path from $u$ to $v$ in $\str{G}$ consisting of distances at most $j$. Every such path has a common closure $c$.

Finally we complete $\leq_{\str{C}'}$ to a linear order that is convex with respect to block equivalences.
This can be done by considering each jump number $j\in J_S$ in decreasing order.  For each $j$ one can choose the relative order of $\sim_{B_j}$-equivalence classes respecting all inequalities in $\leq_\str{C}$ and the relative order of $\sim_{B_k}$-equivalence classes for $k>j$.  The resulting order will be convex and will have the property that the unique closure vertex corresponding to every equivalence class will be last.
This verifies that $\str{C}$ satisfies conditions~\ref{cond:2} and \ref{cond:3} of Definition~\ref{defn:Mp}.
We conclude that $\str{C}$ is an $\ordclasssup{\mathcal M}{S}^+$-completion of $\str{C}'$.
\end{proof}

Now we extend Lemma~\ref{lem:finite4value2} for infinite $S$ where $\oplus_S$ is defined and associative (recall that associativity is equivalent with the 4-values condition whenever $\oplus_S$ is defined):
\begin{prop}
\label{prop:Smetric}
Let $S\subseteq \mathbb R_{>0}$ be a set satisfying the 4-values
condition of which the $\oplus_S$ operation is defined.  Then the class
$\ordclasssup{\mathcal M}{S}$ of all convexly ordered $S$-metric spaces  is a Ramsey class.
\end{prop} 
\begin{proof}
Given $\str B\in \mathcal M_S$, we will find a finite $S'\subseteq S$ on which the $\oplus_{S'}$ operation is well-defined and associative such that $\str B$ uses only distances from $S'$. Lemma~\ref{lem:finite4value2} then gives Ramsey witnesses for $\str B$, which is what we need. It remains to construct such $S'$.

First observe that for every $m\in S$ the operation $\oplus_S^m$ defined as $a\oplus_S^m b = \min(m, a\oplus_S b)$ is well-defined and associative on the set $S^m=\{s\in S : s\leq m\}$ and hence $S_m$ satisfies the 4-values condition. Moreover, $S_m$ has only finitely many blocks and also has a maximum. Let $m$ be the largest distance occurring in $\str B$ and let $S'$ be the subset of $S$ consisting of distances which can be obtained as $\oplus_S^m$-sums of all finite nonempty sequences of distances occurring in $\str B$. Clearly, $\oplus_S^m$ is well-defined and associative on $S'$ and $S'$ contains all distances occurring in $\str B$, hence we only need to observe that $S'$ is finite. As a consequence of Lemma~\ref{lem:archimedean} we get that the smallest block of $S'$ is finite. 
To see that the second smallest block of $S'$ is also finite, one can observe that by associativity every value that is in the $\oplus_S^m$-closure of values in the first two blocks of $S'$ can be written as $(s_1\oplus_S^m s_2\oplus_S^m\cdots\oplus_S^m s_k)\oplus_S^m (s'_1\oplus_S^m s'_2\oplus_S^m\cdots\oplus_S^m s'_\ell)$ where $s_1,s_2,\ldots, s_k$ are distances in $\str B$ which all belong to the first block of $S'$ and $s'_1,s'_2,\ldots, s'_\ell$ are distances in $\str B$ which all belong to the second block of $S'$. Because the first block of $S'$ is finite, there are only finitely many possible values of $(s_1\oplus_S^m s_2\oplus_S^m\cdots\oplus_S^m s_k)$. Applying Lemma~\ref{lem:archimedean} again we get that there are only finitely many possible values of $(s'_1\oplus_S^m s'_2\oplus_S^m\cdots\oplus_S^m s'_\ell)$ and thus only finitely many values in the second block of $S'$. 
By induction, we get that all of the finitely many blocks of $S'$ are finite and thus $S'$ is finite.
\end{proof}

This result covers some countable non-closed sets $S$ (such as positive rationals).
For closed sets we can characterise all $S$-metric Ramsey classes:
\begin{theorem}[Characterisation of Ramsey lifts of $S$-metric spaces for closed $S$]
\label{thm:Smetric}
Let $S$ be a set of positive reals such that $S\cup\{0\}$ is closed. Then the following conditions are equivalent:
\begin{enumerate}
 \item\label{Smetric1} $S$ satisfies the 4-values condition,
 \item\label{Smetric2} $\mathcal M_S$ has the strong amalgamation property,
 \item\label{Smetric3} $\mathcal M_S$ has the amalgamation property, and
 \item\label{Smetric4} the class $\ordclasssup{\mathcal M}{S}$ of all convexly ordered $S$-metric spaces is Ramsey.
\end{enumerate}
\end{theorem}
\begin{proof}
$\ref{Smetric1}\iff \ref{Smetric2}$ by Corollary~\ref{cor:sauer2}.
Clearly $\ref{Smetric2}\implies \ref{Smetric3}$. To see that
$\ref{Smetric3}\implies \ref{Smetric2}$ consider $S$ which fails to satisfy the
4-values condition for $a,b,c,d$ and $x$. Assume to the contrary that $\mathcal M_S$ has the amalgamation property. It follows that the amalgamation of
triangles with distances $a$--$b$--$x$ and $c$--$d$--$x$ over the edge of distance $x$ must
identify vertices. To make this possible, it must hold that $a=c$ and $b=d$, bud then the 4-values condition is trivially satisfied by putting $y=\min(a,b)$, which is a contradiction.

$\ref{Smetric1}\implies \ref{Smetric4}$ follows by a combination of Theorem~\ref{thm:sauer2} and Proposition~\ref{prop:Smetric}.
Finally, we show $\ref{Smetric4}\implies \ref{Smetric3}$. By Proposition~\ref{prop:ramseyhomo} we know that $\ordclasssup{\mathcal M}{S}$ forms an amalgamation class.
It remains to verify that the shadow $\mathcal M_S$ of $\ordclasssup{\mathcal M}{S}$ is also an amalgamation class.
Consider $\str{A},\str{B},\str{C}\in \mathcal M_S$ such that $\str{A}$ is a substructure of both $\str{B}$ and $\str{C}$.
Then it is possible to choose a convex order of $\str{A}$ and extend it to convex orders of $\str{B}$ and $\str{C}$ and use the amalgamation property of $\ordclasssup{\mathcal M}{S}$
to obtain an amalgamation of $\str{B}$ and $\str{C}$ over $\str{A}$.
\end{proof}
Applying Theorem~\ref{thm:sauercharacterisation} we can state the results elegantly in terms of Urysohn $S$-metric spaces:
\begin{corollary}
\label{cor:Smetricfin}
Let $S$ be a set of positive reals such that $S\cup\{0\}$ is closed. Then the following conditions are equivalent:
\begin{enumerate}
 \item The class $\ordclasssup{\mathcal M}{S}$ of all convexly ordered $S$-metric spaces is Ramsey.
 \item There exists a Urysohn $S$-metric space.
\end{enumerate}
\end{corollary}

\subsubsection{Ramsey classes with multiple linear orders}
\label{sec:manyorders}
In this section we focus on the special role of the order in our constructions.

Consider the class of all finite structures with two linear orders $\leq$ and
$\preceq$ (or, equivalently, the class of permutations: the order
$\leq$ represents the original order and the order $\preceq$ represents the
permutation).  It is not obvious how to describe this class as a multiamalgamation
class (techniques of Section~\ref{sec:posets} would apply only for classes
where $\leq$ agrees with $\preceq$).
In the following proposition we show a way of effectively splitting the order
$\leq$ into multiple linear orders free to each other. We proceed more generally.

Let $\K_1$, $\K_2$ be classes of finite structures in disjoint languages
$L_1$ and $L_2$ respectively. Denote by $L$ the language $L_1\cup L_2$.  The
\emph{free interposition of $\K_1$ and $\K_2$} is the class $\K$
containing all structures $\str{A}$ such that the $L_1$-shadow of $\str{A}$ is in 
$\K_1$ and the $L_2$-shadow of $\str{A}$ is in $\K_2$.

\begin{prop}
\label{prop:interpos}
Let $L_1$ and $L_2$ be disjoint languages both containing an order (\eg{}\ $\leq_1\in L_1$ and $\leq_2\in L_2$). Let $\mathcal R_1$ be the class of all finite ordered $L_1$-structures,
let $\mathcal R_2$ be the class of all finite ordered $L_2$-structures, let $\K_1$
be an $(\mathcal R_1,\mathcal U_1)$-multiamalgamation class and let $\K_2$ be an
$(\mathcal R_2,\mathcal U_2)$-multiamalgamation class.  Then the free
interposition $\K$ of $\K_1$ and $\K_2$ is Ramsey.
\end{prop}
Because the notion of a locally finite subclass is more restrictive than the
notion of a multiamalgamation class, Proposition~\ref{prop:interpos} also holds in the case when $\K_1$ is a
locally finite subclass of $\mathcal R_1$ and $\K_2$ is a locally finite
subclass of $\mathcal R_2$.
\begin{proof}
We further extend our language $L=L_1\cup L_2$ to $L^+$ by two unary relations $\rel{}{1}$ and
$\rel{}{2}$ and two binary (closure) relations $\rel{}{U_1}$ and $\rel{}{U_2}$.

The basic idea of the proof is to split every structure $\str{A}\in \K$ into
its $L_1$-shadow $\str{A}_1$ and $L_2$-shadow $\str{A}_2$ and then to take the
``disjoint union'' of $\str{A}_1$ and $\str{A}_2$ in the language $L^+$
where vertices of $\str{A}_1$ are marked by $\rel{}{1}$ and vertices of
$\str{A}_2$ by $\rel{}{2}$. We moreover use the closure relations to describe the natural bijection
between vertices of $\str{A}_1$ and vertices of $\str{A}_2$. This
construction preserves substructures and thus the Ramsey property of such split
structures implies the Ramsey property of $\K$.

The class of such split structures is described as an $(\mathcal R,\mathcal U)$-multi{\-}amal{\-}gamation class as follows:

\begin{enumerate}
 \item The class $\mathcal R$ consists of all ordered $L^+$-structures. $\mathcal R$ is a Ramsey by Theorem~\ref{thm:models}.
 \item The closure description $\mathcal U$ consist of all pairs $(\rel{}{U_i},\str{R}^+_i)$ such that
$(\rel{}{U_i},\allowbreak \str{R}_i)\in U_1$ and $\str{R}^+_i$ is a lift of $\str{R}_i$ adding
every vertex to relation $\rel{}{1}$ and pairs $(\rel{}{U_j},\str{R}^+_j)$ such
that $(\rel{}{U_j},\str{R}_j)\in U_2$ and $\str{R}^+_j$ is a lift of $\str{R}_j$ adding
every vertex to relation $\rel{}{2}$.

Moreover, we extend the closure description to define a bijection between vertices in relation $\rel{}{1}$ and
vertices in relation $\rel{}{2}$: Every vertex in $\rel{}{1}$ has a closure defined by $\rel{}{U_1}$
and every vertex in $\rel{}{2}$ has a closure defined by $\rel{}{U_2}$.
\end{enumerate}
By combining the completion properties of $\K_1$ and $\K_2$ it easily follows that the class described
is an $(\mathcal R,\mathcal U)$-multiamalgamation class and thus by Theorem~\ref{thm:mainstrongclosures}
we get Proposition~\ref{prop:interpos}.
\end{proof}

\begin{remark}
A variant of Proposition~\ref{prop:interpos} was proved in~\cite{Bodirsky2015} for strong amalgamation classes.
However, for the first time we show that even free interpositions of classes with closures
are Ramsey. 
\end{remark}

\subsubsection{Totally ordered structures via incidence closure}
\label{sec:fatorders}
Let $\str{A}$ be a relational structure in a finite language $L$ with an order on its vertices $\leq_\str{A}$ (which is not a part of the language $L$).
Here we show how to handle such structures where moreover each relation is viewed as an ordered set and the embeddings need to preserve these orderings.

Assume that each of the sets $\rel{A}{}$, $\rel{}{}\in L$, is linearly ordered by $\leq^{\rel{}{}}_\str{A}$.
For the time being, we call $\str{A}$ together with the orderings $\leq_\str{A}$ and $\leq^{\rel{}{}}_\str{A}$, $\rel{}{}\in L$, a \emph{totally ordered structure} and denote it by $\vv{\str{A}}$.
For two totally ordered structures $\vv{\str{A}}$ and $\vv{\str{B}}$ we say that a function $f\colon\vv{\str{A}}\to\vv{\str{B}}$ is an \emph{embedding} if it is an embedding $\str{A}\to \str{B}$ which is also an embedding of all orders $\leq_\str{A}$ and $\leq^{\rel{}{}}_\str{A}$, $\rel{}{}\in R$. Explicitly, $f\colon A\to B$ is an embedding provided that the following are satisfied:
\begin{enumerate}
\item For every $\rel{}{}\in L$ it holds that $$(u_1,u_2,\ldots,u_{\arity{}})\in \rel{A}{}$$ if and only if $$(f(u_1),f(u_2),\ldots,f(u_{\arity{}}))\in \rel{B}{},$$
\item for every $\rel{}{}\in L$ it holds that
$$(u_1, u_2,\ldots,u_{\arity{}})\leq^{\rel{}{}}_\str{A}(v_1,v_2,\ldots, v_{\arity{}})$$ if and only if $$(f(u_1),f(u_2),\ldots, f(u_{\arity{}}))\leq^{\rel{}{}}_\str{B} (f(v_1), f(v_2),\ldots, f(v_{\arity{}})),$$
\item $u\leq_\str{A} v$ if and only if $f(u)\leq_\str{B} f(v).$
\end{enumerate}
Totally ordered structures are not relational structures per se. However, they can be easily interpreted as ordered relational structures and this interpretation paves the way to our approach:

For every relation $\rel{A}{}$  of arity $a=\arity{}$ of a totally ordered structure $\vv{\str{A}}$ we consider a relation  $\rel{A}{\leq}$ of arity $2a$ defined by:
$$(x_1, x_2,\ldots, x_a,y_1,y_2,\ldots, y_a)\in \rel{A}{\leq}$$ if and only if $$(x_1, x_2,\ldots, x_a)\leq^{\rel{}{}}_\str{A} (y_1, y_2,\ldots, y_a).$$
The order $\leq_\str{A}$ will be seen, as usual, as a binary relation of $\str{A}$.
The language of such interpretations is $L$ together with relations $\rel{A}{\leq}$ for every $\rel{}{}\in L$ and $\leq_\str{A}$. We will call such a relational structure $\TO(\vv{\str{A}})$.
Observe that $f\colon\vv{\str A}\to\vv{\str B}$ is an embedding of totally ordered structures if and only if it is an embedding $\TO(\vv{\str A})\to\TO(\vv{\str B})$ in the standard sense.

We will denote the extended language by $L,2L$.
Denote by $\TO(L)$ the class of all structures $\TO(\str{A})$ in the language $L,2L$.
We claim the following:
\begin{theorem}[Ramsey theorem for totally ordered structures]
\label{thm:TO}
\label{thm:fatorder}
$\TO(L)$ is a Ramsey class for every finite relational language $L$.
\end{theorem}
Before the proof of Theorem~\ref{thm:TO} let us add the following remark.
\begin{proof}[Proof of Theorem~\ref{thm:TO}]
Fix an arbitrary order $\leq_L$ of $L$.
Given a structure $\str{A}\in \TO(L)$ we describe its lift $\str{A}^+$ which we call the \emph{incidence closure} of $\str A$:
\begin{enumerate}
\item The vertex set of $\str{A}^+$ extends the vertex set of $\str{A}$ by a new vertex for every tuple in every relation. More precisely,
we add a vertex
 $v^{\rel{}{}}_{\vv{r}}$ for every $\rel{}{}\in L$ and $\vv{r}\in \rel{A}{}$.
\item The order $\leq_\str{A}^+$ extends the order of $\leq_\str{A}$ as follows:
\begin{enumerate}
\item For every $\rel{}{}\in L$ we put $v^{{\rel{}{}}}_{\vv{s}}\leq_{\str{A}^+}v^{{\rel{}{}}}_{\vv{r}}$ if and only if $\vv{s}\leq^{\rel{}{}}_\str{A} \vv{r}.$
\item For every $u\in A$ and $v\notin A$ we put $u\leq_{\str{A}^+}v$.
\item For every $\rel{}{1},\rel{}{2}\in L$ such that $\rel{}{1}<_L \rel{}{2}$ and every $v^{{\rel{}{1}}}_{\vv{s}}, v^{{\rel{}{2}}}_{\vv{r}}\in A^+$ we put 
$v^{{\rel{}{1}}}_{\vv{s}}\leq_{\str{A}^+}v^{{\rel{}{2}}}_{\vv{r}}.$
\end{enumerate}
\item For every $\rel{}{}\in L$ we add a function $\nbfunc{\str{A}^+}{\rel{}{}}$ of arity $\arity{}$ and we put $\nbfunc{\str{A}^+}{\rel{}{}}(\vv{s})=v^{{\rel{}{}}}_{\vv{s}}$  if and only if $\vv{s}\in \rel{A}{}.$
\end{enumerate}
Denote class of such lifts (with the incidence closure) by $\TO^+(L)$. After this reformulation we get that $\TO^+(L)$ is a Ramsey lift by a routine application of Theorem~\ref{thm:mainstrong}.
\end{proof}

\begin{remark}
It may seem at first glance that the natural way to prove Theorem~\ref{thm:TO} is to
show that $\TO(L)$ is a locally finite subclass of the class of all finite ordered $L,2L$-structures. This is however not the case
as there is no relationship between the orders $\leq_\str{A}$ and all the orders $\rel{A}{\leq}$.
\end{remark}
\begin{remark}
The incidence closure can be used to put an order on $n$-tuples in general. For example, the following class having a linear order on the neighbourhood of every vertex can be shown to be Ramsey by essentially the same argument:

Denote by $\mathcal Q\mathcal Q$ the class of finite structures $\str{A}$
with one binary relation $\leq_\str{A}$ and one ternary relation $\prec_\str{A}$ with the following properties:
\begin{enumerate}
 \item The relation $\leq_\str{A}$ forms a linear order on $A$, and
 \item for every vertex $a\in A$ the relation $\{(b,c):(a,b,c)\in \prec_\str{A}\}$ forms a linear oder on $A\setminus \{a\}$ (unrelated to $\leq_\str{A}$).
\end{enumerate}
$\mathcal{QQ}$ may be viewed as the class of all structures endowed with a local order on neighbourhoods.
\end{remark}

\subsection {Ramsey lifts of ages of $\omega$-categorical structures}
\label{sec:exampleslifts}
We end this paper by considering particular examples of classes which in fact provided the original motivation for this paper.
This section provides a rich spectrum of Ramsey classes defined by means of forbidden substructures.
We start with a detailed description of the Ramsey lift of the class of finite graphs with a given odd girth (\ie{}\ the size of the smallest odd cycle) and show how this particular
example fits both Theorem~\ref{thm:mainstrong} and Theorem~\ref{thm:main}. These results indicate that the case of forbidden homomorphism is well understood.
We then (in Section~\ref{sec:CSS}) turn our attention to classes defined by forbidden monomorphisms (such as forbidden subgraphs) where the situation is much more complicated even on the model-theoretic side (see \eg{}\ \cite{Cherlin2011}) and this is where we (again) have to use closures.

\subsubsection{Graphs omitting odd cycles of length at most $l$}
\label{sec:cycles}
Perhaps the simplest example of graph classes defined by means of forbidden homomorphisms is the class of all finite (undirected) graphs $\str{G}$ such that there is no homomorphism $\str{C}_l\to \str{G}$, where $\str{C}_l$ is a (graph) cycle on $l$ vertices for odd $l$, that is, the class of all finite graphs in $\Forb(\str{C}_l)$.

By Proposition~\ref{prop:ramseyhomo} we know that every Ramsey lift of such class must have
the amalgamation property. It is easy to see that the class of all finite graphs in $\Forb(\str{C}_l)$ is not
an amalgamation class for any odd $l\geq 5$ so a convenient lift is needed. We illustrated this by 
the smallest non-trivial example $l=5$ discussed already in Example~\ref{example:5cycle}. 

In full generality, an explicit homogenising (and also Ramsey) lift can be described as follows:
Fix an odd $l$. The language of graphs is extended to language $L_l$ by a linear order $\leq$
 and binary relations $\rel{}{2}, \rel{}{3}, \ldots, \rel{}{(l-1)/2}$.  
Given a finite graph $\str{G}\in \Forb(\str{C}_l)$, we define its lift $\str{G}^+$ as follows:
\begin{enumerate}
\item $\leq_{\str{G}^+}$ is an (arbitrary) linear order of $G$.
\item For every or every $1 < i\leq {l-1\over 2}$ it holds that ${u,v}\in \nbrel{\str{G}^+}{i}$ if and only if the graph distance of $u$ and $v$ is $i$. (Distance 1 is already represented by the relation $E_\str{G}$.)
\end{enumerate}
We call this lift the \emph{distance lift} of the graph $\str{G}$.

The lifted class $\K_{\str C_l}$ consists of all possible substructures of all above lifts of finite graphs in $\Forb(\str{C}_l)$.

\begin{remark}
A homogenisation of the class of all graphs in $\Forb(\str{C}_l)$ was first given by
Komj{\'a}th, Mekler and Pach~\cite{Komjath1988} (a corrected proof appears in~\cite{Komjath1999}).
As an early example of a universal graph defined by forbidden homomorphisms, it was later generalised in~\cite{Cherlin1999}, see also~\cite{Cherlin1994, Cherlin1996} for negative results. An alternative homogenisation (in the form of even-odd metric spaces) is given in~\cite{Hubicka2009}. The homogenisation presented here appears in the catalogue of metrically homogeneous
graphs~\cite{Cherlin2013} and is the only one (up to bi-definability) leading to an existentially complete
$\omega$-categorical graph universal for $\Forb(\str{C}_l)$.
\end{remark}

\begin{theorem}
\label{thm:kcycle}
The class $\K_{\str C_l}$ is a Ramsey class. Every lift $\str{A}\in \K_{\str C_l}$ can be viewed as a metric space with distances truncated by $l+1\over 2$. 
More precisely, the following function $d_{\str{A}}\colon A\times A\to \{0,1,2,\ldots {l+1\over 2}\}$ is a metric:
$$d_{\str{A}}(u,v)=
\begin{cases}
0&\hbox{ if }u=v,\\
1&\hbox{ if }(u,v)\in E_{\str{A}},\\
d&\hbox{ if }(u,v)\in \nbrel{\str{A}}{d}, 2\leq d\leq {l-1\over 2}, and\\
{(l+1)\over 2}&\hbox{ otherwise.}
\end{cases}$$
\end{theorem}
As an illustration of the versatility of our techniques we give two different proofs of Theorem~\ref{thm:kcycle}.
\begin{proof}[Proof (using Theorem~\ref{thm:mainstrong})]
We first give a strong amalgamation procedure for $\K_{\str C_l}$: Let $\str{B}_1,\str{B}_2\in \K_{\str C_l}$. Without loss of generality we can assume that both are distance lifts of graphs in $\Forb(\str{C}_l)$ and $\str{A}$ is a structure induced by both $\str{B}_1$ and $\str{B}_2$ on $A=B_1\cap B_2$.
Construct a graph $\str{G}$ as the free amalgamation of $\sh(\str{B}_1)$ and
$\sh(\str{B}_2)$ over $\sh(\str{A})$. That is, $G=B_1\cup B_2$ and $(u,v)\in
E_\str{G}$ if and only if either $(u,v)\in E_{\str{B}_1}$ or $(u,v)\in
E_{\str{B}_2}$. Denote by $\str{C}$ the distance lift of $\str{G}$. We claim that $\str{C}$ is a strong amalgamation of $\str{B}_1$ and $\str{B}_2$ over $\str{A}$.
Because for every pair of vertices $u,v\in A$ we have $d_{\str{B}_1}(u,v)=d_{\str{B}_2}(u,v)=d_\str{A}(u,v)$,
 it is easy to see that the identities are embeddings from $\str{B}_1$ and $\str{B}_2$ to $\str{C}$.  It remains to verify that $\str{G}$ does not contain any odd cycles of
length at most $l$.

Assume, to the contrary, that there exists a cycle
$\widetilde{\str{C}}_k$, $k\leq l$ odd, that is a subgraph of
$\str{G}$.  Among all such choices of $\widetilde{\str{C}}_k$ pick one
with minimal $k$.  Because neither of $\str{B}_1$ and $\str{B}_2$ has 
homomorphic images of $\str{C}_l$ we know that $\widetilde{\str{C}}_k$ contains some
vertices from $B_1\setminus A$  and some from $B_2\setminus A$. Because $\str{G}$ is a free amalgamation
and $\widetilde{\str{C}}_k$ is connected, there are also some vertices in $A\cap \widetilde{C}_k$ which form a vertex cut of $\widetilde{\str{C}}_k$.

Now consider a path
in $\widetilde{\str{C}}_k$ on vertices $v_1,v_2,\ldots, v_n$, such that $n\leq {k-1\over 2}$, $v_1,v_n\in A$
and $v_2,v_3,\ldots,v_{n-1}\notin A$. Without loss of generality assume that the whole path is contained in $\str{B}_1$.
We show that $d_\str{A}(v_1,v_n)=n$:
\begin{enumerate}
 \item Clearly $d_{\str{B}_1}(v_1,v_n)=d_{\str{B}_2}(v_1,v_n)=d_\str{A}(v_1,v_n)\leq n$.
\item Assume $d_\str{A}(v_1,v_n)<n$. In this case we create a cycle $\widetilde{\str{C}}'$ from $\widetilde{\str{C}}_k$ by replacing vertices $v_1,v_2,\ldots, v_n$ with the path
of length $d_\str{A}(v_1,v_n)$ in $\str{G}$. $\widetilde{\str{C}}'$ is a homomorphic image of a cycle of length $k'=k-n+d_\str{A}(v_1,v_n)$ in $\str{G}$. Because $k$ is minimal, we know that $k'$ is even.
It follows that $n+d(v_1,v_n)$ is odd and that vertices $v_1$ and $v_n$ are connected in $\str{B}_1$ both by a path of length $n$ and a path of length $d(v_1,v_n)$.
 Combining these paths together yields
a homomorphic image of an odd cycle in $\str{B}_1$ of length $d(v_1,v_n)+n\leq k$ which is a contradiction
with $\str{B}_1\in \Forb(\str{C}_l)$.
\end{enumerate}
It follows that for every two vertices $v_1,v_n\in \widetilde{C}_k\cap A$ such that their distance within $\widetilde{\str{C}}_k$
is at most $k-1\over 2$ there is a path of the corresponding length in both $\str{B}_1$
and $\str{B}_2$.

Because there
is no copy of $\widetilde{C}_k$ in $\str{B}_1$ or $\str{B}_2$,
we conclude that there is a path  $w_1,w_2,\ldots, w_m$, such that $m>{k-1\over 2}$,
$w_1,w_m\in A$, $w_2,w_3,\ldots, w_{m-1}\notin A$.
Again, without loss of generality we assume that this path is in $\str{B}_1$. Because there
is only one such long path in $\widetilde{C}_k$, we obtain a homomorphic copy of $\widetilde{C}_k$
in $\str{B}_1$ which is a contradiction with $\str{B}_1\in \Forb(\str{C}_l)$.
This finishes the proof that  $\str{C}$ is the strong amalgamation of
$\str{B}_1$ and $\str{B}_2$ over $\str{A}$.

To apply Theorem~\ref{thm:mainstrong}, we observe that $\K_{\str{C}_l}$ is a locally finite
subclass of the class of all ordered structures in the language $L_l$ similarly as for $S$-metric spaces.
\end{proof}
Now we show how the same lift can be shown to be Ramsey using Theorem~\ref{thm:main}.
\begin{proof}[Proof (using Theorem~\ref{thm:main})]
Denote by $\mathcal C_l$ the family containing all possible weak orderings of $\str{C}_l$, the structure containing one vertex with a loop and two structures containing two vertices and a directed edge (in both possible orderings). (The last three structures describe the class of all unoriented graphs.)

It immediately follows that the class $\Age(Forb(\mathcal C_l))$ (of all 
 finite ordered structures in $\Forb(\mathcal C_l)$) is the class of all ordered graphs with no homomorphic image of $\str{C}_l$ and
the existence of a precompact Ramsey lift is given by Theorem~\ref{thm:main}.
However, this result claims more in that it derives a particular lift in the
form of maximal $\F$-lifts as given by Definition~\ref{defn:maximal}.
It remains to check that this homogenisation is equivalent to the one described in the
statement of Theorem~\ref{thm:main}.

The pieces of $\str{C}_l$ (see Definition~\ref{defn:piece}) are all paths of lengths $2, 3, \ldots, l-2$ rooted in
the endpoints. Homomorphism-embedding images of a path of length $k$ rooted in the endpoints are then walks of length at most $k$. The pieces of structures in $\mathcal C_l$ are weakly ordered paths, but because
we consider all possible weak orders, we know that all weakly ordered paths of the same length are $\sim$-equivalent. In the following we can thus
speak only of pieces formed by paths of given length.

Because the construction of a homogenising lift adds relations describing individual pieces and tuples in these relations describe roots of homomorphism-embeddings,
at first glance it seems that the lift constructed is thus more expressive than one we ask for:  It measures the distance of walks of length up to $l-2$
(instead of $l+1\over 2$) and in addition every pair of vertices $(u,v)$ can be in many binary relations.  Here we need to use maximality (as defined in Definition~\ref{defn:maximal}).

We proceed as follows.
Given a pair of vertices $u,v$  of a maximal lift $\str{A}^+$ and its witness $\str{W}$, we verify that the set of
relations (\ie{}\ the set of lengths of permitted walks between $u$ and $v$ in $\str{A}$) is fully determined by the graph distance $l_\str{W}(u,v)$ in $\str{W}$ and that $l_\str{W}(u,v)\leq {(l+1)\over 2}$:
\begin{enumerate}
 \item If the distance $l_\str{W}(u,v)=k$ is even, the existence of walks of all even distances greater than $k$ follows trivially; there is always a homomorphism from the path of length $k+2$ to the path of length $k$ mapping endpoints to endpoints. By maximality there are also all odd walks of distances greater than or equal to $l-k+2$. If such a walk was missing, it would be possible to extend $\str{W}$ by a path of length $l-k+2$ connecting $u$ and $v$ without obtaining a homomorphism-embedding copy of $\str{C}_l$, which would contradict maximality of $\str{A}^+$. We also know that there are no shorter odd walks because every combination of two walks between $u$ and $v$ of length $l$ and $l-k$ produce a homomorphism-embedding copy of $\str{C}_l$.

It follows that (for a given even distance $k$) there is only one possible set of relations between vertices $u$ and $v$ in the maximal lift.
 \item The case of odd distance follows in full analogy.
 \item There are no pairs of vertices of $\str{A}^+$ with distance greater than $(l+1)\over 2$ in $\str{W}$: For any pair of vertices in a greater distance one can add a path of length $(l+1)\over 2$ without introducing a short cycle, again contradicting  maximality of $\str{A}^+$.
\end{enumerate}
\end{proof}
\begin{remark}
While in this simple case both proofs appear similarly complex, in less trivial scenarios it is often a lot easier to analyse the structure of pieces compared to giving an explicit homogenisation and amalgamation procedure. Consider, for example, the class of all graphs having no homomorphic image of the Petersen graph. Pieces of this graph are depicted in Figure~\ref{Petersoni}.
\end{remark}

\subsubsection {Forbidden monomorphisms (Cherlin--Shelah--Shi classes)}
\label{sec:CSS}

The classes defined by forbidden homomorphism-embeddings (\ie{}\ classes $\Forb(\F)$
used in Section~\ref{sec:lifts}) can be seen as a special case of classes defined by
forbidden monomorphisms (or, equivalently, by forbidden non-induced
substructures).  In this section, we treat those monomorphism-defined classes which can be handled by
an application of Theorem~\ref{thm:main}.

Recall that we denoted by $\Forbm(\mathcal M)$ the
class of all finite or countable structures $\str{A}$ such that there is no monomorphism
from any $\str{M}\in \mathcal M$ to $\str{A}$.
The question of the existence of an $\omega$-categorical  universal structure in $\Forbm(\mathcal M)$ was considered by Cherlin, Shelah and
Shi~\cite{Cherlin1999} who gave both a sufficient and a necessary condition  in the form
of local finiteness of the algebraic closure stated below as Theorem~\ref{thm:CSS}.
While the  existence of a universal structure in monomorphism-defined classes
was intensively studied in a series of papers
\cite{Komjath1988,Cherlin1994,Cherlin1996,Cherlin1997,Furedi1997b,Cherlin2007,Cherlin2007a,Cherlin2007b,Cherlin2001,Cherlin2011,Cherlin2015},
it still remains open if the question whether there exists a universal structure in the class of all graphs in $\Forbm(\mathcal M)$ is decidable, even for families $\mathcal M$ consisting of a single finite graph. 
 On the positive side,~\cite{Cherlin1999} proves that for every finite family $\mathcal M$ of finite connected structures which is closed for homomorphic images the class of all graphs in $\Forbm(\mathcal M)$ contains an universal structure (of course in this case $\Forbm(\mathcal M)=\Forbh(\mathcal M)=\Forb(\mathcal M)$). Theorem~\ref{mainthm2}
generalises this result for infinite families.

\begin{figure}
\centering
\includegraphics{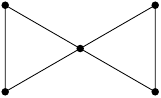}
\caption{The bowtie graph.}
\label{fig:bowtie}
\end{figure}
 It was our analysis of bowtie-free graphs~\cite{Hubivcka2014} (a bowtie is the graph depicted in Figure~\ref{fig:bowtie}) which led to the notion of closure description (Definition~\ref{def:uclosed}). Here we use it to 
 obtain Ramsey lifts of classes defined by forbidden monomorphisms in a greater
generality. This extends the family of known Ramsey classes by non-trivial new examples,
such as forbidden 2-bouquets~\cite{Cherlin2007}, paths~\cite{Komjath1988,Cherlin1999}, complete graphs adjacent to a path~\cite{Komjath1988,Cherlin1999}, bowties adjacent to a path~\cite{Cherlin1999} and in fact all known cases in the work-in-progress catalogue~\cite{Cherlinb}. Some of these classes are really exotic ones. For example, the class of all graph omitting the graph depicted in Figure~\ref{fig:quintafly} contains an $\omega$-categorical universal graph and it is a singular example: It is not possible to change the size of one clique in the picture and again obtain a class containing a universal graph! While in the case of bowtie-free graphs, it is possible to manually analyse the structure of graphs in the class (and this analysis is a core of~\cite{Hubivcka2014}), it is hard to imagine performing such an analysis for the graph in Figure~\ref{fig:quintafly}.
\begin{figure}
\centering
\includegraphics{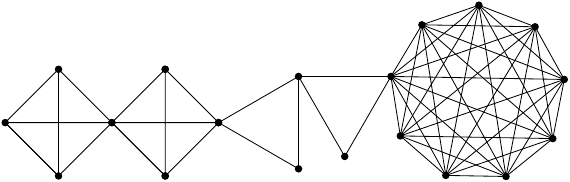}
\caption{An example of a forbidden subgraph (by monomorphism).}
\label{fig:quintafly}
\end{figure}

It is only fitting that we end this paper by combining the Ramsey methods developed here with perhaps the most successful result about the existence of $\omega$-categorical universal objects provided by~\cite{Cherlin1999}. In fact, all these pieces fit together very well.
First we briefly review the terminology of~\cite{Cherlin1999}.

\begin{definition}
\label{def:algebraic}
Let $L$ be a relational language, let $\str{A}$ be an $L$-structure and let $S$ be a finite subset of $A$.  The
\emph{algebraic closure of $S$ in $\str{A}$}, denoted by $\ACl_\str{A}(S)$, is the
set all vertices $v\in A$ for which there is a formula $\phi$ in the language $L$ with $\vert S\vert +1$
variables such that $\phi(\vv{S},v)$ is true and there are only finitely many
vertices $v'\in A$ such that $\phi(\vv{S},v')$ is also true. (Here $\vv{S}$
is an arbitrary ordering of the vertices of $S$.)
\end{definition}
In the following we will use functions to explicitly represent algebraic closures, thereby
obtaining the strong amalgamation property.

We say that a structure $\str{A}$ is \emph{algebraically closed} in a structure $\str{U}$
if for every embedding $e\colon\str{A}\to \str{U}$ it holds that $\ACl_\str{U}(e(A))=e(A)$ (where $e(A)=\{e(x);X\in A\}$).
The algebraic closure in $\str{U}$ is \emph{locally finite} if there exists a
function $f\colon\mathbb N\to \mathbb N$ such that $\vert \ACl_\str{U}(S)\vert \leq
f(\vert S\vert )$ for every finite $S\subseteq U$.

\begin{theorem}[Cherlin, Shelah, Shi~\cite{Cherlin1999}]
\label{thm:CSS}
Let $\mathcal M$ be a finite family of finite connected relational structures.  There
is an $\omega$-categorical universal structure   in $\Forbm(\mathcal M)$
if and only if the algebraic closure  in existentially complete structures in
 $\Forbm(\mathcal M)$ is locally finite. 
\end{theorem}

We make use of the following consequence of Prak's theorem~\cite{Prak,Cherlin1999} which we include without proof.
\begin{lemma}[Lemma 5 of \cite{Cherlin1999}]
\label{lem:lemma5}
Let $L$ be a relational language,
let $\mathcal M$ be a finite collection of finite $L$-structures, let $\str{U}$ be
an $\omega$-categorical existentially complete universal structure in $\Forbm(\mathcal M)$ and
let $\str{A}$ be a finite substructure of $\str{U}$. Then the following are equivalent:
\begin{enumerate}
\item $\str{A}$ is not algebraically closed in $\str{U}$.
\item There is $\str{M}\in \mathcal M$ and a substructure $\str{M}'$ of $\str{U}$ containing $\str{A}$ such that the following hold:
\begin{enumerate}
\item $\str M$ has a homomorphism to $\str M'$, and
\item there is a structure $\str S$ which can be obtained by a series of free amalgamations of $\vert M\vert$ copies of $\str M'$ over $\str A$ such that $\str M$ has a monomorphism to $\str S$.
\end{enumerate}
\end{enumerate}
\end{lemma}
\begin{corollary}
\label{cor:algc2}
Let $\mathcal M$ be a class of finite connected relational structures such that $\Forbm(\mathcal M)$
contains an $\omega$-categorical existentially complete universal structure $\str{U}$.
Let $\str{A}$ be finite algebraically closed substructure of $\str{U}$ and let $\str{B}$ be a substructure of $\str{U}$ containing $\str{A}$.
Then $\str{U}$ contains a strong amalgamation of $\str{B}$ and $\str{B}$ over $\str{A}$
\end{corollary}
\begin{proof}
Assume, to the contrary, that $\str{U}$ does not contain a strong amalgamation of $\str{B}$ and $\str{B}$ over $\str{A}$.
Because $\str{U}$ is existentially complete it follows that there is a monomorphism $f$ from $\str{M}\in \mathcal M$
to the free amalgamation $\str{U}'$ of $\str{U}$ and $\str{B}$ over $\str{A}$. Without loss of generality (by sufficiently extending $\str{B}$) we can assume
that there is a monomorphism $f'$ from $\str{M}$ to the free amalgamation $\str{C}$ of $\str{B}$ and $\str{B}$ over $\str{A}$.
Using Lemma~\ref{lem:lemma5} this contradicts $\str{A}$ being algebraically closed.
\end{proof}
\begin{theorem}
\label{thm:CSSramsey}
Let $\mathcal M$ be a set of finite connected structures in a relational language $L$ such that $\Forbm(\mathcal M)$
contains an existentially complete $\omega$-categorical universal structure $\str{U}$.
Further assume that for every $\str{M}\in \mathcal M$ at least one of the following conditions holds:
\begin{enumerate}
  \item There is no homomorphism-embedding from $\str{M}$ to $\str{U}$, or
  \item $\str{M}$ can be constructed from irreducible structures by a series of free amalgamations over irreducible substructures.
\end{enumerate}
Then the class of all finite algebraically closed substructures of $\str{U}$ has a precompact Ramsey lift.
(By the standard homogenisation argument it also follows that $\Age(\Forbm(\mathcal M))$ has a precompact Ramsey lift.)
\end{theorem}
\begin{proof}
First we expand $L$ by an order. Let $\vv{\str{U}}$ be
the lift of $\str U$ adding a generic linear order.
We further extend language $L$ by necessary functions to represent the
algebraic closure of every finite ordered irreducible substructure of $\vv{\str{U}}$
and by relational symbols denoting each orbit of every ordered irreducible substructure of $\vv{\str{U}}$.
This can be done in an automorphism preserving way by adding only finitely many function symbols of every arity:
Observe that because $\vv{\str{U}}$ is ordered and the algebraic closure in $\str{U}$ is locally finite, for every
$a\geq 1$ there is $f(a)$ determining the largest size of the algebraic closure of a substructure of $\vv{\str{U}}$
with at most $a$ vertices. We thus introduce function symbols of arity $a$ denoted by $\func{}{a,i}$ for every $1\leq i\leq f(a)$.

 Denote by $L^+$ the resulting language and denote by $\vv{\str{U}}^+$ the
$L^+$-lift of $\vv{\str{U}}$ adding the newly introduced functions representing the closures and relations representing the orbits of order-irreducible substructures.
The first is done by putting $\nbfunc{\vv{\str{U}}^+}{a,i}(v_1,v_2,\ldots, v_a)=v$ if and only if $a\geq 1$, $1\leq i\leq f(a)$, all vertices in $(v_1,v_2,\ldots, v_a)$ are distinct and $v$ is the $i$-th vertex of the algebraic closure of $\{v_1,v_2,\ldots, v_a\}$ in the linear order of $\vv{\str{U}}$.
The second is done the same way as in the standard homogenization but only for finite irreducible substructures of $\str{U}$.

Let $\mathcal F^1_\mathcal M$ be the class of all ordered irreducible structures $\str{A}\notin \Age(\vv{\str{U}}^+)$.
Denote by $n$ the size of the largest structure in $\mathcal M$ and by $N=f(n)$ the bound on the size of the algebraic closure of a structure on at most $n$ vertices. 
Let $\F^2_\mathcal M$ denote the class of all weakly ordered structures with at most $N^{2^n}$ vertices
which have no homomorphism-embedding to $\vv{\str{U}}^+$.
Note that $N^{2^n}$ is an easy upper bound on the size of the closure of an $N$-vertex set in $\Forb(\F^1_\mathcal M)$.
The size of irreducible structures that are closures of substructures with at most $n$ vertices is at most $N$ and the closure of a reducible structure
is a result of the corresponding free amalgamation.

Observe that $\mathcal F_\mathcal M=F^1_\mathcal M\cup F^2_\mathcal M$ is a regular family of structures because $F^1_\mathcal M$ consists of irreducible structures (and
thus yields no pieces) and $F^2_\mathcal M$  is finite.
Now apply Theorem~\ref{thm:main} for $\mathcal F_\mathcal M$
to obtain a precompact Ramsey lift $\mathcal K^+_\mathcal M$ of the class $\mathcal K_\mathcal M$ of all finite ordered structures in $\Forb(\mathcal F_\mathcal M)$.
We claim that  the class  $\mathcal K^+_\mathcal M$ is a precompact Ramsey lift of $\Age(\Forbm(\mathcal M))$.

Fix $\str{C}\in \K^+_\mathcal M$ and assume, to the contrary, the existence of $\str{M}\in \mathcal M$ such that there is a monomorphism $m$ from $\str{M}$ to the shadow $\sh(\str{C})$.
Because $\str{C}\in \Forb(\mathcal F_\mathcal M)$, there is a homomorphic image of $\str{M}$  in $\Forbm(\mathcal M)$. From our assumptions it follows that $\str{M}$ can be constructed from irreducible structures by a series of free amalgamations over irreducible substructures.  Denote by $\str{M}_1, \str{M}_2, \ldots, \str{M}_n$ the irreducible structures used to build $\str{M}$ and denote by $\str{M}'_1, \str{M}'_2, \ldots, \str{M}'_n$
the closures of $m(\str{M}_1), m(\str{M}_2),\allowbreak \ldots,\allowbreak m(\str{M}_n)$ in $\str{C}$. Observe that all those structures are also irreducible.

Next construct $\str{M}'$ and a homomorphism $f\colon \str{M}'\to \vv{\str{U}}^+$ by following the same amalgamations which are used to construct $\str{M}$ but with the structures $\str{M}'_1, \str{M}'_2, \ldots, \str{M}'_n$ over the closures of the corresponding amalgamation bases.
Denote by $\str{M}''_1, \str{M}''_2, \ldots, \str{M}''_n$ the corresponding copies of $\str{M}'_1, \str{M}'_2, \ldots, \str{M}'_n$ in $\str{M}'$.
The homomorphism can be constructed by following the amalgamation procedure because of the following properties of our construction:
\begin{enumerate}
\item $\str{M}'$ is a result of series of amalgamations of finite ordered irreducible structures $\str{M}''_1,\allowbreak  \str{M}''_2,\allowbreak  \ldots,\allowbreak  \str{M}''_n$ over their irreducible ordered substructures.
\item $\mathcal F$ contains all finite irreducible ordered structures which are not substructures of $\vv{\str{U}}^+$ and thus each of $\str{M}''_1, \str{M}''_2, \ldots, \str{M}''_n$ embeds to $\vv{\str{U}}^+$.
\item For every pair $\str{M}''_i$ and $\str{M}''_j$ we know that whenever $M''_i\cap M''_j\neq \emptyset$, the substructure $\str{N}$ induced by $\str{M}'$ on  $M''_i\cap M''_j$ is irreducible.
Thus $\str{N}$ corresponds to a unique orbit of the automorphism group of $\vv{\str{U}}^+$ (because we extended $\vv{\str{U}}$ in an automorphism-preserving way by explicitly denoting every orbit of such substructures).
Consequently, every embedding of $\str{M}''_i$ to $\vv{\str{U}}^+$ extends to the amalgamation of $\str{M}''_i$ and $\str{M}''_j$ over $\str{N}$.
\end{enumerate}
Observe that there is also a monomorphism from $\str{M}$ to the shadow of $\str{M}'$.

To arrive to a contradiction, choose a homomorphism-embedding $h'\colon\str{M}'\to \str{U}^+$ which maximises the number of vertices of $h'(M')$.
Because $h'$ is not a monomorphism it follows that there are distinct vertices $u_1,u_1\in M'$ such that $h'(u_1)=h'(u_2)$.
Because $\str{M}'$ has a tree-like structure, there is a unique path $\str{M}''_{i_1},\str{M}''_{i_2},\ldots,\str{M}''_{i_\ell}$ such that $u_1\in M''_{i_1}$, $u_2\in M''_{i_\ell}$ and $M''_{i_j}\cap M''_{i_{j+1}}\neq \emptyset$ for every $1\leq j<\ell$.
We can further assume that $u_1$ and $u_2$ were chosen so that $\ell$ is minimal.

Because $h'$ is a homomorphism-embedding and thus it is not possible that both $u_1,u_2\in \str{M}''_{i_1}$, we get that $\ell>1$.  We seek for $j$ such that the
intersection of $\str{M}'_{i_j}$ and $\str{M}'_{i_{j+1}}$ forms a cut $R$ of
$\str{M}'$ which does not contain $u_1$ and $u_2$. Such a cut exists because of the following observations:
\begin{enumerate}
\item $\str{M}'$ induces an irreducible substructure on every such cut,
and because $h'$ is a homomorphism-embedding, no such cut can contain both $u_1$ and $u_2$. 
\item If $\str{M}'_{i_j}$ contains $u_1$ and $\str{M}'_{i_j+1}$ contains $u_2$ then their intersection contains neither $u_1$ and $u_2$.  
\end{enumerate}

Consequently there is a cut $R$ of $\str{M}$ separating $u_1$ and $u_2$. Because $\ell$ is minimised
we know that $h'(R)$ contains neither $h'(u_1)$ nor $h'(u_2)$.
By Corollary~\ref{cor:algc2} we know that $\str{U}^+$ contains a strong amalgamation $\str{D}$ of $h'(\str{M}')$ and $h'(\str{M}')$ over $h'(R)$. This is a contradiction with the fact that $h$ maximises $|h'(M)|$: we can construct a homomorphism-embedding $h''$ to $\str{D}$ such that $u_1$ maps to the first copy of $\str{M}'$ and $u_2$ to the second.
\end{proof}

\begin{remark}
The order needs to be handled carefully in the proof. It may seem more natural
to first homogenise $\str{U}$ and then add the order. This however often leads
to a more complex structure. If the order is introduced first and
assuming that the language
$L$ contains no relations of arity greater than $k$, the isomorphism type of
closure of a substructure can be uniquely determined by the isomorphism type of
its substructures of size at most $k$. An example of such a class is given in~\cite{Hubivcka2014}.

In addition, lifting by the free order will not give a lift with the lift property for classes
with non-trivial closure.  Such lifts needs more
detailed analysis of the structure of this closure. The special case of the bowtie-free
graphs is analysed in~\cite{Hubivcka2014}.
\end{remark}

\begin{remark}
It is conjectured in~\cite{Cherlin2015,Cherlin2011} that every graph $\str{G}$
such that there exists an $\omega$-categorical universal graph for the class of all graphs in $\Forbm(\{\str{G}\})$ has all 2-connected components irreducible.
If this conjecture is true, Theorem~\ref{thm:CSSramsey} shows the existence
of a precompact Ramsey lift for every class of graphs $\Forbm(\{\str{G}\})$ with an
$\omega$-categorical universal graph. So it seems this is as far as we can go: The existence of a Ramsey lift is equivalent to $\omega$-categoricity of the universal graph under the conjecture.
\end{remark}

\begin{remark}
If $\mathcal M$ consists only of structures constructed from
irreducible structures by a series of free amalgamations over irreducible
substructures the existence of $\omega$-categorical universal structure is
actually necessary in Theorem~\ref{thm:CSSramsey} only to establish the
precompactness of the lift.  Even in the cases where the algebraic closure is not
locally finite, the same technique as above can be used for the class of
homogenising lifts of the structures (which is not precompact and the resulting
\Fraisse{} limit will not be universal, only universal for finite structures of
the age). The resulting Ramsey lift will be a precompact lift of this
homogenising lift.

On the other hand, the class $\mathcal C_{\mathrm{girth}\geq 5}$ given in Example~\ref{example:C4}
has a binary closure:
For every pair of vertices there is at most one vertex connected to
both of them. It is easy to consider a lift adding a partial binary function
$\func{}{C}$  which maps every pair of vertices to the unique vertex connected to both of them if such a vertex exists.
(Note that the closure is not locally finite~\cite{Fueredi1997} and there is no
$\omega$-categorical universal graph for $\Forbm(\str{C}_4)$).  This class has
strong amalgamation (over closed structures), however the existence of a precompact Ramsey lift
is open.
\end{remark}
\begin{remark}
The conditions of Theorem~\ref{thm:CSSramsey} given on the family $\mathcal M$ can be
generalised. Cherlin~\cite{Cherlin2011} gave an example of a class $\Forbm(\mathcal M)$
with non-unary algebraic closure. It is easy to show that the techniques
used in the proof of Theorem~\ref{thm:CSSramsey} apply for this class, too.
\end{remark}

\section{Acknowledgements}
We would like to thank to Andr\'es Aranda, Manuel Bodirsky, Peter Cameron, Gregory Cherlin, David M.~Evans, and
Dugald Macpherson for discussions and remarks that improved quality of this
paper, and particularly to Mat\v ej Kone\v cn\'y.

A large part of work was done while the first author had PIMS Postdoctoral
Fellow position at University of Calgary under lead of Claude Laflamme,
Norbert Sauer and Robert Woodrow. Discussions in Calgary were essential to gain
understanding of the model-theoretic aspects.  We are also
grateful to the anonymous referee, David Bradley-Williams and Miodrag Soki{\'c}.

\bibliographystyle{plain}

\bibliography{ramsey.bib}
\end{document}